\documentclass[11pt]{article}

\usepackage[a4paper,top=2cm,bottom=2.5cm,left=2cm,right=2cm]{geometry}

\usepackage[utf8]{inputenc}
\usepackage[OT4]{fontenc}

\usepackage{amsmath,amsthm,amssymb}
\usepackage{bbm}
\usepackage{mathrsfs}
\usepackage{graphicx}
\usepackage{xcolor}

\usepackage[numbers]{natbib}
\usepackage{array}
\usepackage[retainorgcmds]{IEEEtrantools}

\usepackage{filemod}
\usepackage[breaklinks=true]{hyperref}

\theoremstyle{plain}
\newtheorem{theorem}{Theorem}[section]
\newtheorem{proposition}[theorem]{Proposition}
\newtheorem{lemma}[theorem]{Lemma}
\newtheorem{corollary}[theorem]{Corollary}

\theoremstyle{definition}
\newtheorem{definition}[theorem]{Definition}

\newtheorem{remark}[theorem]{Remark}

\makeatletter

\def\be#1{\begin{equation*}#1\end{equation*}}
\def\ben#1{\begin{equation}#1\end{equation}}

\def\besn#1{\begin{equation}\begin{split}#1\end{split}\end{equation}}

\def\given{\typeout{Command 'given' should only be used within bracket command}}
\newcounter{@bracketlevel}
\def\@bracketfactory#1#2#3#4#5#6{
\expandafter\def\csname#1\endcsname##1{%
\addtocounter{@bracketlevel}{1}%
\global\expandafter\let\csname @middummy\alph{@bracketlevel}\endcsname\given%
\global\def\given{\mskip#5\csname#4\endcsname\vert\mskip#6}\csname#4l\endcsname#2##1\csname#4r\endcsname#3%
\global\expandafter\let\expandafter\given\csname @middummy\alph{@bracketlevel}\endcsname
\addtocounter{@bracketlevel}{-1}}%
}
\def\bracketfactory#1#2#3{%
\@bracketfactory{#1}{#2}{#3}{relax}{1mu plus 0.25mu minus 0.25mu}{0.6mu plus 0.15mu minus 0.15mu}
\@bracketfactory{b#1}{#2}{#3}{big}{1mu plus 0.25mu minus 0.25mu}{0.6mu plus 0.15mu minus 0.15mu}
\@bracketfactory{bb#1}{#2}{#3}{Big}{2.4mu plus 0.8mu minus 0.8mu}{1.8mu plus 0.6mu minus 0.6mu}
\@bracketfactory{bbb#1}{#2}{#3}{bigg}{3.2mu plus 1mu minus 1mu}{2.4mu plus 0.75mu minus 0.75mu}
\@bracketfactory{bbbb#1}{#2}{#3}{Bigg}{4mu plus 1mu minus 1mu}{3mu plus 0.75mu minus 0.75mu}
}
\bracketfactory{clc}{\lbrace}{\rbrace}
\bracketfactory{clr}{(}{)}
\bracketfactory{cls}{[}{]}
\bracketfactory{abs}{\lvert}{\rvert}
\bracketfactory{norm}{\Vert}{\Vert}
\bracketfactory{floor}{\lfloor}{\rfloor}
\bracketfactory{ceil}{\lceil}{\rceil}
\bracketfactory{angle}{\langle}{\rangle}

\newcounter{ctr}\loop\stepcounter{ctr}\edef\X{\@Alph\c@ctr}%
	\expandafter\edef\csname s\X\endcsname{\noexpand\mathscr{\X}}
	\expandafter\edef\csname c\X\endcsname{\noexpand\mathcal{\X}}
	\expandafter\edef\csname b\X\endcsname{\noexpand\boldsymbol{\X}}
	\expandafter\edef\csname I\X\endcsname{\noexpand\mathbbm{\X}}
	\expandafter\edef\csname r\X\endcsname{\noexpand\mathrm{\X}}
        \expandafter\edef\csname til\X\endcsname{\noexpand\widetilde{\X}}
\ifnum\thectr<26\repeat

\newcount\minute
\newcount\hour
\newcount\hourMins
\def\now{%
\minute=\time%
\hour=\time \divide \hour by 60%
\hourMins=\hour \multiply\hourMins by 60%
\advance\minute by -\hourMins%
\zeroPadTwo{\the\hour}:\zeroPadTwo{\the\minute}%
}
\def\zeroPadTwo#1{\ifnum #1<10 0\fi#1}

\numberwithin{equation}{section}
\allowdisplaybreaks[4]

\def\blfootnote{\xdef\@thefnmark{}\@footnotetext}

\makeatother

\def\^#1{\ifmmode {\mathaccent"705E #1} \else {\accent94 #1} \fi}
\def\~#1{\ifmmode {\mathaccent"707E #1} \else {\accent"7E #1} \fi}

\edef\-#1{\noexpand\ifmmode {\noexpand\bar{#1}} \noexpand\else \-#1\noexpand\fi}
\def\>#1{\vec{#1}}

\def\wt#1{\widetilde{#1}}
\def\atop{\@@atop}

\renewcommand{\leq}{\leqslant}
\renewcommand{\geq}{\geqslant}
\renewcommand{\phi}{\varphi}
\newcommand{\eps}{\varepsilon}

\newcommand{\eq}{\eqref}

\newcommand{\lito}{\mathrm{o}}

\newcommand{\Var}{\mathop{\mathrm{Var}}\nolimits}

\newcommand{\Cov}{\mathop{\mathrm{Cov}}}

\newcommand{\bfa}{\textbf{\textit{a}}}
\newcommand{\bfb}{\textbf{\textit{b}}}
\newcommand{\bfc}{\textbf{\textit{c}}}
\newcommand{\bfm}{\textbf{\textit{m}}}
\newcommand{\bfr}{\textbf{\textit{r}}}

\newcommand{\beats}{\succ}

\usepackage{bm}

\newcommand{\bfx}{\textbf{x}}

\newcommand{\Maj}{\mathrm{Maj}}
\DeclareMathOperator{\sgn}{sgn}
\DeclareMathOperator*{\EE}{\mathbb{E}}

\DeclareMathOperator{\diag}{diag}
\DeclareMathOperator{\Bin}{Bin}
\DeclareMathOperator{\erf}{erf}
\DeclareMathOperator{\supp}{supp}
\newcommand{\1}{\mathbbm{I}}

\definecolor{DSgray}{cmyk}{0,0,0,0.7}
\definecolor{DSred}{cmyk}{0,0.7,0,0.7}

\usepackage{tikz}
\usetikzlibrary{calc, arrows, decorations.markings, positioning, fit,
  shapes.geometric}

\renewcommand{\Pr}{\IP}

\def\R{\mathbb{R}}
 \def\E{\mathbb{E}}
\def\lmto{\longmapsto}

\begin{document}

\title{The Probability of Intransitivity in Dice and Close Elections}

\author{
Jan Hązła\thanks{EPFL, {\tt jan.hazla@epfl.ch}.}\and
Elchanan Mossel\thanks{Massachusetts Institute of Technology, {\tt elmos@mit.edu}.
  E.M.~and J.H.~were partially supported by awards ONR N00014-16-1-2227,  
  NSF CCF-1665252 and DMS-1737944. E.M.~was partially supported by
the Simons Investigator award (622132).}\and
Nathan Ross\thanks{University of Melbourne, {\tt nathan.ross@unimelb.edu.au}.
  Partially supported by ARC DP150101459.}\and
\and
Guangqu Zheng\thanks{University of Kansas, {\tt zhengguangqu@gmail.com}.
  Partially supported by ARC DP150101459.}
}
\date{}
\maketitle

\begin{abstract}
  We study the phenomenon of intransitivity in models of dice
  and voting.

  First, we follow a recent thread of research for $n$-sided dice
  with pairwise ordering induced by the probability, relative to $1/2$,
  that a throw from one die is higher than the other. We build on a
  recent result of Polymath showing that three
  dice with i.i.d.~faces drawn from the uniform distribution on $\{1,\ldots,n\}$
  and conditioned on the average of faces equal to $(n+1)/2$
  are intransitive with asymptotic probability $1/4$. 
  We show that if dice faces are drawn from a non-uniform
  continuous mean zero distribution conditioned on the average of faces
  equal to $0$, then three dice are transitive with high probability.
  We also extend our results to  stationary
  Gaussian dice, whose faces, for example, can be  the fractional Brownian increments with   Hurst index $H\in(0,1)$.

  Second, we pose an analogous model in the context of Condorcet voting. 
  We consider $n$ voters who rank $k$ alternatives independently and uniformly
  at random. 
  The winner between each two alternatives is decided by a majority vote based 
  on the preferences. We show that in this model, if all pairwise elections 
 are close to tied, then the asymptotic probability of obtaining any 
 tournament on the $k$ alternatives is equal to 
 $2^{-k(k-1)/2}$, which markedly differs from known results in the
 model without conditioning.
 We also explore the Condorcet voting model where methods other than 
 simple majority are used for pairwise elections. We investigate some natural 
 definitions of ``close to tied'' for general functions and exhibit an example 
 where the distribution over tournaments is not uniform under those definitions. 
\end{abstract}

\section{Introduction}

The phenomenon of intransitivity often arises when one ranks three or more alternatives. 
An early example is the Condorcet paradox, discovered in the 18th century in the context of voting. This type of intransitivity is much more general, as proved by Arrow in his social choice theorem~\cite{Arr50}.  
A different fascinating aspect of intransitivity arises in the context of games of chance: 
The striking phenomenon of non-transitive dice. It was discovered by the statistician Brad Efron~\cite{Gardner70} and has fans such as Warren Buffet (who reportedly tried to trick Bill Gates~\cite{Lowe01}). 
The main motivating question of this paper is: What is the chance of observing intransitivity in natural random setups?
We present some quantitative answers to this question. 
We introduce and discuss our results for dice and voting separately, making comparisons between the two settings where appropriate.

\subsection{Intransitive dice: Transitivity of non-uniform dice}

For the purposes of this paper, we call an \emph{$n$-sided die}
(think of gambling dice) any vector
{\it$\bfa = (a_1, \ldots, a_n)$} of real numbers. The \emph{face-sum} of a die
{\it$\bfa$} is $\sum_{i=1}^n a_i$.
We say that  die
{\it$\bfa$} beats die {\it$\bfb$}, denoted {\it$\bfa \beats \bfb$}, if 
a uniformly random face of {\it$\bfa$} has greater value than a random face of {\it$\bfb$}.
In other words, {\it$\bfa \beats \bfb$} if
$$\left(\sum_{i,j=1}^n \1[a_i > b_j] - \1[a_i < b_j]\right) > 0.$$
We call a finite set of $n$-sided dice
\emph{intransitive} if the ``beats'' relation on the set cannot be extended
to a linear order. That is, a set of dice is intransitive if it contains a subset
{\it$\bfa^{(1)}, \ldots, \bfa^{(k)}$} such that
{\it$\bfa^{(1)} \beats \bfa^{(2)} \beats \ldots \beats \bfa^{(k)} \beats \bfa^{(1)}$}.
A well-known example with three sides is {\it$\bfa = (2, 4, 9)$, $\bfb = (1, 6, 8)$}
and {\it$\bfc = (3, 5, 7)$}. One checks that
{\it$\bfa \beats \bfb \beats \bfc \beats \bfa$.}
If a set of dice forms a linear ordering, then we call it \emph{transitive}.
Because of ties, there can be sets that are neither transitive nor intransitive,
but they occur with negligible probability in the models we study.

Recently, there has been some interest in the quantitative study of intransitive
dice. The main quantity of interest is the probability that three independent dice are
transitive, under different random models.
In particular, as the number of faces grows,
the dice can behave \emph{transitively}, i.e.,
such that a  triple of random dice is transitive with
high probability. At the other end of the spectrum, there can be behavior that we call,
borrowing the term from Kalai's paper on social choice
\cite{Kal10}, \emph{chaotic}: in that regime, three dice are intransitive
with probability\footnote{By considering paths of length two in the
  tournament graph on dice according to the ``beats''
  relation, one can see that $1/4$ is the highest possible probability
  of intransitivity (see~\cite{PolVII}).
  }
approaching $1/4$.

Some (mostly) experimental results
were presented by Conrey, Gabbard, Grant, Liu and Morrison
\citep{CGGLM}.
Among others, they conjectured that the model where
$n$-sided dice are sampled uniformly from multisets of integers between
$1$ and $n$ conditioned on the face-sum equal
to $n(n+1)/2$ is chaotic. A recent collaborative Polymath project
\citep{Pol17} proved this conjecture for a related, but not identical, model
where a die is a random sequence of integers between $1$ and $n$ conditioned on
the face-sum equal to $n(n+1)/2$.

One may wonder what happens without the face-sum conditioning.
In that case it can be seen in
\citep{PolVII} that if the faces are only i.i.d. (with distribution depending on $n$),
then as soon as the face-sums of dice {\it$\bfa$} and {\it$\bfb$}
differ by significantly more than $n \log n$,
the die with the higher face-sum beats the other one with high probability.
In particular, three random dice with uniform faces from $\{1,\ldots,n\}$ without
conditioning are transitive with high probability.

One might just as well study dice with faces drawn from a continuous probability
distribution. In particular, experiments and intuition strongly suggest
that the model where faces are uniform in $(-1, 1)$ and conditioned
on face-sum equal zero is, as in the discrete case,
chaotic.

Our first result indicates that this behavior is quite fragile. If the uniform
faces are replaced with any other continuous distribution (satisfying
some reasonable assumptions), then whether a die beats another is
determined by the value of a real function of the faces of each die
and the model becomes transitive.
\begin{theorem}
  \label{con:weak}
  Take $\bfa$, $\bfb$ and $\bfc$ to be three independent $n$-sided dice
  with i.i.d. faces.
  Assume that the distribution of a single face has density  {\rm(PDF)}
  $f$ and {\rm CDF} $F$, mean zero and variance one.
  Let $\cE_0$ denote the event that the face-sums of $\bfa$, $\bfb$ and $\bfc$
  are all zero.
  Additionally, assume that the  distribution of a single face:
  \begin{itemize}
  \item Has enough {\rm (say, six)} finite moments.
  \item Has PDF $f$ supported on a {\rm (possibly infinite)} closed interval
    $\supp(f)$. Furthermore, $f$ is continuous on $\supp(f)$.
  \item Is  \emph{not} uniform on $\big[-\sqrt{3}, \sqrt{3} \,\big]$.
  \end{itemize}
  Then:
  \begin{enumerate}
  \item
    Conditional on $\cE_0$, with probability tending to one as
    $n \to \infty$,
    \begin{align*}
      \text{$\bfa$ beats $\bfb$ if and only if }
      \sum_{i=1}^n F(a_i) > \sum_{i=1}^n F(b_i) \; .
    \end{align*}
  \item
    As $n \to \infty$,
    $\Pr \left[ \bfa, \bfb, \bfc \text{ are transitive} \mid \cE_0 \right]
    \to 1$.
  \end{enumerate}
\end{theorem}



To understand the differing behavior of  uniform versus non-uniform
dice implied by Theorem~\ref{con:weak} and the Polymath result,
we first recall that, as shown by Polymath \cite{PolVII},
for unconditioned dice with faces uniform in $(0, 1)$,
the face-sums determine if {\it$\bfa$} beats {\it$\bfb$} with high probability.
Taking an arbitrary single-face distribution $F$,
without conditioning on face-sums the distribution of the
random variable $W = \sum_{i,j=1}^n \1[a_i > b_j]$ does not depend
on $F$: this is because
$a_i > b_j$ if and only if $F(a_i)>F(b_j)$, and since
$(F(a_1),\ldots,F(a_n))$ is a die with faces uniform in $(0,1)$;
see also our Theorem \ref{thm41}.
Therefore, considering distribution $F$ conditioned on
$\cE_0$,
for the purposes of the ``beats'' relation, one can just as well think
of a die $(F(a_1),\ldots,F(a_n))$ conditioned on $\sum_{i=1}^n a_i=0$.
As long as $F$ is not affine, one might expect that,
even under~$\cE_0$, the random variables
$F (a_i )$ are distributed (almost) uniformly in $(0,1)$ with only
weak, global dependencies, suggesting that the expression
\be{
  \sgn\left( \sum_{i=1}^n F\left(a_i\right)-F\left(b_i\right)\right)
}
still determines the winner with high probability. Note that this
heuristic fails for the uniform distribution
since in that case the CDF-sum is a determinstic function of
the face-sum.

Applying the same reasoning in reverse, our result can be interpreted as
showing that
\begin{align*}
  \lim_{n\to\infty}
  \Pr\left[\bfa,\bfb,\bfc\text{ are intransitive}
  \mid \sum_{i=1}^n G(a_i) = \sum_{i=1}^n G(b_i) = \sum_{i=1}^n G(c_i) = 0\right]
  =0
\end{align*}
for \emph{uniform} dice $\bfa$, $\bfb$, $\bfc$ for a large class of continuous,
increasing, non-affine functions $G:\mathbb{R}\to\mathbb{R}$. This suggests that the intransitivity phenomenon for uniform dice is strongly linked to conditioning on the slices
$\sum_{i=1}^n a_i = c$.

Note that the assumptions of Theorem~\ref{con:weak} imply that the PDF
$f$ is bounded. We believe that they can be weakened in that respect:
For example, it should be enough that the convolution $f^{(*k)}$
is bounded for some finite $k$ (with the support interval $\supp(f)$
not necessarily closed) and that the assumption of continuity of $f$ is replaced with
piecewise continuity. We do not treat those relaxed assumptions
for the sake of readability. In any case, based, among others,
on experiments involving Cauchy distribution,
we suspect that the first two itemized assumptions in
Theorem~\ref{con:weak} are not necessary for its statement to hold.

The main ingredient of the proof is a variance calculation that establishes
that for two dice
\begin{align*}
  \Var\left[ \sum_{i,j=1}^n \1(a_i>b_j) - n\sum_{i=1}^n \big(F(a_i)-F(b_i)\big)
  \mid \cE_0\right] = o(n^3) \; ,
\end{align*}
while the variance of each term of the difference is of order~$n^3$.
These two facts and an anti-concentration argument then 
 imply Theorem~\ref{con:weak}. The variance calculation
uses a CLT calculation with a rather attentive tracking of errors.
This is interesting in comparison with \cite{Pol17}, since it suggests
that careful application of central limit theorems is important in establishing
both transitivity and intransitivity results.
We also need to establish CLT-like anti-concentration for the random variable
$\sum_{i=1}^n F(a_i)$ conditioned on $\cE_0$. For that, we employ a direct
argument that uses conditioning on the values of pairs
$a_1+a_2, \ldots, a_{n-1}+a_{n}$. The proof is given in
Section~\ref{sec:nonuniform}.

\subsection{Intransitive dice: Stationary Gaussian dice}

In the  setting of Theorem \ref{con:weak} with standard Gaussian  $\mathcal{N}(0, 1)$ faces,
it can be computed that the conditioned die  {\it $\bfa=( a_1, \ldots, a_n )$}
is distributed as a joint centered Gaussian with $\Var[a_i] = 1-1/n$ and
$\Cov[a_i, a_j] = -1/n$ for $i\ne j$. Therefore, it can be seen
as a \emph{locally stationary} Gaussian family, that is, a family where the correlation of $a_i$ and $a_j$ depends only on $n$ and $i-j$ (more precisely, for our conditioning, the correlation depends solely on
whether $i$ is equal to $j$, i.e., $\delta_{ij}$).

In this particular Gaussian case, one can provide another proof of the conclusion of Theorem~\ref{con:weak}  using the so-called Malliavin-Stein machinery (see \cite{bluebook} for a comprehensive treatment). Indeed, one can expand the indicator function $\1[\bullet > 0]$ based on Hermite polynomials (see \eqref{ind_chaos}), then rewrite the random variable $W = \sum_{i,j=1}^n \1[a_i -b_j > 0]$ into an infinite sum of multiple Wiener-It\^o integrals. It is then enough to apply (for example) Theorem 6.3.1 in \cite{bluebook} to get the following CLT:
\[
\frac{1}{n^{3/2}} \Big[ W - \mathbb{E}\big(W \big) \Big] \xrightarrow[n\to+\infty]{\rm law} \mathcal{N}(0, \alpha) \,,
\] 
where the limiting variance $\alpha = \frac{1}{6} - \frac{1}{2\pi}$ can be deduced from standard arguments and Newton's 1676 identity (see Remark \ref{obs11}). On the other hand, one can again use the Hermite expansion to compute that variance of $W - n \sum_{i=1}^n [ F(a_i) - F(b_i) ] $ is $O(n^2)$. Then the transitivity follows from this variance estimate and the above CLT. We leave the details for interested readers.      Meanwhile, it is natural to investigate the (globally) stationary Gaussian case.   It turns out that one can use the Breuer-Major theorem  \cite{BM83} to prove a version of Theorem~\ref{con:weak} for (globally) \emph{stationary Gaussian dice}.

Here is our setting: let $\{G_i, i\in\mathbb{N}\}$ be a centered stationary Gaussian sequence such that $\E[ G_i G_j ] =\rho(i-j)$ for some (correlation) function $\rho: \mathbb{Z}\to\mathbb{R}$. We assume that $\rho(0)=1/2$. The main example of such a correlation function will be that of fractional Brownian increments. That is, we will consider a rich source of examples where
$
\rho(k) = s_H(k) :=\frac{1}{2} \E [ B_1^H  (B^H_{\vert k\vert+1}  - B_{\vert k\vert}^H)  ]$  for $k\in\mathbb{Z}
$
with $B^H$ being the {\it fractional Brownian motion} with Hurst parameter $H\in(0,1)$.  The multiplicative constant $1/2$ is chosen only for normalization purposes and 
\begin{align}\label{s_H}
s_H(k) = \frac{1}{4} \big( \vert k+1\vert^{2H}  + \vert k-1\vert^{2H}- 2 \vert k \vert^{2H} \big)   \, ;
\end{align} 
one can easily check that for $H\neq 1/2$,   as $\vert k\vert\to+\infty$,
\begin{align} \label{cor-est}
s_H(k) \sim  c_H \vert k\vert^{2H-2} \, ,  \end{align}
where $c_H := H(2H-1)/2$ is uniformly bounded by $1/2$. For a brief introduction to the fractional Brownian motion, one can refer to the recent book \cite{Ivanfbm}. 

In the following,  we first present   a very peculiar  phenomenon arising from the fractional Brownian example as a prelude,        and  we postpone   results concerning more general correlation functions $\rho$
to Section~\ref{sec:stationary}.

\begin{theorem} \label{SY-all}
  Let $\textbf{a}, \textbf{b}, \textbf{c}$ be i.i.d.~copies of $\{ G_1, \ldots, G_n\}$ with correlation function $s_H$  for {\bf any} given  $H\in(0, 1)$.    Then,  with high probability,
  \begin{align}\label{corderp}
    \text{$\textbf{a}$ beats $\textbf{b}$}
    \quad \text{if and only if } \quad \sum_{i=1}^n F(a_i) > \sum_{i=1}^n  F(b_i)  \, ,
\end{align}
where $F(x)= \Phi(\sqrt{2} x)$ is the distribution function of $G_1\sim N(0, 1/2)$.
As a consequence,  the probability that three dice $\bfa, \bfb, \bfc$ are transitive tends to one, as $n\to+\infty$ .
 
\end{theorem}

 \begin{remark} (i) The case $H=1/2$ corresponds to the aforementioned
   unconditional Gaussian dice,  and by the standard integral transform,
   it extends to unconditional dice with i.i.d.~faces sampled from a large
   class of distributions; see Theorem \ref{thm41}.
   As already mentioned, \cite{PolVII} gives an elementary proof for unconditioned
   uniform dice.
   
   \smallskip

\noindent (ii) For $k\neq 0$,  $s_H(k) > 0$ if $H\in(1/2, 1)$ while  $s_H(k) < 0$
   whenever $H\in(0, 1/2)$. Theorem~\ref{SY-all} suggests that  \emph{negative correlation or positive correlation among different faces does not influence
     formula \eqref{corderp}}, and therefore also the transitivity of {\it $\bfa, \bfb,\bfc$.   }

 \end{remark}

The proof of Theorem~\ref{SY-all} makes use of the 
very close relation between the Hermite expansions of functions
$\1[\bullet>0]$ and $\Phi$:
\begin{align}
     \1\big[ \bullet > 0\big] = \frac{1}{2} + \sum_{k=0}^\infty d_{2k+1} H_{2k+1}, &\quad\text{with $d_{2k+1} =  \frac{(-1)^k}{2^k k! (2k+1) \sqrt{2\pi}}$,} \label{ind_chaos} \\
 \Phi = \frac{1}{2} + \sum_{k = 0}^\infty \ell_{2k+1} H_{2k+1} \,, &\quad\text{with $\ell_{2k+1} = d_{2k+1}  2^{-k-\frac{1}{2}} $, } \label{Phi-exp}
\end{align}
where the above series converge in $L^2(\mathbb{R}, \exp(-x^2/2)dx)$; see Section \ref{sec:stationary} for more details. 

\subsection{Condorcet paradox: Social chaos for close majority elections}
\label{sec:condorcet-intro}

The Condorcet paradox is a well-known intransitivity phenomenon
in social choice theory.
Consider $n$ voters trying
to decide between $k$ alternatives. Each voter has a ranking (linear ordering)
of the alternatives and we would like to aggregate the $n$ rankings into
a global one.
A natural approach is as follows:
given a pair of alternatives $a$ and $b$, we say that $a$ beats $b$
if a majority of voters put $a$ ahead of $b$ in their rankings
(we always assume $n$ is odd to avoid dealing with ties).
Aggregating these majority elections for all $K := \binom{k}{2}$ pairs of
alternatives, we obtain a tournament graph on $k$ vertices, that is, a complete graph where each edge is directed.

If there exists a Condorcet winner
({\it i.e.}~the alternative that beats all others),
and, in particular, if this tournament is transitive
({\it i.e.}~it induces a linear ordering), we might conclude
that there is a clear global winner of the election.
However,
in Condorcet paradox the pairwise rankings need not produce
a Condorcet winner. For example, we might have three
voters with rankings $a \beats b \beats c$, $b \beats c \beats a$ and
$c \beats a \beats b$, respectively. Majority aggregation results
in $a$ beating $b$, $b$ beating $c$ and $c$ beating $a$.

Assume a probabilistic model with $n$ voters
and $k$ alternatives, where
each voter samples one of $k!$ rankings independently and uniformly.
This is called the \emph{impartial culture} assumption and is the most
common model studied in social choice
(see~\cite{Geh02} for one survey of results in related settings).
Despite the example above, one might hope that under impartial culture,
the paradox is unlikely to arise for a large number of voters.
However, it was one of the earliest results in social choice theory
\cite{Gui52, GK68}
that it is not so:
in particular, letting
$P_{\mathrm{Cond}}(k, n)$ to be the probability of Condorcet winner for $n$ voters
and $k$ alternatives, and
$P_{\mathrm{Cond}}(k) := \lim_{n \to \infty} P_{\mathrm{Cond}}(k, n)$,
we have
\begin{align}
  P_{\mathrm{Cond}}(3) =
  \frac{3}{2\pi}\arccos(-1/3) \le 91.2\% \; .
  \label{eq:43}
\end{align}
For $k \ge 4$ there is no simple expression, but the numerical values up
to $k=50$ were computed by Niemi and Weisberg \cite{NW68}; for example,
$P_{\mathrm{Cond}}(10) \approx 51.1\%$ and
$P_{\mathrm{Cond}}(27) \approx 25.5\%$,
and the asymptotic behavior is given by May \cite{May71} as
\begin{align}
  P_{\mathrm{Cond}}(k) =  \frac{\sqrt{8\pi \log k} }{k}
  \big(1 + O (1/\log k )\big) \; ,
  \label{eq:21}
\end{align}
in particular 
$\lim_{k \to \infty} P_{\mathrm{Cond}}(k) = 0$.
If one is interested in the probability of a completely transitive outcome,
the best asymptotic estimate known \cite{Mos10} is $\exp(-\Theta(k^{5/3}))$.

Given the dice models studied in
\cite{CGGLM}
and \cite{Pol17},
it seems reasonable to study the probability of Condorcet paradox under impartial culture,
conditioned on all pairwise elections being close to tied. The conditioning on elections being almost tied seems natural also given the abundance of real life elections that are close to tied. 

To define the model more precisely, for each pair of alternatives $\{a,b\}$, define the random variable $S^{(ab)}$
to be the number of voters that prefer $a$ to $b$, minus the number of voters preferring
$b$ to $a$. In other words, the sign of $S^{(ab)}$ determines the alternative
that wins the pairwise election. Let $Y^{(ab)} := \sgn(S^{(ab)})$ and $Y$ be
the random tuple encoding the $K$ pairwise winners via the $Y^{(ab)}$, having $K$ entries with values in $\{-1, 1\}$.
Furthermore, for $d \ge 1$, let $\cE_d$ be the event
that $\left|S^{(ab)}\right| \le d$ for every pair $\{a,b\}$.
We think of the event $\cE_d$ as ``the elections are $d$-close'',
with $d=1$ corresponding to almost perfectly tied elections.

\medskip

Our main result for voting uses a multidimensional local limit theorem to show that the probability
of Condorcet winner for almost tied elections goes to zero much faster
than  in \eqref{eq:21}.
Actually, we prove the following stronger result.\begin{theorem}
  \label{thm:elections-random}
  Let $n$ be odd, $d \ge 1$ and
  $y \in \{-1, 1\}^K$. Then,
  \begin{align}
    \Big|
    \Pr\left[ Y = y \mid \cE_d \right] - \frac{1}{2^K}
    \Big| \le \alpha_k \frac{d^2}{n} + o_k(1) \; ,
    \label{eq:20}
  \end{align}
  where $\alpha_k > 0$ depends only on $k$ and $o_k(1)$ denotes a function
  that depends only on $k$ (but not on $d$ or $y$) and goes to zero, as $n$
  goes to infinity.

  In particular,
  \begin{align}
    \Big| \Pr\left[ Y \text{ is transitive } \mid \cE_d \right] - \frac{k!}{2^K} \Big|
    \le \beta_k \frac{d^2}{n} + o_k(1)
    \label{eq:24}
  \end{align}
  and
  \begin{align}
    \Big| \Pr\left[ Y \text{ has Condorcet winner} \mid \cE_d \right]
    - \frac{k}{2^{k-1}} \Big| \le \gamma_k \frac{d^2}{n} + o_k(1) \; 
    \label{eq:25}
  \end{align}
  for some $\beta_k, \gamma_k > 0$.
\end{theorem}

One interpretation of this result is that the probability of Condorcet paradox,
which is already substantial without conditioning, increases to reach
the fully chaotic behavior
for elections that are almost three-way ties.
The event $\cE_d$ for $d=o(\sqrt{n})$ has subconstant probability, but on the
other hand such ``close'' elections seem to be a natural case to study
(and one might argue that in practice they arise more often than the model
suggests). Furthermore, some other interesting phenomena in social choice can be shown
to arise only with polynomially small probability, see, {\it e.g.}~the quantitative
Gibbard-Satterthwaite theorem \cite{FKKN11, IsKiMo:12, MosselRacz:15}. 

Comparing Theorem~\ref{thm:elections-random}
to intransitivity of random uniform dice conditioned on their face-sums,
first note that for almost tied elections and $k=3$, the asymptotic probability of
Condorcet winner computed from \eqref{eq:25} is $3/4$, which is equal to the probability of transitivity
for dice. 
On the other hand, there is a difference in the transition between the
transitive and chaotic regimes. Assuming dice with faces uniform in $(-1, 1)$,
the model is chaotic when conditioned on face-sums equal to zero, but,
as shown by Polymath \cite{PolVII},
it becomes transitive as soon as we condition on face-sums of absolute value at
most $d$ for $d = \omega(\log n)$. However, the voting outcomes
behave chaotically for $d$-close elections for any $d = o(\sqrt{n})$
and transition into the ``intermediate'', rather than transitive,
regime given by~\eqref{eq:43}.
Furthermore,~\eqref{eq:20} means
that the tournament on $k$ alternatives determined by $Y$ is
asymptotically random.
\cite{CGGLM}
conjectured that $k$ random dice also form a random
tournament, however \cite{Pol17} report experimental evidence against this
conjecture.

We also note that the proof of Theorem~\ref{thm:elections-random} can be
modified such that its statement holds even when conditioning on
only $K-1$ out of $K$ pairwise elections being $d$-close.

The above-mentioned work by Kalai \cite{Kal10} calls the situation when
$Y$ is a random tournament \emph{social chaos}.
He considers impartial culture model (without conditioning)
and an arbitrary monotone odd function $f\colon\{-1, 1\}^n \to \{-1, 1\}$
for pairwise elections (the setting we considered so far corresponds to $f = \Maj_n$).
Under these assumptions, he proves that social chaos is equivalent
to the asymptotic probability of Condorcet winner for three alternatives being
equal to $3/4$.
\cite{Kal10} contains another equivalent condition for social chaos,
stated in terms of noise sensitivity of function $f$ for only
two alternatives. It is interesting to compare it with the reduction
from three to two dice in Lemma~2.1 of \cite{Pol17}.

\subsection{Condorcet paradox: Generalizing close elections -- A case study}

It would be interesting to extend Theorem~\ref{thm:elections-random}
to other natural pairwise comparison functions  such as 
weighted majorities and recursive majorities,
similar to the electoral college in the USA.
However, in order to formulate such a result, it is first necessary to define $d$-close elections for an
arbitrary function. The results of this section deal with the question if such a definition exists. Somewhat surprisingly, we show that natural definitions of close elections do not lead to a chaotic outcome when ranking three alternatives. We do so by presenting a simple   example, for which two of the most natural definitions do not result in chaotic outcome. 

For this we consider the following function. 
Let us assume that there are
three candidates $a$, $b$, $c$ and a number of voters
$n$ that is divisible by three, letting $m := n/3$.
We take
$f\colon\{-1, 1\}^n \to \{-1, 1\}$ to be
\begin{align*}
  f(x_1, \ldots, x_n) :=    \sgn\left( \sum_{i=1}^m \sgn\left( x_{3i-2}+x_{3i-1}+x_{3i} \right) \right) \; .
\end{align*}
In words, $f$ is a two-level majority:
majority of votes of $m$ triplets, where the vote of each triplet is decided by majority.

The function $f$ possesses many pleasant properties: it is odd,
transitive symmetric\footnote{
  A voting function $f:\{-1,1\}^n\to\{-1,1\}$
  is transitive symmetric if for every $i,j\in[n]$ there exists a
  permutation $\sigma:[n]\to[n]$ such that $\sigma(i)=j$ and
  $f\circ\sigma = f$, where
  $(f\circ\sigma)(x_1,\ldots,x_n)=f(x_{\sigma(1)},\ldots,x_{\sigma(n)})$.
  Informally, every two voters play the same role.
}
and is a polynomial threshold function of degree three.
We would like to devise a natural notion of $d$-close elections according
to $f$. In light of Theorem~\ref{thm:elections-random} it might be
argued that the ``right'' notion of closeness should result in the chaotic outcome,
same as for majority.  We show that for two natural definition of closeness, {\em this is not the case}. 

To start with, let $w_i := x_{3i-2} + x_{3i-1} + x_{3i}$. In the following
we will sometimes treat $f$ as a function of $\mathbf{w} := (w_1, \ldots, w_m)$,
i.e., $f\colon \{\pm 1, \pm3\}^m \to \{\pm 1\}$, with the distribution of $\mathbf{w}$
induced by the distribution of $\mathbf{x}$, {\it i.e.}, $w_i = \pm 3$ and
$w_i = \pm 1$ with probabilities $1/8$ and $3/8$, respectively.
A CLT argument as in
Theorem~\ref{thm:elections-random} implies chaotic behavior of $f$
if we define
``$d$-close'' as
``$\big\vert \sum_{i=1}^m \sgn \big(w_i^{(kk')} \big) \big\vert \le d$''
for every pair of candidates $(kk')$. However, this is not very satisfactory
for at least two reasons. First, it does not seem to extend to other functions
that do not have such an ``obvious'' summation built into them.
Second, it does not accord well with our intuition of
closeness. This second problem becomes more apparent considering analogous
condition for another two-level majority, with $\sqrt{n}$ groups of
$\sqrt{n}$ voters each. In this case of ``electoral college'' an election
that was close in every ``state'' in favor of a single candidate
would not be considered close overall.

Another idea is to define ``$d$-close'' the same way as in
Theorem~\ref{thm:elections-random}, that is as
``\,$\big| \sum_{i=1}^n x_i^{(kk')} \big| \le d$\,''. Clearly, this is not a good
closeness measure for an arbitrary comparison method ({\it e.g.}, weighted majority
with large differences between weights),
but one could argue that it is relevant
at least for transitive symmetric functions. Using another CLT argument,
we find that for this definition of closeness, the behavior of $o(\sqrt{n})$-close
elections under $f$ is not chaotic: the
asymptotic Condorcet paradox probability is slightly less than $25\%$.
Note that for three candidates, the Condorcet paradox occurs if and only if 
$f (\mathbf{x}^{(ab)} ) =
f (\mathbf{x}^{(bc)} ) = f (\mathbf{x}^{(ca)} )$.

\begin{theorem}
  \label{thm:triplet-sum}
  Under the notation above and the event $\cE_d$ as defined in
  Section \ref{sec:condorcet-intro}, for $d = \sqrt{n}/\log n$,
  \begin{align*}
    \lim_{n \to \infty }\Pr\left[
    f (\mathbf{x}^{(ab)} ) = f (\mathbf{x}^{(bc)} )
    = f (\mathbf{x}^{(ca)} ) \mid \cE_{d} 
    \right] = \alpha^* \; ,
  \end{align*}
  where $\alpha^* \approx 23.2\%$ is an absolute constant.
\end{theorem}
For comparison, without conditioning the Condorcet paradox
probability is $\approx 12.5\%$ when the elections are according to $f$
and $\approx 8.8\%$ according to majority.

The idea for the proof of Theorem~\ref{thm:triplet-sum} is to use multivariate
Berry-Esseen theorem for random variables
\begin{align*}
\left(A^{(kk')}, B^{(kk')}\right)_{(kk')} :=
\left(\sum_{i=1}^n x_i^{(kk')}, \sum_{i=1}^m \sgn\left(w_i^{(kk')}\right)
  \right)_{(kk')},\; kk' \in \{ab, bc, ca\}\;.
\end{align*}
We are looking at sign patterns of $B^{(kk')}$ conditioned on small absolute
values of $A^{(kk')}$. $A^{(kk')}$ and $B^{(kk')}$ are not perfectly correlated
and it turns out that part of (negative) correlations between
$B^{(ab)}, B^{(bc)}$ and $B^{(ca)}$ is not attributable to
correlations between $A^{(ab)}$, $A^{(bc)}$ and $A^{(ca)}$. Hence, even
after conditioning on small $A^{(kk')}$ there remains a small constant
correlation between $B^{(kk')}$, which prevents completely chaotic behavior.

Another promising definition of closeness involves the noise operator
$T_\rho$ from the analysis of Boolean functions (see {\it e.g.,}~\cite{Dol14}
for more details). Let $\rho \in [-1, 1]$
and $\mathbf{x} \in \{-1, 1\}^n$. Define a probability distribution
$N_{\rho}(\mathbf{x})$ over $\{-1, 1\}^n$ such that
$y_1, \ldots, y_n$ are sampled independently with $y_i = -x_i$ with
probability $\eps := \frac{1-\rho}{2}$ and $y_i = x_i$ otherwise.
Note that $\EE[x_iy_i] = \rho$, hence we say that a pair $(\mathbf{x},
\mathbf{y})$ sampled as uniform $\mathbf{x}$ and then $\mathbf{y}$
according to $N_\rho(\mathbf{x})$ is \emph{$\rho$-correlated}.
The \emph{noise operator}
$T_\rho$ is defined as
\begin{align*}
  T_\rho f(\mathbf{x}) := \EE_{\mathbf{y} \sim N_\rho(\mathbf{x})} \left[
  f(\mathbf{y}) \right] \; .
\end{align*}

For $\rho \in (0, 1)$ one can think of
$N_\rho(\mathbf{x})$ as a distribution over $\{-1, 1\}^n$ with probabilities
that are decreasing in the Hamming distance from $\mathbf{x}$.
Furthermore, for $f$ being majority and $d = o(\sqrt{n})$
the condition $\left| \sum_{i=1}^n x_i \right| \le d$
is asymptotically equivalent to
$\left|T_\rho \Maj\left(\mathbf{x}\right)\right| \le C_\rho d /\sqrt{n}$. 
This suggests that it may be fruitful to define ``$d$-close'' as
``$ |T_\rho f (\mathbf{x}^{(kk')} ) | \le d/\sqrt{n}$''.
The idea becomes even more appealing when considering a Fourier-analytic
Condorcet formula discovered by Kalai \cite{Kal02}.
He showed that for an odd function $g\colon \{-1, 1\}^n \to \{-1, 1\}$,
the probability of Condorcet paradox without conditioning is equal to
\begin{align}
  \Pr\left[ g (\mathbf{x}^{(ab)} ) = g (\mathbf{x}^{(bc)} )
  = g (\mathbf{x}^{(ca)} ) \right]
  &=
    \frac{1}{4}\left(1 - 3\EE_{\mathbf{x}, \mathbf{y}}
    \left[ g(\mathbf{x})g(\mathbf{y})\right] \right)\nonumber\\
  &= \frac{1}{4}\left(1 - 3\EE_{\mathbf{x}} \left[
    g(\mathbf{x}) T_{1/3} g(\mathbf{x}) \right]\right) \; ,
    \label{eq:44}
\end{align}
where $(\mathbf{x}, \mathbf{y})$ are $1/3$-correlated.

Another feature of the $T_\rho$ operator is that for noise sensitive functions
(which \cite{Kal10} proved to be exactly those that result in chaotic
elections without conditioning) the value $|T_\rho f(\mathbf{x})|$ is $o(1)$
with high probability over $\mathbf{x}$. If we decide to use
$|T_\rho f(\mathbf{x})|$ as a measure of closeness, then this fact can be
given the following (though by no means the only possible) interpretation:
elections held according to a noise sensitive function are almost
always close.

Recall our ``majority of triplets'' function $f$ and
define the event $\cF_{\rho, d}$ as
\begin{align*}
  \cF_{\rho, d} :\equiv \quad
  \max\left(
  \big| T_{\rho} f ( \mathbf{x}^{(ab)}  ) \big|,
  \big| T_{\rho} f ( \mathbf{x}^{(bc)}  ) \big|,
  \big| T_{\rho} f ( \mathbf{x}^{(ca)}  ) \big|
  \right) \le \frac{d}{\sqrt{m}} \; .
\end{align*}
At first sight, \eqref{eq:44} suggests that the event $\cF_{\rho, d}$, with $\rho=1/3$ and
$d = o(\sqrt{m})$, should cause the expectation term in \eqref{eq:44}
to vanish and the probability of Condorcet paradox to approach $1/4$.
Surprisingly, this is not the case for $f$:

\begin{theorem}
  \label{thm:triplet-noise}
  Fix $\rho \in (0, 1)$ and take
  $d := \sqrt{m}/\log m$.
  Then,
  \begin{align*}
    \lim_{n \to \infty} \Pr\left[
    f (\mathbf{x}^{(ab)} ) = f (\mathbf{x}^{(bc)} )
    = f (\mathbf{x}^{(ca)} ) \mid \cF_{\rho, d}
    \right] = \alpha(\rho) \; ,
  \end{align*}
  where $\alpha(\rho) \in [0.17, \alpha^*]$ with
  $\alpha^*$ the constant from Theorem~\ref{thm:triplet-sum}
  and $\alpha(\rho) \to \alpha^*$ as $\rho \to 0^+$.
\end{theorem}

The proof of Theorem~\ref{thm:triplet-noise} is a variation on 
the proof of Theorem~\ref{thm:triplet-sum}. For $\mathbf{w} \in \{\pm 3, \pm 1\}^m$
and $b \in \{\pm 3, \pm 1\}$, we let
$W_b(\mathbf{w}) := \left| \left\{ i \in [m]: w_i = b \right\}\right|$
and 
$V_b(\mathbf{w}) := W_b(\mathbf{w}) - \EE_{\mathbf{w'}}\left[W_b(\mathbf{w'})\right]$.
Then, we observe that, just as for majority the value of $T_\rho \Maj(\mathbf{x})$
is proportional to the number of ones in $\mathbf{x}$ minus $n/2$,
also for $f$ the value of $T_\rho f(\mathbf{w})$ is proportional to a certain
linear combination of $V_b(\mathbf{w})$. This allows us to proceed with
an identical argument as in Theorem~\ref{thm:triplet-sum} with appropriately
redefined random variables $A^{(kk')}$.

Some more recent results show that, without conditioning, majority in fact
maximizes the probability of Condorcet winner among ``low-influence functions''
(see \cite{MOO10} for three voters and
\cite{Mos10, IsakssonMossel:12} for general case). This contrasts
with Theorems~\ref{thm:triplet-sum} and~\ref{thm:triplet-noise}
for different definitions of close elections.

\subsection{Arrow's theorem for dice}

To further consider the parallels between dice and social
choice, we also ask if there is a dice analogue of Arrow's theorem
(and its quantitative version). We obtain a rather generic statement
that does not use any properties of dice and a quantitative version
which is a restatement of a result on tournaments by Fox
and Sudakov \cite{FS08}.

\paragraph{Organization of the paper}
The proofs of our main theorems are located in Sections~\ref{sec:nonuniform}
(Theorem~\ref{con:weak}),~\ref{sec:stationary} (Theorem~\ref{SY-all}),~\ref{sec:condorcet}
(Theorem~\ref{thm:elections-random}) and~\ref{sec:triplet}
(Theorems~\ref{thm:triplet-sum} and~\ref{thm:triplet-noise}).
Section~\ref{sec:arrow} contains the discussion of Arrow's theorem for dice.
The sections are mostly self-contained and can be read in any order.

\section{Transitivity of non-uniform dice}
\label{sec:nonuniform}

In this section we are going to prove Theorem~\ref{con:weak}. Let us start
with some notation.
For the sake of readability, in this section we drop the bold typesetting for dice vectors. 
We let 
$$W^{(kk')}_{ij} := \1(k_i > k'_j)$$ for $k,k' \in \{a,b,c\}$ and
$$W^{(kk')} = \sum_{i,j=1}^n W^{(kk')}_{ij}.$$
We also let $V^{(kk')} := \sum_{i=1}^n F(k_i)-F(k'_i)$.
An important value that we will use is
\begin{align}
  A := \EE[a_1 F(a_1)] \; .
\end{align}
The constant $A$ is significant because it distinguishes the uniform
distribution: by Cauchy-Schwarz we have
\begin{align*}
A^2 = \EE[a_1F(a_1)]^2 = \EE[a_1(F(a_1)-1/2)]^2 \le
  \Var[a_1] \cdot \Var[F(a_1)] = \frac{1}{12} 
\end{align*}
(note that $F(a_1)$ is uniform in $(0, 1)$, so $\EE[F(a_1)]=1/2$
and $\Var[F(a_1)]=1/12$). On the other hand, since $a_1$ and $F(a_1)$
are linearly dependent if and only if distribution of $a_1$ is uniform
on $(-\sqrt{3},\sqrt{3})$, the equality $A^2 = 1/12$ is achieved
exactly for the uniform distribution. In the non-uniform case, this leads
to a key cancellation leading to~\eq{eq:62} below.

Since for a non-uniform distribution clearly we have
$$\Pr\left[\sum_{i=1}^n F(k_i) = \sum_{i=1}^n F(k'_i)\mid\cE_0\right] = 0$$
(see also the proof of Proposition~\ref{prop:general-anti}), the second
statement of Theorem~\ref{con:weak} follows from the first.
What needs to be done can be summed up in two propositions.
In the following proof we assume conditioning on $\cE_0$ and drop it
from the notation for readability.
We also note that 
constants hidden in $O(\cdot), o(\cdot)$, etc., are allowed to depend on the
distribution~$F$.

\begin{proposition}
  \label{prop:half-variance}
  \begin{align}
    \Var\left[ W^{(ab)} - n V^{(ab)} \right]
    &= o(n^{3}) \; . \label{eq:62}
  \end{align}
\end{proposition}


\begin{proposition}
  \label{prop:general-anti}
  For every $C \in \mathbb{R}$ and $\eps > 0$,
  \begin{align}
    \label{eq:69}
    \Pr\left[ \frac{V^{(ab)}}{\sqrt{n}} \in [C-\eps, C+\eps]
    \right] = O(\eps)+O\left(\frac{1}{\sqrt{n}}\right)\; ,
  \end{align}
  where the $O(\cdot)$ constants do not depend on $C$ or $\eps$.
\end{proposition}

We note that during the proof of Proposition~\ref{prop:half-variance}
we establish $\Var[W^{(ab)}],\allowbreak\Var[nV^{(ab)}]\allowbreak\ge\Omega(n^3)$,
so indeed Proposition~\ref{prop:half-variance} is saying that
these two random variables are closely correlated.

\begin{proof}[Theorem~\ref{con:weak} follows from the propositions]
  Let $\overline{W}^{(kk')} := W^{(kk')} - \EE[W^{(kk')}] = W^{(kk')} - n^2/2$.
  It is enough to prove that
  \begin{align*}
    \Pr\left[\sgn\left(V^{(ab)}\right) \ne \sgn\left(\overline{W}^{(ab)}\right)\right] =
    o(1) \; .
  \end{align*}
  For any $\delta > 0$, note that
  $\sgn\left(V^{(ab)}\right) \ne \sgn\left(\overline{W}^{(ab)}\right)$
  implies that 
  \begin{center}
  either
  $\left|nV^{(ab)}-\overline{W}^{(ab)}\right| > \delta$
  or $\left|nV^{(ab)}\right| \le \delta$.
  \end{center}
  Furthermore, by Chebyshev's inequality and \eqref{eq:62},
  \begin{align*}
    \Pr\left[\left|\overline{W}^{(ab)}-nV^{(ab)}\right| > \delta\right]
    < \frac{o(n^3)}{\delta^2} \; .
  \end{align*}  
  Taking appropriate $\delta := o(n^{3/2})$, we finally compute
  \begin{align*}
    &\Pr\left[\sgn\left(V^{(ab)}\right)
      \ne\sgn\left(\overline{W}^{(ab)}\right)\right]\\
    &\qquad\qquad\le\Pr\left[\left|nV^{(ab)}-\overline{W}^{(ab)}\right| > \delta \right]
      + \Pr\left[\left|nV^{(ab)}\right| \le \delta\right]\\
    &\qquad\qquad
      =o(1) + O\left(\frac{\delta}{n^{3/2}}\right)
      = o(1) \; ,
  \end{align*}
  where we used \eqref{eq:69} in the last line.
\end{proof}

\begin{remark}
  It is also true that with high probability $a$ beats $b$
  if and only if $\sum_{i=1}^n F_n(a_i)\allowbreak> \sum_{i=1}^n F_n(b_i)$,
  where $F_n$ is the CDF of the \emph{conditional marginal}
  of $a_1$ (or any $a_i$) conditioned on $\cE_0$, rather than the unconditional marginal $F$ as in Theorem~\ref{con:weak}.   
  (Some numerical experiments suggest that
  $F_n$ is a better predictor of the ``strength'' of a die than
  $F$.) To see why this is true, if $V'^{(ab)} := \sum_{i=1}^n F_n(a_i)-F_n(b_i)$,
  then similar calculations
  to those in the proof of Proposition~\ref{prop:half-variance} yield
  \begin{align*}
    \Var\left[V'^{(ab)}-V^{(ab)}\right] = o(n) \; ,
  \end{align*}
  and using this in the bound
  \begin{align*}
    &\quad\Pr\left[\sgn\left(V'^{(ab)}\right)\ne\sgn\left(\overline{W}^{(ab)}\right)
    \right]\\
    &\le\Pr\left[\left|nV^{(ab)}-\overline{W}^{(ab)}\right|>\delta\right]
    +\Pr\left[\left|nV^{(ab)}-nV'^{(ab)}\right|>\delta\right]
    +\Pr\left[\left|nV^{(ab)}\right|\le\delta\right],
  \end{align*}
 the result follows similar to above.
\end{remark}

We proceed to prove the propositions, starting with the shorter proof
of Proposition~\ref{prop:general-anti}. In both proofs
we do \emph{not} assume conditioning on $\cE_0$ by default.

\subsection{Proof of Proposition~\ref{prop:general-anti}}

For simplicity we will assume that $n = 2m$.
The idea of the proof is as follows: First, by independence, it is
enough to establish anti-concentration for the single-die random variable
$\sum_{i=1}^n F(a_i)$.
Since the single-face distribution
is not uniform, there must exist two points $x^*, y^* \in \supp(f)$ such that
\begin{align}
  F(x^*) + F(y^*) \ne 2F(z^*) \; ,\label{eq:104}
\end{align}
where $z^* := \frac{x^*+y^*}{2}$. Consider random variables
$d_1, \ldots, d_m$ given by
\begin{align}
  \label{eq:107}
  d_i := a_{2i-1}+a_{2i} \; .
\end{align}
By a concentration argument, with high probability,
for a constant fraction of coordinates $i \in \{1,\ldots,m\}$,
it must be that $d_i \approx 2z^*$.
Furthermore, after conditioning on $d_1, \ldots, d_m$, for each such coordinate
it must be that for $d_i \approx 2z^*$, both
\begin{equation}\label{eq:105}
\begin{split}
  &\qquad a_{2i-1} \approx x^*, a_{2i} \approx y^*\;,\\
  &\qquad a_{2i-1}, a_{2i} \approx z^*\;,
  \end{split}
\end{equation}
are possible with constant probability. But \eqref{eq:104} and
\eqref{eq:105} imply that, even conditioned on $d_1, \ldots, d_m$,
the variance of $\sum_{i=1}^n F(a_i)$ is at least $\Omega(n)$, and that allows us
to apply Berry-Esseen theorem to establish a (conditional) CLT and
anti-concentration. Below we present this argument in more detail,
starting with an auxiliary concentration lemma.

\begin{lemma}
  \label{lem:conditional-concentration}
  Let $x \in \supp(f)$ and $\delta > 0$. There exist constants
  $\alpha := \alpha(f, \delta) > 0, \beta := \beta(f, \delta) > 0$
  such that
  \begin{align}
    \label{eq:106}
    \Pr\big[ \left|\left\{i \in [n]:
    x-\delta\le a_i\le x+\delta \right\}\right| < \alpha n\mid\cE_0\big]
    \le O\left(\exp\left(-\beta n\right)\right) \; .
  \end{align}
\end{lemma}
\begin{proof}
  We will think of sampling
  $a_1, \ldots, a_n$ conditioned on $\cE_0$ as an experiment on
  $n-k$-dimensional space for some $k \in \mathbb{N}$, where
  the density
  of $(a_1, \ldots, a_{n-k})$
  is proportional to $\prod_{i=1}^{n-k} f(a_i) \cdot f^{(*k)}(-a_0)$,
  with $a_0 := \sum_{i=1}^{n-k} a_i$ and
  $f^{(*k)}$ being the $k$-fold convolution of the PDF $f$.

  Take $\eps > 0$ and consider a set
  \begin{align*}
    I_{k, \eps} := \left\{ x \in \mathbb{R}: f^{(*k)}(x) > \eps \right\} \;.
  \end{align*}
  Since $f$ is continuous and its support is an interval that necessarily contains
  zero, it must be that for every $L > 0$ there exist
  $k$ large enough and $\eps$ small enough such that we have the inclusion
  \begin{align*}
    [-L, L] \subseteq I_{k,\eps} \; .
  \end{align*}
  We take such large enough $L$ (as soon specified)
  and fix $k$ and $\eps$ accordingly.
  Consider the i.i.d.~choice of $a_1, \ldots, a_{n-k}$. By the
  Berry-Esseen theorem, 
    \begin{align}
    \Pr_{a_1,\ldots,a_{n-k}}
    \left[-L \le -a_0 \le L \right]
    &= \Pr\left[\frac{-L}{\sqrt{n-k}} \le g \le \frac{L}{\sqrt{n-k}}
    \right] + O\left(\frac{1}{\sqrt{n}}\right) \notag\\
    &=
    \Omega\left(\frac{1}{\sqrt{n}}\right) \; ,   \label{eq:77}
  \end{align}
  where $g$ is a standard Gaussian random variable, and the last equality uses
  that $L$ can be chosen large enough to overcome the (potentially negative) error in the normal approximation.
  
  Let $\cF$ be the event from \eqref{eq:106}, the probability of which we are
  bounding and define another event $\cF'$ as
  \begin{align*}
    \cF':\equiv \left|\left\{i\in[n-k]: x-\delta\le a_i\le x+\delta\right\}\right|
    < \alpha n \; .
  \end{align*}
  Taking $M$ to be an upper bound on $f^{(*k)}(y)$ for $y \in \mathbb{R}$ and setting $\alpha:=\IP(x-\delta\le a_1\le x+\delta)/2$,
  we compute
 \begin{align*}
   \Pr\left[\cF\mid\cE_0\right]
   &\le \Pr\left[\cF'\mid\cE_0\right]\\
    &=\frac{\idotsint f(a_1) \cdots f(a_{n-k}) \cdot f^{(*k)}(-a_0) \cdot \1[\cF']
      \, \mathrm{d}a_1 \cdots \mathrm{d}a_{n-k}}
      {\idotsint f(a_1) \cdots f(a_{n-k}) \cdot f^{(*k)}(-a_0)
      \, \mathrm{d}a_1 \cdots \mathrm{d}a_{n-k}}\\
    &\le\frac{M \cdot \Pr_{a_1, \ldots, a_{n-k}} [\cF']}
      {\eps \cdot
      \Pr_{a_1, \ldots, a_{n-k}} [-L \le -a_0 \le L]}\\
   &
     \le O\left(\sqrt{n}\right) \cdot \exp\left(-\beta n\right)
     \le O\left(\exp(-\beta' n)\right)\; ,
  \end{align*}
  where in the last line we used a standard Chernoff bound, since
  the random variable
  \begin{align*}
    \left|\left\{ i \in [n-k]: x-\delta\le a_i\le x+\delta \right\}\right|
  \end{align*}
  can be written as a sum of $n-k$ i.i.d.~Bernoulli random variables
  with mean $2\alpha>0$.
\end{proof}

We continue with the proof of Proposition~\ref{prop:general-anti}, following
the plan from the beginning of the section. For now, we will focus only on
one half of the expression $V^{(ab)}$, namely the sum $\sum_{i=1}^n F(a_i)$.

Recall that by \eqref{eq:104} we have $x^*$, $y^*$, $z^* = (x^*+y^*)/2$ such
that 
$$\gamma := |F(x^*)+F(y^*)-2F(z^*)| > 0.$$ Furthermore, since $F$ is
continuous, we can assume that both $x^*$ and $y^*$ lie in the interior
of the support of $f$. Take small $\delta > 0$ such that
\begin{align*}
[x^*-\delta,x^*+\delta], [y^*-\delta, y^*+\delta], [z^*-\delta, z^*+\delta]
  \subseteq \supp(f)
\end{align*}
and, at the same time, 
\begin{align*}
  \left| w-x^* \right|\le 2\delta
  &\implies \left| F(w)-F(x^*) \right| \le\gamma/10 \;,\\
  \left|w-y^*\right|\le2\delta
  &\implies \left|F(w)-F(y^*)\right|\le\gamma/10\;,\\
  \left|w-z^*\right|\le 2\delta
  &\implies \left|F(w)-F(z^*)\right|\le\gamma/10\;.
  \end{align*}

Recall the random variables
$d_1, \ldots, d_m$ that we defined in \eqref{eq:107}. Note that the
distribution of $d_1/\sqrt{2}=(a_1+a_2)/\sqrt{2}$ satisfies the assumptions of
Theorem~\ref{con:weak}. Therefore, we can apply
Lemma~\ref{lem:conditional-concentration} to $d_1, \ldots, d_m$,
$x = 2z^* \in \supp(f^{(*2)})$ and $\delta$ to obtain that except with
probability $\exp(-\Omega(n))$, we have that,
conditioned on $\cE_0$,
\begin{align}
  \label{eq:108}
  \left|\left\{
  i \in [m]: 2z^*-\delta \le d_i \le 2z^*+\delta
  \right\}\right| \ge \Omega(n) \; .
\end{align}
Observe that the distribution $a_1, \ldots, a_n$ conditioned on $\cE_0$
can be obtained by first sampling $d_1, \ldots, d_m$ conditioned
on $\sum_{i=1}^m d_i = 0$ and then sampling
$a_{2i-1}$ and $a_{2i}$ conditioned on $a_{2i-1}+a_{2i}=d_i$
independently for each $i\in[m]$.

Fix a choice of $d_1, \ldots, d_m$ satisfying \eqref{eq:108}.
We will call $i \in [m]$ that fulfills the condition from \eqref{eq:108}
\emph{good}. We will now show that any such good $i$ assumes values
from \eqref{eq:105} with constant probability. To that end,
let us assume without loss of generality that  $d_1$ is good
and consider
$d \in [2z^*-\delta, 2z^*+\delta]$. We compute (where
$o(1)$ is a function that uniformly goes to zero as $\delta$ goes to zero)
\begin{IEEEeqnarray}{l}
  \Pr\left[ x^*-\delta \le a_{1} \le x^*+\delta \mid a_{1}+a_{2}=d\right]
  =
  \frac{\int_{x^*-\delta}^{x^*+\delta}f(x)f(d-x)\,\mathrm{d}x}
    {\int_{\mathbb{R}}f(x)f(d-x)\,\mathrm{d}x}\nonumber\\
  \qquad\ge \frac{\int_{x^*-\delta}^{x^*+\delta}(f(x^*)+o(1))(f(y^*)+o(1))\,\mathrm{d}x}
  {\max_{d \in [2z^*-\delta, 2z^*+\delta]}f^{(2)}(d)}\nonumber\\
  \qquad\ge c \cdot \delta f(x^*)f(y^*) \ge c' > 0 \; ,
  \label{eq:74}
\end{IEEEeqnarray}
where $c'$ is a positive constant achieved for small enough $\delta$.
A similar argument gives
\begin{align}
  \label{eq:75}
  \Pr\left[z^*-\delta \le a_1 \le z^*+\delta \mid a_1+a_2=d\right]
  \ge c' > 0 \;.
\end{align}
Observe that $a_1 \in [x^*-\delta, x^*+\delta]$
implies $\left|F(a_1)-F(x^*)\right| \le \gamma/10$,
$a_2 \in [y^*-2\delta, y^*+2\delta]$, $|F(a_2)-F(y^*)|\allowbreak\le\gamma/10$
and finally
\begin{align*}
  |F(a_1)+F(a_2)-F(x^*)-F(y^*)| \le \gamma/5 \;,
\end{align*}
giving the overall conclusion 
\begin{align}
  \label{eq:109}
  \Pr\Big[ F(a_1)+F(a_2)\le F(x^*)+F^(y^*)+\gamma/5 \mid a_1+a_2=d \Big] \ge c'\;.
\end{align}
Similarly, $a_1 \in [z^*-\delta,z^*+\delta]$ implies
$a_2 \in [z^*-2\delta,z^*+2\delta]$ and consequently
\begin{align*}
  \left|F(a_1)+F(a_2)-2F(z^*)\right| \le \gamma/5\;,
\end{align*}
in particular
\begin{align*}
  F(a_1)+F(a_2)\ge 2F(z^*)-\gamma/5 \ge F(x^*)+F(y^*)+\gamma/5+\gamma/2\
\end{align*}
and
\begin{align}
  \label{eq:110}
  \Pr\Big[
  F(a_1)+F(a_2)\ge F(x^*)+F(y^*)+\gamma/5+\gamma/2\mid a_1+a_2=d
  \Big]\ge c'\;.
\end{align}
Bounds in~\eqref{eq:109} and~\eqref{eq:110}
together imply that
for any good $i$ we can uniformly lower bound the conditional variance
\begin{align*}
  \Var\left[F(a_{2i-1})+F(a_{2i})\mid a_{2i-1}+a_{2i}=d_i\right] \ge
  \Omega(\gamma^2) \ge \Omega(1)\;.
\end{align*}
Since after
conditioning on $d_1, \ldots, d_m$ satisfying \eqref{eq:108},
the random variables $F(a_{2i-1})+F(a_{2i})$ are bounded and independent
with total variance $\Omega(m)$, we
can apply Berry-Esseen theorem and anti-concentration properties
of a standard Gaussian to obtain
\begin{align*}
  &\Pr\left[ C-\eps \le \sum_{i=1}^n \frac{F(a_i)}{\sqrt{n}}
    \le C+\eps \;\Bigm|\; d_1,\ldots,d_m \right]\\
  &\qquad\qquad=
  \Pr\left[ C-\eps \le \sum_{i=1}^m \frac{F(a_{2i-1)}+F(a_{2i})}{\sqrt{2m}}
  \le C+\eps \;\Bigm|\; d_1,\ldots,d_m \right]\\  
  &\qquad\qquad\le O(\eps) + O\left(\frac{1}{\sqrt{n}}\right) \; .
\end{align*}
Actually, since the sums $\sum_{i=1}^n F(a_i)$ and $\sum_{i=1}^n F(b_i)$
are independent even after conditioning on $\cE_0$, we also get
\begin{align*}
  \Pr\left[ C-\eps \le \frac{V^{(ab)}}{\sqrt{n}}
  \le C+\eps \;\Bigm|\; d_1,\ldots,d_m,d'_1, \ldots,d'_m \right]
  \le O(\eps) + O\left(\frac{1}{\sqrt{n}}\right) \; .
\end{align*}
where $d'_i = b_{2i-1}+b_{2i}$ and $d'_1,\ldots,d'_m$ satisfy
condition~\eqref{eq:108}.
Finally, we get~\eqref{eq:69} by averaging over $d_1, \ldots, d_m,
d'_1, \ldots, d'_m$ and absorbing
exponentially small terms coming from the choices that do not satisfy
\eqref{eq:108}.\qed

\begin{remark}
  One could also prove a variant of Proposition~\ref{prop:general-anti}
  by a two-dimensional local CLT argument. For example, Theorem 19.1 in
  \cite{BR10} could be applied to show that $V^{(ab)}/\sqrt{n}$ conditioned
  on $\cE_0$ converges in law to a Gaussian. However, to apply \cite{BR10}
  it needs to be shown that there exists a finite $k$ such that
  the joint distribution of
  \begin{align*}
    \left(\sum_{i=1}^k a_i, \sum_{i=1}^k F(a_i)\right)
  \end{align*}
  has bounded density. Note that since $F(a_i)$ is a deterministic function
  of $a_i$, for $k=1$ the density does not exist. In some cases it
  is not difficult to show that a small $k > 1$ is enough. For example,
  for a shifted exponential distribution with the PDF
  \begin{align*}
    f(x) = \exp(-x-1)
  \end{align*}
  for $x \in [-1,+\infty)$ we can see that $(a_1+a_2,F(a_1)+F(a_2))$ has
  bounded density since the equation system
  \begin{align*}
      a_1+a_2 &= a\\
      F(a_1)+F(a_2)&=a'
  \end{align*}
  has at most one solution for every pair $(a, a')$. On the other hand,
  a distribution with support $[-2, 2]$ that is (up to normalization) uniform on
  $[-2, -1] \cup [1, 2]$ and Gaussian on $(-1, 1)$
  does not have bounded density for any finite $k$.
\end{remark}

\subsection{Proof of Proposition~\ref{prop:half-variance}}

We prove Proposition~\ref{prop:half-variance} by a somewhat tedious computation.
Recall that in this proof we do \emph{not} assume conditioning on $\cE_0$ by default.
Also, for $k\in\{a,b,c\}$, we will denote by $\cE_k$ the single die event
$\sum_{i=1}^n k_i=0$.

The variance we are looking at can be broken down as
\begin{align}
  &\Var\left[ W - n \sum_{i=1}^n F(a_i) - F(b_i) \mid \cE_0 \right]
  = n^2\Var\left[ \sum_{i=1}^n F(a_i) - F(b_i) \mid \cE_0 \right]\nonumber\\
  &\qquad  +\Var[W \mid \cE_0]
   - 2n \sum_{i,j,k=1}^n \EE\left[\1(a_i > b_j)\cdot(F(a_k)-F(b_k))\mid\cE_0\right]
    \; . \label{eq:81}
\end{align}
The idea is to subdivide each of the three terms above into yet smaller pieces,
each of which can be written down as a certain probability involving
(conditioned and unconditioned) die faces. For example,
\begin{align*}
  \EE\left[\1(a_1>b_1)F(a_2)\mid\cE_0\right]
  =\Pr\left[a_1>b_1\land a_2>c_1\mid\cE_a\cap\cE_b\right] \;.
\end{align*}
Each of those probabilities can be estimated using the following idea:
How does the joint distribution of $(a_1, a_2)$ change after conditioning on $\cE_a$?

Let $\tilde{\phi}_{n-2}(x)$ be the PDF of the distribution of
the sum $\sum_{i=3}^n a_i/\sqrt{n-2}$.
The joint density~$f_n$ of $(a_1, a_2)$ conditioned on $\cE_a$ must be proportional
to $f(a_1)f(a_2)$ multiplied by a ``correction factor''
\be{
\phi_{n-2}(-a_1-a_2) := \sqrt{2\pi}\tilde{\phi}_{n-2}((-a_1-a_2)/\sqrt{n-2}),
} 
which is $\sqrt{2\pi (n-2)}$ times larger than the
density of $\sum_{i=1}^{n-2} a_i$ conditioned on $\cE_a$
(our normalization is chosen so that $\phi_{n-2}(x) \approx 1$
for $x\approx 0$):
\begin{align*}
  f_n(a_1, a_2) = C_n f(a_1)f(a_2)\phi_{n-2}(-a_1-a_2)
\end{align*}
for some normalization constant $C_n \approx 1$. By the CLT, we should have
\begin{align}
  \label{eq:85}
  \phi_{n-2}(-x) \approx \exp\left(-\frac{x^2}{2(n-2)}\right)
  \approx
  1-\frac{x^2}{2n}\;,
\end{align}
and consequently
\begin{align}
  &\Pr\left[a_1 > b_1 \land a_2 > c_1 \mid \cE_a\cap\cE_b\right]
    \nonumber\\
  &\qquad\approx
  C_nC'_n\iint_D f(a_1)f(a_2)f(b_1)f(c_1)\left(1-\frac{(a_1+a_2)^2+b_1^2}{2n}\right)
    \, da_1da_2db_1dc_1 \; ,
    \label{eq:102}
\end{align}
where $D := \{(a_1,a_2,b_1,c_1): a_1 > b_1 \land a_2 > c_1\}$ and $C'_n$
is another normalization constant corresponding to the one-dimensional ``density"
$\phi_{n-1}(-b_1)$. From here,
\eqref{eq:102} can be handled by elementary calculus.
The actual computations are more complicated, since we have to carefully track
errors, including those introduced by the CLT.

\paragraph{Calculation lemma}
We will go over the variance computation assuming the following lemma, which will
be proved afterwards.

\begin{lemma}
  \label{lem:double-cov}
  Let $x$ be a random variable distributed according to $F$ and let
  \begin{align*}
    A &:= \EE[x \cdot F(x)]\;,\\
    B &:= \EE[x^2\cdot F(x)]\;,\\
    \alpha_1 &:= \frac{5\gamma_3^2}{24}-\frac{\gamma_4}{8}\;,\\
    \alpha_2 &:= \frac{\gamma_3}{2}\;,
  \end{align*}
  where $\gamma_j$ denotes the $j$th cumulant of $x$.
  For $k \in \{a,b,c\}$, denote by $\cE_k$ the single-die
  event $\sum_{i=1}^n k_i = 0$.
  We have the following expressions:
  \begin{align}
    \Pr\left[a_1>b_1\land a_2>b_2\mid\cE_0\right]
    &= \frac{1}{4} - \frac{2A^2}{n}
      + o(n^{-1}) \; ,
      \label{eq:79}\\
    \Pr\left[a_1>b_1\mid\cE_a\right]
    &= \frac{1}{2}+\frac{1}{4n}+\frac{\alpha_2A}{n}-\frac{B}{2n}+o(n^{-1})\;,
      \label{eq:79a}\\
    \Pr\left[a_1>b_1\land a_2>b_2\mid\cE_a\right]
    &=\frac{1}{4}+\frac{1}{4n}+\frac{\alpha_2A}{n}-\frac{B}{2n}-\frac{A^2}{n}
      +o(n^{-1})\;,
      \label{eq:79b}\\
    \Pr\left[a_1>b_1\land a_2>c_1\mid \cE_a\cap\cE_b\right]
    &=\frac{1}{4}+\frac{1}{8n}+\frac{\alpha_2A}{2n}-\frac{B}{4n}-\frac{A^2}{n}
      +o(n^{-1})\;.
      \label{eq:79c}
  \end{align}
  Furthermore:
  \begin{align}
    \Pr\left[a_1>b_1\land a_1>b_2\mid\cE_0\right]
    &=\frac{1}{3}+o(1)\;,
      \label{eq:78c}\\ 
    \Pr\left[a_1>b_1\land a_1>b_2\mid\cE_a\right]
    &=\frac{1}{3}+o(1)\;,
      \label{eq:78a}\\
    \Pr\left[a_1>b_1\land a_1>c_1\mid\cE_a\cap\cE_b\right]
    &=\frac{1}{3}+o(1)\;.
      \label{eq:78b}
  \end{align}
\end{lemma}
Since these expressions might look intimidating, let us point out what
we think is one of the most important properties: In contrast to \eqref{eq:79}, it turns out that
\begin{align*}
  \Pr\left[a_1 > b_1 \land a_2 > c_1\mid\cE_0\right]
  = \frac{1}{4}-\frac{A^2}{n} + o(n^{-1}) \; .
\end{align*}
The fact that the errors of order $n^{-1}$ in those two expressions
differ by exactly a factor of two
turns out to imply that $W^{(ab)}+W^{(bc)}+W^{(ca)}$ has small variance,
which, together with anticoncentration argument for $W^{(ab)}$,
implies transitivity similarly as in the proof of Theorem~\ref{con:weak}.
Lemma~\ref{lem:double-cov} is more complicated since we are relating
random variables $W^{(ab)}$ and $V^{(ab)}$, but the $\frac{A^2}{n}$
terms are still crucial, with other terms canceling out one way
or another.

\paragraph{Proof of Proposition~\ref{prop:half-variance} assuming
  Lemma~\ref{lem:double-cov}}
We address each of the three terms in \eqref{eq:81} in turn. First,
using \eqref{eq:78c} and \eqref{eq:79},
\begin{align}
 & \Var[W\mid\cE_0]
  = \Var\left[ \sum_{i,j=1}^n W_{ij}\mid\cE_0 \right] \nonumber\\
  &= O(n^2) + 2n^2(n-1)\Cov\left[W_{11}, W_{12}\mid\cE_0\right]
    + n^2(n-1)^2\Cov\left[W_{11}, W_{22}\mid\cE_0\right] \nonumber\\
  &=O(n^2) + 2n^2(n-1)
    \left(\Pr[a_1 > b_1 \land a_1 > b_2 \mid \cE_0]-\frac{1}{4}\right)
  \nonumber\\
  &\qquad+n^2(n-1)^2\left(\Pr[a_1>b_1\land a_2>b_2\mid \cE_0]-\frac{1}{4}
    \right)\nonumber\\
  &= n^3\left(\frac{1}{6} - 2A^2\right) + o(n^{3}) \; .
    \label{eq:82}
\end{align}
Second, by \eqref{eq:78a}, \eqref{eq:79a} and \eqref{eq:79b},
\begin{align}
 & \Var\left[\sum_{i=1}^n F(a_i)-F(b_i)\mid\cE_0\right]
  = 2\Var\left[\sum_{i=1}^n F(a_i)\mid\cE_0\right]\nonumber\\
  &= 2n\Var[F(a_1)\mid\cE_0] + 2n(n-1)\Cov[F(a_1),F(a_2)\mid\cE_0]\nonumber\\
  &=2n\left(\EE\left[F(a_1)^2\mid\cE_0\right]
    -\EE\left[F(a_1)\mid\cE_0\right]^2\right)\nonumber\\
  &\qquad+2n(n-1)\left(\EE\left[F(a_1)F(a_2)\mid\cE_0\right]
    -\EE\left[F(a_1)\mid\cE_0\right]^2\right)\nonumber\\
  &=2n\left(\Pr\left[a_1>b_1\land a_1>b_2\mid\cE_a\right]
    -\Pr\left[a_1>b_1\mid\cE_a\right]^2\right)\nonumber\\
  &\qquad+2n(n-1)\left(\Pr\left[a_1>b_1\land a_2>b_2\mid\cE_a\right]
    -\Pr\left[a_1>b_1\mid\cE_a\right]^2
    \right)\nonumber\\
  &= n\left(\frac{1}{6} - 2A^2\right) + o(n) \; .
    \label{eq:83}
\end{align}
Finally, recalling $F_n$ is the conditional CDF of $a_1$ given $\cE_a$, and using \eqref{eq:78b}, \eqref{eq:79c} and \eqref{eq:79a} again, we have
\begin{align}
  &\sum_{i,j,k=1}^n \EE\left[\1(a_i>b_j)\left(F(a_k)-F(b_k) \right)\mid\cE_0\right]
    \nonumber\\
  &= \sum_{i,j,k=1}^n \EE\Big[
    F_n(a_i)F(a_k)-(1-F_n(b_j))F(b_k) \mid \cE_0
    \Big]\nonumber\\
  &=2n\sum_{i,j=1}^n\EE\left[F_n(a_i)F(a_j)\mid\cE_0\right]
    -n^2\sum_{i=1}^n\EE\left[F(a_i)\mid\cE_0\right]\nonumber\\
  & =
    2n^2\EE\left[F_n(a_1)F(a_1)\mid\cE_0\right]
    +2n^2(n-1)\EE\left[F_n(a_1)F(a_2)\mid\cE_0\right]
    -n^3\EE\left[F(a_1)\mid\cE_0\right]
    \nonumber\\
  &=
    2n^2\Pr\left[a_1>b_1\land a_1>c_1\mid\cE_a\cap\cE_b\right]
    +2n^2(n-1)\Pr\left[a_1>b_1\land a_2>c_1\mid\cE_a\cap\cE_b\right]
    \nonumber\\
  &\qquad\qquad-n^3\Pr\left[a_1>b_1\mid\cE_a\right]
    \nonumber\\
  & = n^2\left(\frac{1}{6}-2A^2\right) + o(n^{2}) \; .
    \label{eq:84}
\end{align}
Substituting \eqref{eq:82}, \eqref{eq:83} and \eqref{eq:84} into
\eqref{eq:81} gives
\begin{align*}
  \Var\left[W - n \sum_{i=1}^n F(a_i) - F(b_i) \right] = o(n^{3}) \; .
\end{align*}
\qed

It remains to prove Lemma~\ref{lem:double-cov}.

\paragraph{Integration lemma}
The technical part of the proof of Lemma~\ref{lem:double-cov}
consists of the following lemma that replaces the expressions
for $\phi_{n-2}$ and $\phi_{n-1}$ with an appropriate polynomial approximation.
Recall the constants $\alpha_1$ and $\alpha_2$ defined in the statement
of Lemma~\ref{lem:double-cov} and that we defined $\phi_{n-k}$ as the
PDF of $\sum_{i=1}^{n-k} a_i$ multiplied by $\sqrt{2\pi (n-k)}$.

\begin{lemma}\label{lem:technical}
  Let $D$ be a measurable set in $\mathbb{R}^4$ and write
  \[
  f(a,b,c,d) := f(a)f(b)f(c)f(d) \quad  \text{and} \quad f(a,b) := f(a)f(b).\]
  Setting
  $a := a_1+a_2$ and $b := b_1+b_2$
  and denoting Lebesgue integration over $da_1da_2db_1db_2$ by
  $dab$, we have
  \begin{align}
    &\iint_D f(a_1, a_2, b_1, b_2) \cdot \phi_{n-2}(-a)\phi_{n-2}(-b)
      \, dab \nonumber\\
    &\qquad
      = \iint_D f(a_1,a_2,b_1,b_2) \cdot \left(1 + \frac{2\alpha_1}{n} +
      \frac{\alpha_2(a+b)}{n} - \frac{a^2+b^2}{2n}\right) \, dab
        + o(n^{-1}) \; .
      \label{eq:93}
  \end{align}
  Furthermore, using similar notational conventions, we
  get, for $a := a_1$ and $b := b_1$ (and $D\subseteq\mathbb{R}^2$):
  \begin{align}
    &\iint_D f(a, b) \cdot \phi_{n-1}(-a)
      \, dab \nonumber\\
    &\qquad
      = \iint_D f(a, b) \cdot \left(1 + \frac{\alpha_1}{n} +
      \frac{\alpha_2a}{n} - \frac{a^2}{2n}\right) \, dab
        + o(n^{-1}) \; ;
      \label{eq:93a}
  \end{align}
  for $a := a_1+a_2$ and $b:=b_1+b_2$:
  \begin{align}
    &\iint_D f(a_1, a_2, b_1, b_2) \cdot \phi_{n-2}(-a)
      \, dab \nonumber\\
    &\qquad
      = \iint_D f(a_1,a_2,b_1,b_2) \cdot \left(1 + \frac{\alpha_1}{n} +
      \frac{\alpha_2 a}{n} - \frac{a^2}{2n}\right) \, dab
        + o(n^{-1}) \; ;
      \label{eq:93b}
  \end{align}
  and for $a := a_1+a_2$, $b := b_1$ and $c := c_1$:
  \begin{align}
    &\iint_D f(a_1, a_2, b, c) \cdot \phi_{n-2}(-a)\phi_{n-1}(-b)
      \, dabc \nonumber\\
    &\qquad
      = \iint_D f(a_1,a_2,b,c) \cdot \left(1 + \frac{2\alpha_1}{n} +
      \frac{\alpha_2(a+b)}{n} - \frac{a^2+b^2}{2n}\right)\,dabc
        + o(n^{-1}) \; .
      \label{eq:93c}
  \end{align}
\end{lemma}
We state all formulas that we need explicitly
in order to avoid defining and handling new
notation, but we point out the pattern in these expressions: the $\alpha_1/n$
factor is multiplied by the number of the densities in the expression,
the $\alpha_2/n$ factor is multiplied by the sum of all variables
featured in the densities and the quadratic factor is consistent with
the approximation~\eqref{eq:85}.

  Before proving the lemma we point out a corollary that follows by
  setting $D$ to the full integration space and some simple integration
  (keeping in mind $\EE[a_1] = 0$ and $\EE[a_1^2] = 1$).
  The corollary allows us to estimate the normalization constants
  $C_n$ and $C'_n$ (see~\eqref{eq:102}).
  
  \begin{corollary}\label{cor:normalization}
    Keeping the notation from Lemma~\ref{lem:technical}, we have
    \begin{align*}
      \iint_{\mathbb{R}^4}
      f(a_1,a_2,b_1,b_2) \cdot \phi_{n-2}(-a)\phi_{n-2}(-b) \, dab
      &= 1 + \frac{2\alpha_1}{n} - \frac{2}{n} + o(n^{-1})\;,\\
      \iint_{\mathbb{R}^2}
      f(a,b) \cdot \phi_{n-1}(-a) \, dab
      &= 1 + \frac{\alpha_1}{n} - \frac{1}{2n} + o(n^{-1})\;,\\
      \iint_{\mathbb{R}^4}
      f(a_1,a_2,b_1,b_2) \cdot \phi_{n-2}(-a) \, dab
      &= 1 + \frac{\alpha_1}{n} - \frac{1}{n} + o(n^{-1})\;,\\
      \iint_{\mathbb{R}^4}
      f(a_1,a_2,b,c) \cdot \phi_{n-2}(-a)\phi_{n-1}(-b) \, dabc
      &= 1 + \frac{2\alpha_1}{n} - \frac{3}{2n} + o(n^{-1})\;.
    \end{align*}
    Consequently, letting
    $D=\{(a_1,a_2,b_1,b_2):a_1>b_1\land a_2>b_2\}$, we have
    \begin{align}
      & \Pr\left[a_1>b_1\land a_2>b_2\mid\cE_0\right]
        \nonumber\\
      &\qquad=\frac
        {\iint_Df(a_1,a_2,b_1,b_2)\phi_{n-2}(-a)\phi_{n-2}(-b)\,dab}
        {\iint_{\mathbb{R}^4}f(a_1,a_2,b_1,b_2)\phi_{n-2}(-a)\phi_{n-2}(-b)\,dab}\nonumber\\
      &\qquad=\left(1-\frac{2\alpha_1}{n}+\frac{2}{n}\right)
        \iint_D     f(a_1,a_2,b_1,b_2)
        \Bigg(1+\frac{2\alpha_1}{n}+\frac{\alpha_2(a+b)}{n} \notag \\
        &\qquad\qquad\qquad\qquad\qquad\qquad \qquad\qquad\qquad\qquad \qquad -\frac{a^2+b^2}{2n}
        \Bigg)\,dab   +o(n^{-1})   \nonumber\\
      & \qquad=\left(1+\frac{2}{n}\right)
        \iint_Df(a_1,a_2,b_1,b_2)
        \Bigg(1+\frac{2\alpha_2(a_1+b_1)}{n}  \notag \\
        &\qquad\qquad\qquad\qquad\qquad  \quad     -\frac{a_1^2+b_1^2+a_1a_2+b_1b_2}{n}   \Bigg) \,dab +o(n^{-1})\;.     \label{eq:95}
    \end{align}
    Similarly, we have
    \begin{align}
      &\Pr\left[a_1>b_1\mid\cE_a\right]
        \nonumber\\
      &=\left(1-\frac{\alpha_1}{n}+\frac{1}{2n}\right)
        \iint_Df(a,b)
        \Big(1+\frac{\alpha_1}{n}+\frac{\alpha_2a}{n}-\frac{a^2}{2n}
        \Big)dab
        +o(n^{-1})\;,   \nonumber\\
      &=\left(1+\frac{1}{2n}\right)
        \iint_Df(a_1,b_1)
        \Big(1+\frac{\alpha_2a_1}{n}-\frac{a_1^2}{2n}
        \Big)dab
        +o(n^{-1})\;,   \label{eq:95a}
        \end{align}
 where $ D=\{(a_1,b_1):a_1>b_1\}$;
 \begin{align}
      &\Pr\left[a_1>b_1\land a_2>b_2\mid\cE_a\right]
        \nonumber\\
      & =\Big(1-\frac{\alpha_1}{n}+\frac{1}{n}\Big)
        \iint_Df(a_1,a_2,b_1,b_2)
        \left(1+\frac{\alpha_1}{n}+\frac{\alpha_2a}{n}-\frac{a^2}{2n}
        \right)\,dab
        +o(n^{-1})
        \nonumber\\
      & =\left(1+\frac{1}{n}\right)
        \iint_Df(a_1,a_2,b_1,b_2)
        \left(1+\frac{2\alpha_2a_1}{n}-\frac{a_1^2+a_1a_2}{n}
        \right) \,dab +o(n^{-1})\;,
        \label{eq:95b}
\end{align}
   where $ D=\{(a_1,a_2,b_1,b_2):a_1>b_1\land a_2>b_2\}$;
\begin{align}      &\Pr\left[a_1>b_1\land a_2>c_1\mid\cE_a\cap\cE_b\right]
        \nonumber\\
      &\qquad=\left(1-\frac{2\alpha_1}{n}+\frac{3}{2n}\right)
        \iint_Df(a_1,a_2,b,c)
        \Bigg(1+\frac{2\alpha_1}{n}+\frac{\alpha_2(a+b)}{n} \notag \\
        &\qquad \qquad\qquad\qquad\qquad\qquad \qquad\qquad -\frac{a^2+b^2}{2n}  \Bigg)\,dabc
        +o(n^{-1})
        \nonumber\\
      &\qquad=\left(1+\frac{3}{2n}\right)
        \iint_Df(a_1,a_2,b_1,c_1)
        \Bigg(1+\frac{\alpha_2(2a_1+b_1)}{n}  \notag \\
        &\qquad \qquad\qquad\qquad\qquad\qquad-\frac{2a_1^2+b_1^2+2a_1a_2}{2n}   \Bigg) \,dabc +o(n^{-1})\;,
        \label{eq:95c}
\end{align}
where $ D=\{(a_1,a_2,b_1,c_1):a_1>b_1\land a_2>c_1\}$.
\end{corollary}

We point out that an important feature of the expressions
\eqref{eq:95}--\eqref{eq:95c} is that the number of mixed $a_1a_2$ and $b_1b_2$
terms depends on the number of $\phi_{n-2}$ densities in the expression.

\paragraph{Proof of Lemma~\ref{lem:double-cov} assuming
  Lemma~\ref{lem:technical}}
We delay the proof of Lemma~\ref{lem:technical} and prove
Lemma~\ref{lem:double-cov} now. For this we need some elementary integral
computations. First, in the case with two variables $a$, $b$ and
$D_2 := \{(a,b):a>b\}$:
\begin{align}
  &\iint_{D_2} f(a,b)\,dab = \frac{1}{2}\;,
  \nonumber\\
  &\iint_{D_2}f(a,b)\cdot a\,dab
    =\int_{-\infty}^{+\infty}af(a)\int_{-\infty}^af(b)\,dbda
    =\EE\left[a\cdot F(a)\right] = A\;,
  \label{eq:119}\\
  &\iint_{D_2}f(a,b)\cdot a^2\,dab
    =\int_{-\infty}^{+\infty}a^2f(a)\int_{-\infty}^af(b)\,dbda
    =\EE\left[a^2\cdot F(a)\right]=B\;.
    \nonumber
\end{align}
In the four-variable case with
$D := \{(a_1,a_2,b_1,b_2):a_1>b_1 \land a_2>b_2\}$,
$f := f(a_1,a_2,b_1,b_2)$
and $dab = da_1da_2db_1db_2$:
\begin{align}
  &\iint_D f
    \, dab = \frac{1}{4} \; ,
  \nonumber\\
  &\iint_D f \cdot a_1 \, dab =
    \frac{1}{2}\iint_{D_2}f(a_1,b_1)\cdot a_1 \, dab=\frac{A}{2}\;,
  \nonumber\\
  &\iint_D f \cdot b_1 \, dab =
    \frac{1}{2}\int_{-\infty}^{+\infty} b_1f(b_1) \int_{a_1}^{+\infty} f(a_1)
    \, da_1db_1
  \nonumber\\
  &\qquad \qquad\qquad 
    = \frac{1}{2}\int_{-\infty}^{+\infty} b_1f(b_1)(1-F(b_1)) \, db_1
    = \frac{\EE[b_1]-\EE[b_1\cdot F(b_1)]}{2} = -\frac{A}{2} \; ,
  \label{eq:120}\\
  &\iint_D f \cdot a_1^2
    \, dab =
    \frac{1}{2} \iint_{D_2} f(a_1,b_1) \cdot a_1^2 \,dab
    = \frac{B}{2} \; ,
  \nonumber\\
  &\iint_D f \cdot b_1^2
    \, dab =
    \frac{1}{2} \int_{-\infty}^{+\infty} b_1^2f(b_1) \int_{b_1}^{+\infty}
    f(a_1)
    \, da_1db_1
    = \frac{\EE[b_1^2] - \EE[b_1^2\cdot F(b_1)]}{2}
  \nonumber\\
  &\qquad\qquad\qquad = \frac{1-B}{2} \; ,
  \nonumber\\
  &\iint_D f \cdot a_1a_2
    \, dab =
    \left(\iint_{D_2} f(a_1,b_1) \cdot a_1\,dab\right)^2
    =A^2\;,
  \nonumber\\
  &\iint_D f \cdot b_1b_2
      \, dab =
      \left( \int_{-\infty}^{+\infty} b_1f(b_1) \int_{b_1}^{+\infty}
      f(a_1)
      \, da_1db_1 \right)^2
    =\EE[b_1(1-F(b_1)]^2
  \nonumber\\
  &\qquad\qquad\qquad\quad= A^2 .
    \nonumber
\end{align}
Now all that is left is to insert the expressions computed above
into equations \eqref{eq:95}--\eqref{eq:95c} in
Corollary~\ref{cor:normalization}.
For example, in case of \eqref{eq:95c} we get
\begin{align}
  &\Pr\left[a_1>b_1\land a_2>c_1\mid\cE_a\cap\cE_b\right]
  \nonumber\\
  &\qquad=\left(1+\frac{3}{2n}\right)
    \iint_D f(a_1,a_2,b_1,c_1)
    \Bigg(1+\frac{\alpha_2(2a_1+b_1)}{n}
  \nonumber\\
   &\qquad \qquad\qquad\qquad\qquad\qquad\qquad -\frac{2a_1^2+b_1^2+2a_1a_2}{2n}\Bigg)
     \,dabc + o(n^{-1})
  \nonumber\\
  &\qquad=\left(1+\frac{3}{2n}\right)
    \left(\frac{1}{4}+\frac{\alpha_2 A}{n}-\frac{\alpha_2 A}{2n}
    -\frac{B}{2n}-\frac{1-B}{4n}-\frac{A^2}{n}
    \right)+o(n^{-1})
  \nonumber\\
  &\qquad=\frac{1}{4}+\frac{1}{8n}+\frac{\alpha_2 A}{2n}-\frac{B}{4n}
    -\frac{A^2}{n} + o(n^{-1}) \; ,
    \label{eq:118}
\end{align}
which is exactly~\eqref{eq:79c} that we wanted to prove.
Equations~\eqref{eq:79}--\eqref{eq:79b}
and~\eqref{eq:78c}--\eqref{eq:78b}
are handled in analogous ways 
and we provide the explicit computations only in the
\hyperref[sec:appendix]{Appendix}.
\qed

\paragraph{Proof of Lemma~\ref{lem:technical}}
Finally, we turn to Lemma~\ref{lem:technical}.
Let $\tilde \phi_j$ denote the density of $j^{-1/2}\sum_{i=1}^j a_i$.
Since the density of $\sum_{i=1}^k a_i$ is bounded for all $k$
(recall that $a_i$ has density continuous on closed support),
\cite[Theorem~15, pp.~206-7]{Petrov1975} implies
\ben{\label{eq:exp1}
  \tilde \phi_j(y)=\frac{1}{\sqrt{2\pi}} e^{-y^2/2} \left( 1 + \frac{\frac{\gamma_3}{3!}H_3(y)}{\sqrt{j}}-\frac{\frac{1}{2}\bclr{\frac{\gamma_3}{3!}}^2 H_6(y) + \frac{\gamma_4}{4!}H_4(y)}{j}\right) + \lito(j^{-1}),
}
where $\gamma_j$ denotes the $j$th cumulant, the error is uniform in $y\in\IR$, and the $H_j$ are Hermite polynomials:
\besn{\label{eq:herm}
H_3(y)&=y^3-3y, \\
H_4(y)&=y^4-6y+3, \\
H_6(y)&=y^6-15y^4+45 y^2-15.
}
Since we defined $\phi_j(x)=\sqrt{2\pi} \tilde \phi_j(x j^{-1/2})$,~\eq{eq:exp1} implies
\begin{align}
\phi_j(x) 
&= e^{-x^2/(2j)} 
  \left( 1 + \frac{\frac{\gamma_3}{3!}H_3(x j^{-1/2})}{\sqrt{j}}-\frac{\frac{1}{2}\bclr{\frac{\gamma_3}{3!}}^2 H_6(x j^{-1/2}) + \frac{\gamma_4}{4!}H_4(x j^{-1/2})}{j}\right) \notag \\
 &\qquad  +\lito(j^{-1})
  \nonumber\\
&=e^{-x^2/(2j)}
\left(1 + \frac{\alpha_1}{j} - \frac{\alpha_2 x}{j}\right)
+O\left(\frac{\max(|x|,x^6)}{j^{3/2}}\right)
  +o(j^{-1}) \; ,
  \label{eq:101}
\end{align}
where in the last line the additional remainder term comes from writing out the
Hermite polynomials~\eq{eq:herm} and then noting that what is left out of the
main term has smallest order terms $j^{3/2}$ in the denominator,
and largest order terms in the numerator $x$ or $x^6$, depending on
$|x|\leq 1$ or $|x|>1$.
Substituting this into the left-hand side of~\eqref{eq:93} and using the fact
that the sixth moment is finite, we get (letting $f := f(a_1,a_2,b_1,b_2)$)
\begin{align}
  &\iint_D f \cdot \phi_{n-2}(-a)\phi_{n-2}(-b) \, dab\nonumber\\
  &\qquad=\iint_D f\cdot
    \Bigg[\exp\left(-\frac{a^2}{2(n-2)}\right)
    \Big(1+\frac{\alpha_1}{n-2}+\frac{\alpha_2 a}{n-2}\nonumber\\
  &\qquad\qquad\qquad\qquad\qquad\qquad
    +O\left(\frac{\max(|a|,a^6)}{n^{3/2}}\right)+o(n^{-1})
    \Big)\Bigg]\nonumber\\
  &\qquad\qquad\qquad\cdot
    \Bigg[\exp\left(-\frac{b^2}{2(n-2)}\right)
    \Big(1+\frac{\alpha_1}{n-2}+\frac{\alpha_2 b}{n-2}\nonumber\\
  &\qquad\qquad\qquad\qquad\qquad\qquad 
    +O\left(\frac{\max(|b|,b^6)}{n^{3/2}}\right)+o(n^{-1})
    \Big)\Bigg]\,dab
    \nonumber\\
  &\qquad=\iint_D f \cdot \exp\left(-\frac{a^2+b^2}{2(n-2)}\right)
    \left(1+\frac{2\alpha_1}{n}+\frac{\alpha_2(a+b)}{n}\right)
    \, dab + o(n^{-1})\nonumber\\
  &\qquad=\iint_D f \cdot \left(1-\frac{a^2+b^2}{2(n-2)}
    +O\left(\min\left(
    \frac{a^2+b^2}{n},\frac{(a^2+b^2)^2}{n^2}
    \right)\right)\right)\nonumber\\
  &\qquad\qquad\qquad\cdot
    \left(1+\frac{2\alpha_1}{n}+\frac{\alpha_2(a+b)}{n}\right)
    \, dab + o(n^{-1})\nonumber\\
  &\qquad=\iint_D f\cdot
    \Bigg[1+\frac{2\alpha_1}{n}+\frac{\alpha_2(a+b)}{n}
    -\frac{a^2+b^2}{2n}\nonumber\\
  &\qquad\qquad\qquad
    +O\left(\min\left(
    \frac{a^2+b^2}{n},\frac{(a^2+b^2)^2}{n^2}
    \right)\right)
    \Bigg]\,dab+o(n^{-1})\;,
  \label{eq:99}
\end{align}
where we used the approximation $\exp(-x)=1-x+O(\max(x,x^2))$ for $x\ge 0$.

Inspecting~\eqref{eq:99}, we see that all that is left to establish \eqref{eq:93}
is to show
\begin{align}\label{eq:121}
  \iint_{\mathbb{R}^4} f\cdot\left(\min\left(
  \frac{a^2+b^2}{n},\frac{(a^2+b^2)^2}{n^2}
  \right)\right)\, dab = o(n^{-1})\;.
\end{align}
We do that by dividing the integration area into
two parts: 
\begin{center}$D_1 := \{(a_1, a_2, b_1, b_2): a^2+b^2 < n^{1/3}\}$
and $D_2 := \mathbb{R}^4 \setminus D_1$,\end{center}
and computing
\begin{align*}
  &
  \iint_{\mathbb{R}^4} f\cdot\left(\min\left(
  \frac{a^2+b^2}{n},\frac{(a^2+b^2)^2}{n^2}
    \right)\right)\, dab\\
  &\qquad\le
    \iint_{D_1} f \cdot    
    \left(\frac{(a^2+b^2)^2}{n^2}\right) \, dab
    +\iint_{D_2} f\cdot\left(\frac{a^2+b^2}{n}\right)\, dab
   =O(n^{-4/3})\;,
\end{align*}
where in the inequality we bound the minimum by one of the terms,
and then use the fact that small moments of $a$ and $b$ are finite,
and that on~$D_2$, we have $1 \leq (a^2+b^2)n^{-1/3}$, and therefore
$\frac{a^2+b^2}{n}\le\frac{(a^2+b^2)^2}{n^{4/3}}$.
Therefore, we have shown \eqref{eq:93}.
Similar calculations
concerning \eqref{eq:93a}--\eqref{eq:93c} are skipped here and provided
in the \hyperref[sec:appendix]{Appendix}.
Note that we always need at most sixth finite moment when estimating~\eqref{eq:101}.
\qed

\section{Stationary Gaussian dice}
\label{sec:stationary}

\subsection{Preparation}

Before we state and prove our results,
let us start with some useful facts about Gaussian Hilbert spaces.
It is a well-known fact that the Hermite polynomials $\{ H_k, k\geq 0\}$ are \emph{orthogonal} polynomials with respect to the standard Gaussian measure $\gamma(A) = \int_A \phi(x)\, dx$, for any Borel set  $A\subset\R$. Here $\phi$ is the standard Gaussian density function and $H_k$ can be defined via {\it Rodrigues' formula}:
  $
  H_k(x)  = (-1)^k \phi(x)^{-1} \frac{d^k}{dx^k}(\phi(x))$.
For any $f\in L^2(\R, \gamma)$, we have 
\[
f = \sum_{q\geq 0} \text{coef}(q) H_q \quad \text{with     $\text{coef}(q)  : =  \frac{1}{q!} \int_\R  H_q(x)f(x) \, \gamma(dx)  $} \,,
\]
 where the above series converges in $L^2(\R, \gamma)$; see \cite[Section 1.4]{bluebook}. In our work, we only need \eqref{ind_chaos} and \eqref{Phi-exp}.  We can find the expansion \eqref{ind_chaos}, for instance, in \cite[page 7]{MarWig2011}.  Suppose $Z\sim N(0,1)$, noting that $\E\big[ ( \1[Z>0] - 2^{-1})^2 \big] = 1/4$, we deduce from the orthogonality relation of Hermite polynomials that 
\begin{align} 
\frac{1}{4} = \sum_{k\geq 0} d_{2k+1}^2 (2k+1)!   \,, \label{dueTo}
\end{align}
from which together with the explicit expression of $d_{2k+1}$'s,  we can deduce one of \emph{Srinivasa  Ramanujan}'s ingenious identities (in a different form):
\begin{align}\label{c2k}
\pi  =  \sum_{k\geq 0 } \frac{1}{2^{2k-1}(2k+1)  } {\binom{2k}{k}} \,.
\end{align}
Ramanujan's identity reads as follows:
$$
\frac{\pi}{2} = 1 + \frac{1}{2} \left(\frac{1}{3} \right) + \frac{1\cdot 3}{2\cdot 4} \left(\frac{1}{5} \right) + \frac{1\cdot 3 \cdot 5}{2\cdot 4 \cdot 6} \left(\frac{1}{7}\right)  + \cdots\,;
$$ see \cite{Rama1912}.     

To obtain~\eqref{Phi-exp}, note that  $\Phi(x) = \E\big( \1[ - Z <  x] \big)$, then using the  expansion \eqref{ind_chaos}, we get
\begin{align*}
 \Phi(x) & = \E\big( \1[ Z/\sqrt{2} + x/\sqrt{2} > 0] \big) = \frac{1}{2} + \E\left(   \sum_{q\geq 0} d_{2q+1} H_{2q+1}\big( Z/\sqrt{2} + x/\sqrt{2} \big) \right) \,, \\
 & = \frac{1}{2} + \E\left(   \sum_{q\geq 0} d_{2q+1}   \sum_{k=0}^{2q+1} {\binom{2q+1}{k}} 2^{-q-\frac{1}{2}} H_k(Z) H_{2q+1-k}(x)                 \right)
\end{align*}
 where  we deduce the last equality from  the    well-known identity: for $a,b\in\R$ satisfying $a^2 + b^2 = 1$,
$
H_n(ax+by) = \sum_{k=0}^n {n \choose k} a^k b^{n-k} H_k(x) H_{n-k}(y)
$.
 Note that $\E[ H_k(Z) \big] = 0$ for any $k\geq 1$ and $\E[ H_0(Z)] = 1$.  Therefore, the expansion \eqref{Phi-exp} is established.

 \begin{remark} \label{obs11}  Newton's 1676 identity reads as follows: (see \cite[Page 228]{ArHae})
   $$ \frac{\pi}{6} = \arcsin(1/2) = \frac{1}{2} + \frac{1}{2}\cdot \frac{1}{3 \cdot 2^3} +  \frac{1\cdot 3}{2\cdot 4} \cdot\frac{1}{5 \cdot 2^5} +    \frac{1\cdot 3\cdot 5}{2\cdot 4\cdot 6} \cdot\frac{1}{7 \cdot 2^7} +\cdots  ~ ,$$
   which is equivalent to 
   \begin{align}\label{d2k}
     \pi = \sum_{q=0}^\infty \frac{3}{(2q+1) 2^{4q}} {2q \choose q}  \,.
   \end{align} 
   Using the explicit expression \eqref{Phi-exp} for $\ell_{2q+1}$
   and noting that $\Phi(G)$ for standard Gaussian $G$ has distribution
   that is uniform in $(0,1)$,
 we easily check that
 \begin{align}
   \label{eq:78}
   \frac{1}{6} = \sum_{q=0}^\infty (2q+1)! 2^{-2q} d_{2q+1}^2 \; ,
 \end{align}
 from which we have $\alpha = \frac{1}{6} - \frac{1}{2\pi} =   \sum_{q=1}^\infty (2q+1)! 2^{-2q} d_{2q+1}^2$.
 \end{remark}

\begin{lemma}\label{betan}    Suppose $X, Y$ are two centered (jointly) Gaussian random variables with mean zero and variance one such that $\E[ XY] = \rho $. Let $\Phi$ be the CDF of $X$, then, 
\begin{align*}
 \E\big[ \Phi(X) \Phi(Y) \big]  =  \frac{1}{4} +  \sum_{q\geq 0}  \ell_{2q+1}^2 (2q+1)! \rho^{2q+1}  
 =  \frac{1}{4}  + \frac{\rho}{4\pi} + O(\rho^3)
\end{align*}
where $\ell_{2q+1} = d_{2q+1} 2^{-q-\frac{1}{2}} = \dfrac{(-1)^q}{\sqrt{\pi}   (2q+1) 2^{2q+1} q!}$ for each integer $q\geq 0$.
\end{lemma}
\begin{proof}
Recall from \eqref{Phi-exp}      the expansion
  $
 \Phi = \frac{1}{2} + \sum_{q\geq 0} \ell_{2q+1} H_{2q+1} $. It is also known (see \emph{e.g.} Proposition 2.2.1 in \cite{bluebook}) that for $X,Y\sim N(0,1)$ jointly Gaussian and any integers $m,n\geq 0$,
\begin{align}\label{prop221}
\E\big[ H_m(X) H_n(Y) \big] = m! \big( \E[ XY] \big)^m ~\delta_{mn} \;.
\end{align}
Therefore,
\begin{align*}
& \E\big[ \Phi(X) \Phi(Y) \big] \\
 & = \frac{1}{4} +  \sum_{q\geq 0} \ell_{2q+1}^2 \E\big[ H_{2q+1}(X) H_{2q+1}(Y) \big] = \frac{1}{4} +  \sum_{q\geq 0}  \ell_{2q+1}^2 (2q+1)! \rho^{2q+1} \\
 & = \frac{1}{4}  + \frac{\rho}{4\pi}  + \frac{1}{\pi} \sum_{q\geq 1}   \frac{1}{ (2q+1) 2^{4q+2}  } {2q\choose q}   \rho^{2q+1}  \\
 &=  \frac{1}{4}  + \frac{\rho}{4\pi}   + O(\rho^3) \,,
\end{align*}
where the last big-$O$ estimate follows from  the Newton's identity \eqref{d2k}.
\end{proof}



\subsection{Our results}

Now we are in a position to present our results for  stationary Gaussian dice.
Recall from the introduction that   $\{G_i, i\in\mathbb{N}\}$ is a centered stationary Gaussian sequence with the correlation function $\rho$ such that $\rho(0)=1/2$.  Let
{\it$\bfa$}, {\it$\bfb$}, {\it$\bfc$} be
i.i.d.~copies of $\{G_1, \ldots, G_n\}$, then  for $i,j,k,\ell\in[n]$, $(a_i - b_j,   a_k - b_\ell)$ is centered bivariate Gaussian with $\Var\big( a_i - b_j \big) = \Var\big( a_k - b_\ell \big)=1$ and $\E\big[ (a_i - b_j )(a_k - b_\ell )\big] = \rho(i-k) + \rho(j-\ell)$.  Therefore,   we can compute the variance of $  W^{(ab)} : = \sum_{i,j\in[n]} \1[ a_i > b_j]$ using the expansion \eqref{ind_chaos} and the relation \eqref{prop221}:
\begin{align}
  \Var\left(   W^{(ab)} \right)
  &= \sum_{i,j,k,\ell\in[n]} \Big\{
    \E\big(\1[ a_i > b_j \land a_k > b_\ell] \big) - \frac{1}{4} \Big\} \notag \\
& = \sum_{i,j,k,\ell\in[n]} ~ \sum_{q\geq 0}d_{2q+1}^2 (2q+1)! \big( \rho(i-k) + \rho(j-\ell) \big)^{2q+1}  \notag \\
& = \sum_{q\geq 0}d_{2q+1}^2 (2q+1)!  \sum_{i,j,k,\ell\in[n]}  \big( \rho(i-k) + \rho(j-\ell) \big)^{2q+1}  \,. \label{VARH}
          \end{align} 
Let us first look at the almost trivial case where $\rho = s_{1/2}$, that is, when $\rho(i-k)  =\frac{1}{2} \delta_{ik}$.  In this case, we have by \eqref{eq:78},
\begin{align}
\Var\left(    W^{(ab)}  \right)   =   \left(\sum_{q\geq 0}d_{2q+1}^2 (2q+1)! \frac{1}{2^{2q}}\right) n^3 + O(n^2) = \frac{1}{6}n^3 +  O(n^2) \,.  \label{28hou}
 \end{align}
 Then,  by standard   computations and the above variance estimate, we have
 \[
 \Var\left( W^{(ab)}  - n  \sum_{i\in[n]} \big[ F(a_i) -  F(b_i)\big] \right) = O(n^2),
 \]
while due to the classical CLT,  $n^{-1/2}  \sum_{i\in[n]} \big[ F(a_i) -  F(b_i)\big] $ converges in law to $N(0, 1/6)$. Therefore, we can conclude that the CDF-ordering property \eqref{corderp} occurs with high probability in this setting.   This relation also implies the following more general result.

\begin{theorem} \label{thm41} Let $\mathbf{x} = (x_1, \ldots, x_n)$ be a sequence of i.i.d   random variables such that $x_1$ has a density function with a support
  which is a countable collection of (possibly infinite) intervals. 
 Assume $\mathbf{y}$  and $\mathbf{z}$ are two  i.i.d.~copies of $\mathbf{x}$, then with high probability, 
\[
\text{$\mathbf{x}$ beats $\mathbf{y}$ if and only if } \quad \sum_{i=1}^n \mathcal{F}(x_i) >  \sum_{i=1}^n \mathcal{F}(y_i)  \,,
\]
where  $\mathcal{F}$ is the distribution function (CDF) of $x_1$.
In particular, the probability that $\mathbf{x}, \mathbf{y}, \mathbf{z}$ are intransitive tends to zero, as $n\to+\infty$.
\end{theorem}

 \begin{proof} Let $\bfa,\bfb$ be given as in the case where $\rho = s_{1/2}$ and $F$ be the distribution function of $a_1\sim N(0,1/2)$, then  by integral transform, we can assume that
 \[
   \big\{ (x_i, y_i) : i\in\mathbb{N} \big\} =  \Big\{   \big(\mathcal{F}^{-1}\circ
   F(a_i),  \mathcal{F}^{-1}\circ F(b_i)\big) : i\in\mathbb{N} \Big\} \,,
 \]
 where $\mathcal{F}^{-1}(p) : = \inf\{ x\in\mathbb{R}: \mathcal{F}(x) \geq p \}$ is the generalized inverse of $\mathcal{F}$. It is clear that $F(a_i)\in(0,1)$ almost surely and due to our assumption on $\mathcal{F}$, we have $\mathcal{F}\circ \mathcal{F}^{-1}(p) = p$ for any $p\in(0,1)$. 
 It follows that
\begin{align*}
 &\quad  \sum_{i=1}^n \mathcal{F}(x_i)  >   \sum_{i=1}^n \mathcal{F}(y_i) \\
 & ~ \overset{\text{with  prob. 1}}{\Longleftrightarrow}  \quad  \sum_{i=1}^n F(a_i)  >   \sum_{i=1}^n F(b_i)  \\
  &   \overset{\text{with high prob.}}{\Longleftrightarrow}  ~  \sum_{i,j=1}^n \1[a_i > b_j] > \frac{n^2}{2}   \Leftrightarrow   \sum_{i,j=1}^n \1\big[ F(a_i) > F(b_j) \big] > \frac{n^2}{2}    \\
   & ~ \overset{\text{with  prob. 1}}{\Longleftrightarrow}  \quad \sum_{i,j=1}^n \1[\mathcal{F}^{-1}\circ F(a_i) > \mathcal{F}^{-1}\circ F(b_j) ] > \frac{n^2}{2}  \\
      &\qquad  \Longleftrightarrow \quad\quad  \sum_{i,j=1}^n \1[x_i > y_j ] > \frac{n^2}{2}  \, .
\end{align*}
Hence the desired conclusions follow immediately. 
 
 \end{proof}

 \medskip

 In the following, we provide the proof of  our Theorem \ref{SY-all} as well as some  results for the general stationary Gaussian dice. We first state two results of central importance to our approach.
  
\begin{theorem}[\cite{BM83}, Breuer-Major theorem] \label{BM-thm} Fix an integer $d\geq 1$. Assume $f\in L^2(\R, \gamma)$ admits the following expansion in $L^2(\gamma)$ ~ $($Recall $\gamma(dx)=\frac{1}{\sqrt{2\pi}}\exp(-x^2/2)dx)$:
\[
    f = \sum_{q=d}^\infty \text{\rm coef}(q) H_q \quad \text{with $\text{\rm coef}(d) \neq 0$;  \quad $d$ is called the Hermite rank of $f$.}
\]
Assume also that $(X_k, k\in\mathbb{Z})$ is a centered stationary Gaussian sequence with {\bf unit} variance\footnote{That is, $\widetilde{\rho}(0) = 1$, which is different from $\rho(0) =1/2$.} such that its correlation function $\widetilde{\rho}$ belongs to $\ell^d(\mathbb{Z})$, where $\widetilde{\rho}(i-j) = \E[ X_i X_j ]$ for any $i,j\in\mathbb{Z}$.  

Then 
\[
\frac{1}{\sqrt{n}} \sum_{k=1}^n f(X_k) ~\text{converges in law to }~ N(0, \sigma^2) ~\text{as $n\to+\infty$}\,,
\]
where 
${\displaystyle
\sigma^2: = \sum_{q=d}^\infty q! \text{\rm coef}(q)^2 \sum_{v\in\mathbb{Z}} \widetilde{\rho}(v)^q \in [0, +\infty)
}$ is part of the conclusion.
\end{theorem}
For a modern proof using fourth moment theorems, one can refer to \emph{e.g.,} Theorem 7.2.4 in \cite{bluebook}. In particular, we also need one ingredient from this proof, which we state in the following. 

\begin{lemma}\label{eq727} Let the assumptions of Theorem \ref{BM-thm} be satisfied, that is, $\wt{\rho}\in\ell^d(\mathbb{Z})$. For any integer $q\geq d\vee 2$, and any $r\in\{1, \ldots, q-1\}$, we have 
\[
    n^{-1 + \frac{r}{q}} \sum_{\vert j\vert < n} \vert \widetilde{\rho}(j) \vert^r = o(1) \quad \text{\rm as $n\to+\infty$; \quad see  equation (7.2.7) in \cite{bluebook}.}
\]

\end{lemma}

 \paragraph{Proof of Theorem \ref{SY-all}.}  Note that we have proved the case where $H =1/2$.  Our proof then consists of only two parts: in the first part, we prove our result for  $H\in(1/2, 1)$ and in the second part, we prove a stronger result (Theorem \ref{Lygon1}) that includes  the case $H\in(0,1/2)$.

  We proceed in the same way as in previous subsection: we first estimate the variance of the difference 
 $W^{(ab)}  - n  \sum_{i=1}^n \big[ F(a_i) -  F(b_i)\big]$, then prove a CLT for $W^{(ab)}$. We begin with the following two lemmas dealing with two  variance estimates.

\begin{lemma}\label{Lui2} Let ${\it\bfa}, {\it\bfb},{\it\bfc}$ be i.i.d.~copies of the centered stationary Gaussian sequence  $\{ G_i, i\in\mathbb{N} \}$ with the  correlation function $\rho$ such that $\rho(0)=1/2$. Then 
 \begin{align}
  &   \Var\Big(  W^{(ab)}  - n  \sum_{i=1}^n \big[ F(a_i) -  F(b_i)\big]  \Big)  =   \frac{1}{3} \Var\big( W^{(ab)} +  W^{(bc)} +  W^{(ca)} \big)    \label{parkville1}  \\
  & = \sum_{q\geq 1} d_{2q+1}^2 (2q+1)! \sum_{v=1}^{2q} {2q+1 \choose v}  \left(  \sum_{\vert i\vert < n}  (n - \vert i\vert )\rho(i)^v \right)   \notag \\
   &\qquad\qquad\qquad\qquad \times \left(   \sum_{\vert j\vert < n}  (n - \vert j\vert )\rho(j)^{2q+1-v} \right)  .  \label{parkville2}  
  \end{align}
\begin{itemize}
\item[$(1)$] If $\rho\in\ell^3(\mathbb{Z})$,  then {\rm
\[
\Var\left(  W^{(ab)}  - n  \sum_{i=1}^n \big[ F(a_i) -  F(b_i)\big]  \right)  =  o(n^3) ~.
 \]
}

\item[$(2)$] Consider $\rho = s_H$, then the case $H\in(0, 5/6)$ is covered by point $(1)$; if $H\in[ 5/6, 1)$, we have 
\[
 \Var\Big(  W^{(ab)}  - n  \sum_{i=1}^n \big[ F(a_i) -  F(b_i)\big]  \Big)   \sim \frac{H^2(2H-1)}{16\pi (4H-3)} n^{6H-2} \,.
\]

\end{itemize}

\end{lemma}

The proofs of the above lemma and  the following lemma will be postponed to the end of this section.

 \begin{lemma}\label{Belval} Let ${\it\bfa}, {\it\bfb}$ and  $\{ G_i, i\in\mathbb{N} \}$ be given as in Lemma \ref{Lui2}. The following statements hold true.
 \begin{itemize}
 \item[\rm (1)] If $\rho\in\ell^1(\mathbb{Z})$,  then,  with $\beta: =  2  \sum_{q\geq 0}d_{2q+1}^2 (2q+1)! \sum_{i\in\mathbb{Z}} \rho(i)^{2q+1}   \in[0,+\infty)$,
 \[
  \Var\big( W^{(ab)} \big) =  \beta n^3 + o(n^3)  \,.
 \]
 
 \item[\rm (2)]  Consider the case where $\rho = s_H$  is given as in \eqref{s_H}   : 
 
 \begin{enumerate}
 
 \item[\rm (i)] for $H\in(0,1/2]$,   $\Var\big( W^{(ab)} \big) = \beta n^3 + o(n^3)$ with $\beta$ defined as in point {\rm (1)}; moreover, $\beta > 0$ in this case.

 \item[\rm (ii)] for $H\in(1/2, 1)$,  $\Var\big( W^{(ab)} \big) = \dfrac{1}{2\pi}n^{2H+2} + o(n^{2H+2})$.
 
 \end{enumerate}
 
 \end{itemize}
  \end{lemma}
 
 \medskip
 
\noindent{\it Assuming Lemma \ref{Lui2} and \ref{Belval}}, we prove Theorem \ref{SY-all} in the following.  As announced, we split our proof into two cases.
 
 \medskip
 
 \noindent{\it case 1: $\underline{H\in(1/2, 1)}$}. In this case, we deduce from the above two lemmas that 
 \begin{align} \label{2lemmas}
  \Var\Big(  W^{(ab)}  - n  \sum_{i=1}^n \big[ F(a_i) -  F(b_i)\big]  \Big)\big/  \Var\big( W^{(ab)} \big) = o(1) \,.
 \end{align}
 And we have, with $\ell_0 = \frac{1}{2\sqrt{\pi}}  $ (see \eqref{Phi-exp})
 \[
   \sum_{i\in[n]} \Big(F(a_i) - \frac{1}{2} \Big) = \sum_{i\in[n]}
   \Big( F(a_i) - \frac{1}{2} - \ell_0 \sqrt{2} a_i  \Big) +  \sqrt{2} \ell_0 \sum_{i\in[n]} a_i
 \]
 and it is clear that  the second part in the above sum is a centered Gaussian with
\[
  \Var\left( \sqrt{2} \ell_0 \sum_{i=1}^n  a_i   \right)= \frac{1}{2\pi} \sum_{i,j=1}^n s_H(i-j) \sim \frac{1}{4\pi} n^{2H} \,, \quad \text{as $n\to+\infty$,}
\]where the asymptotic behavior is implied by \eqref{cor-est}.       We know from \eqref{2lemmas} and point (ii) in Lemma \ref{Belval} that 
 \begin{align}\label{eq:use}
\Var\left(\sum_{i=1}^n \big[ F(a_i) -  F(b_i)\big] \right) \Big/    \Var\left( \sqrt{2} \ell_0 \sum_{i=1}^n  (a_i - b_i)  \right) \xrightarrow{n\to\infty} 1 \,.
 \end{align}
Recall the Slutsky's lemma, which says 
\begin{center} 
if  $X_n \xrightarrow[n\to\infty]{\rm law} X$ and $Y_n \xrightarrow[n\to\infty]{\rm law} 0$, then $X_n+Y_n \xrightarrow[n\to\infty]{\rm law} X$.
\end{center}
Thus, we deduce from  \eqref{eq:use} and the orthogonality property of Hermite polynomials,  
$n^{-H}  \sum_{i=1}^n \big[ F(a_i)\allowbreak-  F(b_i)\big] $ converges in law to $N\big(0, \frac{1}{2\pi}\big)$, as $n\to+\infty$. 
Combining~\eqref{2lemmas} with Slutsky's lemma again yields
\[
   \frac{1}{n^{H+1}} \Big( W^{(ab)}  -\frac{n^2}{2}  \Big)    \xrightarrow[n\to+\infty]{\rm law} N\big(0, \frac{1}{2\pi} \big) \, .
\]
Hence the desired conclusions follow from  similar  arguments as in the proof of Theorem   \ref{con:weak}.
           For the sake of completeness, we sketch it below:  first we define  
$
V_n = n \sum_{i=1}^n \big( F(a_i) - F(b_i) \big) 
$, then we have for any $\delta > 0$, 
\begin{align*}
&\mathbb{P}\left\{  \sgn(V_n) \neq  \sgn\big(W^{(ab)} -\frac{n^2}{2} \big)  \right\} \\
&\quad \leq \mathbb{P}\left\{ \Big\vert  \frac{W^{(ab)} -\frac{n^2}{2} - V_n }{n^{H+1}}   \Big\vert    > \delta    \right\}  + \mathbb{P}\left\{    \Big\vert  \frac{W^{(ab)}  -\frac{n^2}{2}  }{n^{H+1}}   \Big\vert    \leq \delta   \right\} ,
\end{align*}
 where the $\limsup$ of the RHS, as $n\to+\infty$, is bounded by $2\delta$. This implies that for $H\in(1/2, 1)$, the relation \eqref{corderp} occurs with high probability and thus, the probability of $a,b,c$ being intransitive asymptotically vanishes.

  \bigskip
 
 \noindent{\it  case 2: $\underline{H\in(0, 1/2)}$}.   In this case, the correlation function $s_H\in\ell^1(\mathbb{Z})$ and by  Lemma \ref{Belval}, 
 $
 \beta = 2  \sum_{q\geq 0}d_{2q+1}^2 (2q+1)! \sum_{i\in\mathbb{Z}} s_H(i)^{2q+1}  \in(0, +\infty) \,.
 $
 Then, \emph{case 2} is an immediate consequence of the following theorem. 
 
 \begin{theorem}\label{Lygon1}  Let ${\it\bfa}, {\it\bfb}, {\it\bfc}$ be i.i.d.~copies of $\{ G_1, \ldots, G_n\}$ with correlation function $\rho\in\ell^1(\mathbb{Z})$ such that the constant $\beta$ defined in Lemma \ref{Belval} is {\bf strictly} positive.   Then, with high probability,
   \begin{align}\label{corderp-rho}
  \sum_{i,j=1}^n \1[ a_i > b_j] > \frac{n^2}{2} \quad \text{if and only if } \quad \sum_{i=1}^n F(a_i) > \sum_{i=1}^n  F(b_i)  \, ,
\end{align}
where $F(x)= \Phi(\sqrt{2} x)$ is the distributional function of $G_1\sim N(0, 1/2)$.  As a consequence,  the probability of three dice ${\it\bfa}, {\it\bfb},{\it\bfc}$ being intransitive tends to zero, as $n\to+\infty$ .

 \end{theorem}

 \begin{proof}[Proof of Theorem \ref{Lygon1}]
 
 Let us first summarize what we have so far, concerning this proof:
 \begin{itemize}
 \item $\Var\big( W^{(ab)} \big) = \beta n^3 + o(n^3)$, with $\beta\in(0,+\infty)$; see Lemma \ref{Belval}.
 
 \item $ {\displaystyle  \Var\left(  W^{(ab)}  - n  \sum_{i=1}^n \big[ F(a_i) -  F(b_i)\big]  \right) = o(n^3) }$; see Lemma \ref{Lui2}.

 \end{itemize}
 Putting $X_i = \sqrt{2} a_i$ for each $i\in\mathbb{N}$ and $\widetilde{\rho} = 2\rho$, we apply  Theorem \ref{BM-thm} for $d=1$, $f= \Phi -1/2 = \sum_{q\geq 0} \ell_{2q+1}H_{2q+1}$ and we obtain the following CLT:
 \[
\frac{1}{\sqrt{n}} \sum_{k=1}^n\big( F(a_k) - \frac{1}{2} \big) =\frac{1}{\sqrt{n}} \sum_{k=1}^n f(X_k) \xrightarrow[n\to+\infty]{\rm law} N(0, \beta/2) \, ,
\] 
where the limiting variance, due to Breuer-Major's theorem, should be  
$$
\sum_{q= 0}^\infty (2q+1)! \ell^2_{2q+1} \sum_{v\in\mathbb{Z}}  (2\rho(v))^{2q+1} \, ,
$$
which is indeed equal to $\beta/2$ because of $d_{2q+1}^2 = \ell_{2q+1}^22^{2q+1}$ for each integer $q\geq 0$.

\medskip

Thus, we deduce from the above CLT and Slutsky's lemma that 
\[
\frac{1}{\sqrt{n}} \sum_{i=1}^n \big[ F(a_i) -  F(b_i)\big] \xrightarrow[n\to+\infty]{\rm law} N(0, \beta) \quad \text{and} \quad   \frac{W^{(ab)}  -\frac{n^2}{2}  }{n^{3/2}}   \xrightarrow[n\to+\infty]{\rm law} N(0, \beta) \, .
\]
Hence the desired conclusions follow from  the same arguments as in the ending paragraph of {\it case 1}.  

 \end{proof}
 
 To conclude this section, it remains to prove Lemmas \ref{Lui2} and \ref{Belval}.  One may have noticed that we haven't used the relation \eqref{parkville1} in the above proofs. In fact, the relation \eqref{parkville1} and the following Lemma \ref{Esch} together imply the point (1) in Lemma \ref{Lui2}, and besides the independent interest of such a relation, its proof contains some ingredients for our proof of Lemma \ref{Lui2}.

\begin{lemma}\label{Esch} Let ${\it\bfa}, {\it\bfb}, {\it\bfc}$ be i.i.d.~copies of $\{ G_i, i\in\mathbb{N} \}$. Assume that $\rho\in\ell^3(\mathbb{Z})$,  then  {\rm
\begin{align}
\Var\Big(  W^{(ab)} +  W^{(bc)} +  W^{(ca)} \Big) = o(n^3) \,. \label{cfB2}
\end{align}
}
\end{lemma}
  \begin{proof} Using Hermite expansion of $x\in\R\lmto \1[x>0]$, we have 
 \begin{align*}
 &\quad W^{(ab)} +  W^{(bc)} +  W^{(ca)}  =  \sum_{i,j=1}^n \big( \1[a_i  > b_j] + \1[b_i  > c_j]  + \1[c_i  > a_j]  \big)  \\
 &=  \frac{3n^2}{2} +  \sum_{q\geq 0} d_{2q+1} \sum_{i,j=1}^n\Big[H_{2q+1}(a_i-b_j) + H_{2q+1}(b_i-c_j)  + H_{2q+1}(c_i-a_j)  \Big]
  \end{align*}
 so that
 \begin{align}
 &    \Var\Big(  W^{(ab)} +  W^{(bc)} +  W^{(ca)} \Big)  = \sum_{q\geq 0} d_{2q+1}^2 (2q+1)!  \times  \notag\\
 &   \sum_{i,j,k,\ell=1}^n \Bigg(  \E[ (a_i - b_j)(a_k - b_\ell) ]^{2q+1} +  \E[ (a_i - b_j)(b_k - c_\ell) ]^{2q+1} \notag \\
  &   +  \E[ (a_i - b_j)(c_k - a_\ell) ]^{2q+1}  +  \E[ (b_i - c_j)(a_k - b_\ell) ]^{2q+1} + \E[ (b_i - c_j)(b_k - c_\ell) ]^{2q+1}  \notag \\
 &    +  \E[ (b_i - c_j)(c_k - a_\ell) ]^{2q+1}   +  \E[ (c_i - a_j)(a_k - b_\ell) ]^{2q+1} \notag \\
 &   +  \E[ (c_i - a_j)(b_k - c_\ell) ]^{2q+1} +  \E[ (c_i - a_j)(c_k - a_\ell) ]^{2q+1} \Bigg) \,.  \notag 
 \end{align}
 Then, using the specific correlation structure of $a, b, c$ as well as their independence, we get 
  \begin{align}
 & \quad \frac{1}{3}\Var\Big(  W^{(ab)} +  W^{(bc)} +  W^{(ca)} \Big) \notag   \\
 &  = \sum_{q\geq 1} d_{2q+1}^2 (2q+1)!   \sum_{i,j,k,\ell=1}^n \Big[   \big( \rho(i-k) + \rho(j-\ell) \big)^{2q+1} -  \rho(i-k)^{2q+1} \notag \\
 &\qquad\qquad\qquad\qquad\qquad\qquad\qquad\qquad -  \rho(j-\ell)^{2q+1} \Big]  \label{sum-here}  \,.
         \end{align}
 Let us now look at the second sum in  \eqref{sum-here}, which can be rewritten using the binomial formula, as follows:
\begin{align}
&\quad \sum_{i,j,k,\ell=1}^n \sum_{v=1}^{2q} {2q+1 \choose v}   \rho(i-k)^{v}   \rho(j-\ell)^{2q+1-v} \notag   \\
& = \sum_{v=1}^{2q} {2q+1 \choose v}  \left(  \sum_{i,k=1}^n   \rho(i-k)^{v}   \right)\left(  \sum_{j,\ell=1}^n   \rho(j-\ell)^{2q+1-v}   \right) \label{label2} \\
& =  \sum_{v=1}^{2q} {2q+1 \choose v}  2^{-1-2q}   \left(  \sum_{\vert i\vert < n} (n- \vert i\vert )   \widetilde{\rho}(i)^{v}   \right)    \left(  \sum_{\vert j\vert < n} (n- \vert j\vert )   \widetilde{\rho}(j)^{2q+1-v}   \right) \notag
\end{align}
by putting $\wt{\rho} = 2 \rho$. It is clear that  the  term  $2^{-1-2q}$ will     compensate the term  $ \sum_{v=1}^{2q} {2q+1 \choose v} $ above. Therefore, we only need  the following rough estimate:  for $q\geq 1$
\begin{align*}
  \sum_{i,j,k,\ell=1}^n & \Big[   \big( \rho(i-k) + \rho(j-\ell) \big)^{2q+1} -  \rho(i-k)^{2q+1} -  \rho(j-\ell)^{2q+1} \Big]  \\
  &=  O\left\{  n^2  \left(  \sum_{\vert i\vert < n}    \vert\widetilde{\rho}(i)\vert   \right)      \left(  \sum_{\vert i\vert < n}    \vert\widetilde{\rho}(i)\vert^2   \right)    \right\} \,,
\end{align*}
implying 
\[
\Var\Big(  W^{(ab)} +  W^{(bc)} +  W^{(ca)} \Big) = O\left\{  n^2  \left(  \sum_{\vert i\vert < n}    \vert\widetilde{\rho}(i)\vert   \right)      \left(  \sum_{\vert i\vert < n}    \vert\widetilde{\rho}(i)\vert^2   \right)    \right\}  \,.
\]
The desired estimate \eqref{cfB2} follows from Lemma \ref{eq727} and the assumption  $\widetilde{\rho}\in\ell^3(\mathbb{Z})$.   

 \end{proof}

 \medskip

\begin{proof}[Proof of Lemma \ref{Lui2}]

  As in previous variance calculations,
  we have 
 \begin{align}
&\quad \Var\left(  W^{(ab)}  - n  \sum_{i=1}^n \big[ F(a_i) -  F(b_i)\big]  \right)\notag\\
&= \Var\big(  W^{(ab)} \big) +  \frac{1}{2} n^4 - 2n^2 \sum_{i,j=1}^n \E\big[ F(a_i) F(a_j) \big] \notag \\
&=  \Var\big(  W^{(ab)} \big) +  \frac{1}{2} n^4 - 2n^2  \left( ~ \frac{1}{3}n + 2 \sum_{1\leq  i < j \leq n}  \E\big[ F(a_i) F(a_j)\big]    \right) \, . \label{label0}
 \end{align}
 
 It follows from Lemma \ref{betan}  that for $i\neq j$, \quad (also due to $d_{2q+1}^2 = \ell_{2q+1}^2 2^{2q+1}$)
 \begin{align} 
  \E\big[ F(a_i) F(a_j)\big]    &=    \E\big[ \Phi(\sqrt{2}a_i) \Phi(\sqrt{2}a_j)  \big]  \notag  \\
  &=   \frac{1}{4} + \sum_{q\geq 0} d_{2q+1}^2 (2q+1)!  \rho(i-j)^{2q+1} \,.  \label{label1}
            \end{align}
 Therefore, it is routine to verify using \eqref{VARH}, \eqref{label0}, \eqref{label1},\eqref{label2} and \eqref{sum-here} that
 \begin{align}
  &\quad \Var\left(  W^{(ab)}  - n  \sum_{i=1}^n \big[ F(a_i) -  F(b_i)\big]  \right)  \label{yoyoq}\\
  & = \sum_{q\geq 1} d_{2q+1}^2 (2q+1)! \sum_{v=1}^{2q} {2q+1 \choose v}  \left(  \sum_{\vert i\vert < n}  (n - \vert i\vert )\rho(i)^v \right)  \notag \\
  &\qquad  \times \left(   \sum_{\vert j\vert < n}  (n - \vert j\vert )\rho(j)^{2q+1-v} \right)     =   \frac{1}{3} \Var\big( W^{(ab)} +  W^{(bc)} +  W^{(ca)} \big)  \,. \notag
  \end{align}
Therefore, the relations \eqref{parkville1} and  \eqref{parkville2} are established.  If $\rho\in\ell^3(\mathbb{Z})$, Lemma \ref{Esch} implies that the variance in \eqref{yoyoq} is $o(n^3)$.

\medskip

To prove point (2), we consider the particular case where $\rho = s_H$. One can easily verify using the asymptotic relation \eqref{cor-est} that $s_H\in\ell^3(\mathbb{Z})$ if and only if $H\in(0, 5/6)$.  Now suppose that $H\in[ 5/6, 1)$, the relation \eqref{parkville2} still holds true, that is, we have  
 \begin{align*}
& \quad \Var\left(  W^{(ab)}  - n  \sum_{i=1}^n \big[ F(a_i) -  F(b_i)\big]  \right)  \\
& =  \frac{1}{2\pi}  \left(  \sum_{\vert i\vert < n}  (n - \vert i\vert )s_H(i) \right)   \left(   \sum_{\vert j\vert < n}  (n - \vert j\vert )s_H(j)^{2} \right)  \\
&\qquad +  \sum_{q\geq 2} d_{2q+1}^2 (2q+1)! \sum_{v=1}^{2q} {2q+1 \choose v}  \left(  \sum_{\vert i\vert < n}  (n - \vert i\vert )s_H(i)^v \right) \\ 
&\qquad\qquad\qquad\qquad\qquad\times \left(   \sum_{\vert j\vert < n}  (n - \vert j\vert )s_H(j)^{2q+1-v} \right) \,.  
  \end{align*}
One can readily check using  \eqref{cor-est} that for $H\in[5/6, 1)$, 
 \[
  \sum_{\vert i \vert < n}  \big(n - \vert i\vert \big) s_H(i)  \sim \frac{1}{2}n^{2H} \,
 \quad
 \text{and }
 \quad
  \sum_{\vert i \vert < n}  \big(n - \vert i\vert \big) s_H(i)^2 \sim \frac{H^2(2H-1)}{4(4H-3)} n^{4H-2}  \,,
 \]
 and 
 \begin{align*}
  \sum_{\vert i \vert < n}  \big(n - \vert i\vert \big) s_H(i)^3
   \sim
    \begin{cases}
   \dfrac{H^3(2H-1)^3}{8(6H-5)(3H-2)} n^{6H-4}  \quad & \text{if $H\in(5/6, 1)$} \\
   \quad\\
    2 (5/18)^3  n \log n     & \text{if $H=5/6$.}
   \end{cases}
 \end{align*}
All these estimates  imply, whenever $H\in[ 5/6, 1)$,
 \begin{align*}
 \Var\Big(  W^{(ab)}  - n  \sum_{i=1}^n \big[ F(a_i) -  F(b_i)\big]  \Big)&  =   \frac{1}{3} \Var\big( W^{(ab)} +  W^{(bc)} +  W^{(ca)} \big) \\
 & \sim \frac{H^2(2H-1)}{16\pi (4H-3)} n^{6H-2} \,.
 \end{align*}
Hence the proof of Lemma \ref{Lui2} is completed.

 \end{proof}

 \bigskip

 \begin{proof}[Proof of Lemma \ref{Belval}]  Assume first that $\rho\in\ell^1(\mathbb{Z})$ and  recall from \eqref{VARH} that
 \[
 \Var\big( W^{(ab)} \big) =  \sum_{q\geq 0}d_{2q+1}^2 (2q+1)!  \sum_{i,j,k,\ell =1}^n  \big( \rho(i-k) + \rho(j-\ell) \big)^{2q+1}  
 \]
 and in view of \eqref{sum-here}, we have
 \begin{align}
 \Var\big( W^{(ab)} \big)  = 2 \sum_{q\geq 0}d_{2q+1}^2 (2q+1)!  \sum_{i,j,k,\ell = 1}^n  \rho(i-k)^{2q+1}  +o(n^3) \, . \label{suf-est}
 \end{align}
 The second sum in \eqref{suf-est} is equal to $n^2 \sum_{\vert i \vert < n} (n-\vert i\vert) \rho(i)^{2q+1} $. Since $\rho\in\ell^1(\mathbb{Z}) $ and for $q\geq 0$,
 \[
\lim_{n\to+\infty} \sum_{\vert i \vert < n} \frac{ n-\vert i\vert}{n} \rho(i)^{2q+1}  = \sum_{i\in\mathbb{Z}} \rho(i)^{2q+1} \quad\text{by \it dominated convergence.}
 \]
Therefore, as $n\to+\infty$,
\begin{align*}
& \quad n^{-3} \sum_{q\geq 0}d_{2q+1}^2 (2q+1)!  \sum_{i,j,k,\ell = 1}^n  \rho(i-k)^{2q+1} \\
&=  \sum_{q\geq 0}d_{2q+1}^2 (2q+1)!  \sum_{\vert i \vert < n} \frac{ n-\vert i\vert}{n} \rho(i)^{2q+1}\to  \sum_{q\geq 0}d_{2q+1}^2 (2q+1)! \sum_{i\in\mathbb{Z}} \rho(i)^{2q+1},
\end{align*}
so that  
$
 \Var\big( W^{(ab)} \big) =  \beta n^3 + o(n^3)
$.    
Note that $\beta\in[0,+\infty)$ under  the assumption $\rho\in\ell^1(\mathbb{Z})$ is an easy consequence of Theorem \ref{BM-thm}. It is clear that $\widetilde{\rho} = 2\rho$ satisfies the assumption of Theorem \ref{BM-thm}, then using $d_{2q+1}^2 = \ell_{2q+1}^2 2^{2q+1}$, we get
\[
\frac{1}{2} \beta =  \sum_{q\geq 0}d_{2q+1}^2 (2q+1)! 2^{-1-2q} \sum_{i\in\mathbb{Z}}  \widetilde{\rho}(i)^{2q+1} =  \sum_{q\geq 0}\ell_{2q+1}^2 (2q+1)!   \sum_{i\in\mathbb{Z}}  \widetilde{\rho}(i)^{2q+1}  \,.
\]
So, with $f(x) = \Phi(x) - \frac{1}{2}$ and $d=1$, one can see that $\beta\in[0,+\infty)$. 

\bigskip

Now let us look at the fractional case, and note   that the case $H=1/2$ was stated in \eqref{28hou}.

 If $H < 1/2$, then $s_H$  is summable so that  ${\displaystyle \sum_{i\in\mathbb{Z}} s_H(i)}$ is finite, which is the limit of  
\[
  \sum_{\vert k\vert \leq n}  s_H(k) = \frac{1}{4}   \sum_{\vert k\vert \leq n} \big( \vert k+1\vert^{2H}  + \vert k-1\vert^{2H}- 2 \vert k \vert^{2H} \big) = \frac{1}{2}   \big( \vert n+1\vert^{2H} - \vert n \vert^{2H} \big)
\]
as $n\to+\infty$. This limit is zero.  For later reference, we summarize some basic properties of $s_H$ for $H\in(0, 1/2)$:
\begin{align} \label{centerf}
\quad \text{ $s_H(0) = \dfrac{1}{2}$,  $  - \dfrac{1}{2} < s_H(v) < 0$ for $v\neq 0$;  and ${\displaystyle \sum_{v\in\mathbb{Z}} s_H(v) = 0}$.}
\end{align}
It follows that for $q\geq 1$, from
$$
   1 = 2s_H(0) =  \sum_{v\neq 0} \big[  -2s_H(v)  \big] >  \sum_{v\neq 0} \big[  -2s_H(v)  \big]^{2q+1}
$$
we obtain
 $\sum_{i\in\mathbb{Z}} s_H(i)^{2q+1} \in (0,+\infty)$.
Thus,   point (2)-(i) is proved.

\medskip
  
 If $H\in(1/2, 1)$, then $s_H(v) > 0$. One can verify by using \eqref{label0}, \eqref{label1} and the fact $1/6 = \sum_{q=0}^\infty   d_{2q+1}^2 (2q+1)! 2^{-2q}$ from Remark \ref{obs11} that
    \begin{align*}
 \Var\big( W^{(ab)} \big)  & = \Var\left(  W^{(ab)}  - n  \sum_{i=1}^n \big[ F(a_i) -  F(b_i)\big]  \right)   \\
 &\qquad\qquad\qquad + 2n^2 \sum_{q\geq 0}     d_{2q+1}^2 (2q+1)!  \sum_{i,j\in[n]  }   s_H(i-j)^{2q+1}  \,.
  \end{align*}
 The first term in the above sum  is of order $o(n^{2H+2})$, by Lemma \ref{Lui2}.     It remains to use \eqref{cor-est} to estimate the second term in the above sum:   
\[
 2n^2d_0^2 \sum_{i,j\in[n]  }   s_H(i-j) \sim \frac{1}{2\pi} n^{2H+2}
 \] 
 gives the dominant contribution. Hence     our  proof of  Lemma  \ref{Belval} is now completed. 

\end{proof}

\section{Condorcet paradox for close elections: Majority}
\label{sec:condorcet}

This section contains the proof of Theorem~\ref{thm:elections-random}.

\subsection{Notation}
We start with recalling and extending the model and notation.
There are $n$ voters (where $n$ is odd) and each of them independently
chooses one of $k!$ rankings of the alternatives uniformly at random.
For voter $i$, such a random ranking gives rise to a random tuple
$x_i =  (x_i^{(1)}, \ldots, x_i^{(K)} )$
in $\{-1, 1\}^{K}$ representing
$K := \binom{k}{2}$
pairwise choices (according to some fixed ordering of pairs).
We call each of $k!$ tuples in the support of $x_i$ \emph{transitive}.
Any other tuple is \emph{intransitive}. We say that a tuple has
a \emph{Condorcet winner} if it has an alternative that beats everyone else.

We denote aggregation over voters by boldface. Therefore, we write
$\mathbf{x} = (x_1, \ldots x_n)$ for the random vector of voter
preferences (where each element is itself a random tuple of length $K$).

For $j=1,\ldots,K$,  let $S_i^{(j)} := \sum_{i'=1}^i x^{(j)}_{i'}$ and $S^{(j)} := S_n^{(j)}$,
and write
\begin{align*}
  Y^{(j)} = \Maj_n(\mathbf{x}^{(j)}) = \sgn(S^{(j)}) \; . 
\end{align*}
Furthermore, we write $Y = \left(Y^{(1)}, \ldots, Y^{(K)}\right)$
and $S = \left(S^{(1)}, \ldots, S^{(K)}\right)$ for the aggregated tuples.

Given voter preferences, we say that the voting
outcome is intransitive if the aggregated tuple $Y$
is intransitive. Similarly, we say that there is a Condorcet winner if
tuple $Y$ has a Condorcet winner.

We are interested in situations where elections are ``almost tied''
or, more precisely, ``$d$-close'' for $d \ge 1$.
Specifically, we define $\cE_d$ to be the event
where $\|S\|_{\infty} \le d$, i.e.,
$|S^{(j)}|$ is at most $d$ for every $j \in [K]$.

\subsection{Local CLT}

We use a theorem and some definitions from the textbook on random
walks by Spitzer \cite{Spi76}. In accordance with the book, we make
\begin{definition}
  A $k$-dimensional random walk $(X_i)_{i \in \mathbb{N}}$
  is a Markov chain over $\mathbb{Z}^k$ with $X_0 = 0^k$
  and a distribution of one step $Z_{i+1} := X_{i+1} - X_i$ that
  does not depend on $i$.
\end{definition}

Defining $S_i:=(S_i^{(1)},\ldots,S_i^{(K)})$, note that $\left(S_i\right)_{i \in \{0,\ldots,n\}}$
is a $K$-dimensional random walk
and that we want to calculate $\Pr(\sgn(S_n)=y|\cE_d)$, for  $y\in\{-1,1\}^K$.
  There is one technicality we need to address to apply a local CLT:  since the
  steps of our random walk are in $\{-1, 1\}^K$,
  the values of $(S_i)$ lie on a proper sublattice of $\mathbb{Z}^K$, namely,
  $S_i^{(j)}$ always has the same parity as $i$. To deal with this, we define
  $T_i^{(j)} := (S^{(j)}_{2i+1}-1)/2$. Note that $\left(T_i\right)$
  is still a $K$-dimensional random walk, with one catch:
  the starting point
  $T_0$ is not necessarily the origin, but rather one of $k!$
  points in $\{-1, 0\}^K$ corresponding to the transitive tuple picked
  by the first voter.

  Before we state the local CLT, we need another definition:
  \begin{definition}[\cite{Spi76}, D1 in Section 5]
    A $K$-dimensional random walk is \emph{strongly aperiodic} if
    for every $t \in \mathbb{Z}^K$, the subgroup of $\mathbb{Z}^K$ generated by the points
    that can be reached from $t$ in one step is equal to $\mathbb{Z}^K$.
  \end{definition}

  Now we are ready to state the theorem:
  \begin{theorem}[Local CLT, Remark after P9 in Section 7 of \cite{Spi76}]
    \label{thm:lclt}
    Let $\left(T_i\right)_{i \in \mathbb{N}}$ be a strongly aperiodic
    $K$-dimensional random walk,
    starting at origin and with a single step $Z$, i.e.,
    $T_{i+1}-T_i$ distributed according to $Z$.

    If $\EE[Z] = 0^K$ and $Q$ is the $K\times K$
    (finite) covariance matrix of $Z$, then matrix $Q$ is invertible and
    for every $t \in \mathbb{Z}^K$,
    \begin{align*}
      \left|
      \left(2\pi n\right)^{K/2}
      \Pr\left[T_n = t\right]
      - |Q|^{-1/2} \exp\left(\frac{-t^T Q^{-1} t}{2n} \right)
      \right| = o(1) \; ,
    \end{align*}
    where the $o(1)$ function depends on $n$, but not on $t$.
  \end{theorem}

  Our main lemma states that
  the distribution of $T_n$ conditioned on $\|T_n\|_{\infty}$
  being small is roughly uniform.

  \begin{lemma}
    \label{lem:t-uniform}
    For the random walk $\left(T_i\right)$ defined above and
    $t \in \mathbb{Z}^K, d \ge 1$ such that $\|t\|_{\infty} \le d$,
    there are some $\alpha_k, \beta_k > 0$ such that
    \begin{align}
    \left| \alpha_k n^{K/2} \Pr\left[ T_n = t \right] - 1 \right| \le
      \beta_k \frac{d^2}{n}  + o_k(1) \; .
      \label{eq:22}
    \end{align}
  \end{lemma}

  \begin{proof}
    We first deal with the technicality that we mentioned before:
    the starting point $T_0$ of the random walk
    is itself a random variable. In the proof below
    we proceed by conditioning on $T_0 = 0^K$. After reading the proof it should
    be clear how to modify it for other starting points in $\{-1, 0\}^K$.
    \eqref{eq:22} is obtained from those conditional results by triangle
    inequality.
    
    We need to check that the random walk $\left(T_i\right)$ satisfies
    hypothesis of Theorem~\ref{thm:lclt}. First, note that the ``step''
    random variable $Z$ for $(T_i)$ has the same distribution as $(X_1+X_2)/2$,
    i.e., two steps of our original random process.

    Clearly, $\EE[Z] = (\EE[X_1] + \EE[X_2])/2 = 0^K$. Equally clearly, all
    covariances in the matrix $Q$ are finite.

    To show that $(T_i)$ is strongly aperiodic, let
    $(e^{(1)}, \ldots, e^{(K)})$ be the standard basis of $\mathbb{Z}^K$.
    Note that it is enough to show that for each $z \in \mathbb{Z}^K$,
    all of $z, z+e^{(1)}, \ldots, z+e^{(K)}$ are reachable from $z$ in one step.
    But this is so:
    \begin{itemize}
    \item It is possible to stay at $z$ by choosing a permutation (ranking) $\tau$
      for $X_1$
      and then its reverse $\tau^R$ for $X_2$.
    \item We explain how one can move from $z$ to $z+e^{(j)}$ on an
      example and hope it is clear how to generalize it.
      For $k=5$ and $e^{(j)}$ corresponding
      to the $b$ vs.~$d$ comparison, one can choose a ranking
      $b > d > a > c > e$ for $X_1$ followed by $e > c > a > b > d$ for $X_2$.
    \end{itemize}

    Since Theorem~\ref{thm:lclt} applies, we have
    \begin{align*}
      \left|
      \left(2\pi n\right)^{K/2}
      \Pr\left[T_n = t\right]
      - |Q|^{-1/2} \exp\left(-t^T Q^{-1} t / 2n \right)
      \right| = o_k(1) \; ,
    \end{align*}
    which can be rewritten as
    \begin{align*}
      \left|
      \alpha_k  n^{K/2}
      \Pr\left[T_n = t\right]
      - \exp\left(-t^T Q^{-1} t / 2n \right)
      \right| = o_k(1) \; .
    \end{align*}
    Since $1-x \le \exp(-x) \le 1$ for $x \ge 0$, it follows that
    \begin{align*}
      \left|
      \alpha_k n^{K/2}
      \Pr\left[T_n = t\right]
      - 1
      \right| \le \frac{t^T Q^{-1} t}{2n} +  o_k(1) \; .
    \end{align*}
    Finally we observe that $t = d t'$ for some $t'$ with
    $\|t'\|_{\infty} \le 1$, so we have
    \begin{align*}
      \frac{t^T Q^{-1}t}{2n} \le \beta_k \frac{d^2}{n} \; , 
    \end{align*}
    as we needed.
  \end{proof}

Lemma~\ref{lem:t-uniform} implies:
\begin{corollary}
  \label{cor:s-uniform}
  Let $n$ be odd, $d \ge 1$ and $s \in (2\mathbb{Z}+1)^K$ be a tuple such that
  $\|s\|_{\infty} \le d$. Then for some $\alpha_k, \beta_k > 0$,
  \begin{align*}
    \left| \alpha_k \left(n-1\right)^{K/2} \Pr\left[ S = s \right] - 1 \right| \le
    \beta_k \frac{d^2}{n}  + o_k(1) \; .
  \end{align*}
\end{corollary}

\begin{proof}
  Letting $t := (s - 1^K)/2$, note that
  $\Pr[S_n = s] = \Pr\left[T_{(n-1)/2} = t\right]$ and that $\|t\|_{\infty} \le d$.
  We get the result by applying
  Lemma~\ref{lem:t-uniform}.
\end{proof}

\subsection{Proof of Theorem~\ref{thm:elections-random}}

Recall that we want to prove \eqref{eq:20}, that is
\begin{align*}
  \left| \Pr \left[ Y = y \mid \cE_d \right] - \frac{1}{2^K} \right|
  \le \alpha_k \frac{d^2}{n} + o(1) \; .
\end{align*}
After we have \eqref{eq:20}, the bounds
\eqref{eq:24} and \eqref{eq:25} easily follow 
by triangle inequality.

For $y \in \{-1,1\}^K$, let
$\cS_y := \big\{ s \in (2\mathbb{Z}+1)^K: 
  \bigwedge_{j \in [K]} \sgn\left(s^{(j)}\right) = y^{(j)} \land
\|s\|_{\infty} \le d \big\}$. Observe that
$\Pr[ Y = y \land \cE_d] = \sum_{s \in \cS_y} \Pr[S = s]$. Furthermore, note that
$|\cS_y| = |\cS_{y'}|$ for every $y, y'$. Set $M := |\cS_y|$ as the
common cardinality of the $\cS_y$ sets.

First, we use Corollary~\ref{cor:s-uniform} to show that the probability
$\Pr[Y = y \mid \cE_d]$ must be close
to $q := \frac{1}{\alpha_k (n-1)^{K/2}} \cdot \frac{M}{\Pr[\cE_d]}$,
where $\alpha_k$ is the constant from Corollary~\ref{cor:s-uniform}:
\begin{align*}
  \left| \frac{\Pr[Y = y \mid \cE_d]}{q} - 1 \right|
  &=
  \left| \frac{\alpha_k(n-1)^{K/2}\Pr[\cE_d]}{M}
  \cdot \Pr[Y = y \mid \cE_d]
    - 1 \right| \\
  &=  
    \left| \frac{\alpha_k(n-1)^{K/2}}{M} \cdot
    \sum_{s \in \cS_y} \Pr[S = s]
    - 1 \right|
  \\ &\le
  \frac{1}{M} \sum_{s \in \cS_y}
  \left| \alpha_k(n-1)^{K/2}\Pr[S = s] - 1 \right|
  \le \beta_k \frac{d^2}{n} + o(1) \; .     
\end{align*}
The value of $q$ depends on $k$, $n$ and $d$, but not on $y$.
The implication is that the conditional probabilities must be almost equal
for every pair $y, y'$:
\begin{align*}
  \Big| \Pr[Y = y \mid \cE_d] - \Pr[Y = y' \mid \cE_d] \Big|
  &\le
  \Big| \Pr[Y = y \mid \cE_d] - q \Big| + \Big|q-\Pr[Y=y'\mid\cE_d]\Big|
  \\ &\le
  2q \left(\beta_k \frac{d^2}{n} + o(1) \right)
  \le \beta'_k\frac{d^2}{n} + o(1) \; .     
\end{align*}
But this is all we need, since
\begin{align*}
  \left| \Pr[Y = y \mid \cE_d] - \frac{1}{2^K} \right|
  &\le
  \frac{1}{2^K} \sum_{y' \in \{-1,1\}^K}
  \Big| \Pr[Y = y\mid \cE_d] - \Pr[Y=y'\mid\cE_d]  \Big|
  \\ &\le  
  \beta_k \frac{d^2}{n} + o(1) \; .
\end{align*}
\qed

\begin{remark}
  A similar bound with an explicit $o(1)$ term of the order
  $O_k \big(\frac{d}{\sqrt{n}} \big) + O_k \big(\frac{n^{K/2-1}}{d^K} \big)$
  (implying chaotic behavior for
  $n^{1/2-1/K} \ll d \ll n^{1/2}$)
  can be achieved
  using the multidimensional Berry-Esseen theorem instead of the local
  CLT.
\end{remark}

\begin{remark}
  As we mentioned in Section~\ref{sec:condorcet-intro}, the proof of
  Theorem~\ref{thm:elections-random} can be modified to give a similar bound
  \begin{align*}
    \Pr\left[Y = y \mid \cE^{(a_0 b_0)}_d \right] = \frac{1}{2^K} +
    o(1)
  \end{align*}
  for $d = o(\sqrt{n})$
  also in case the event $\cE^{(a_0 b_0)}_d$ is defined as
  $\left|S^{(ab)}\right| \le d$ for all pairwise comparisons $(ab)$ different
  from $(a_0b_0)$.

  The reason for this is that if we remove conditioning from just one
  $S^{(a_0 b_0)}$,
  there are still no covariance factors in the CLT computation
  that would steer the distribution of $Y$ away from uniform.
\end{remark}


\section{Condorcet paradox for close elections: Majority of triplets}
\label{sec:triplet}

Recall that we are considering
odd $n = 3m$ voters, alternatives $a, b, c$ and random variables
$x_1^{(kk')},\allowbreak\ldots,\allowbreak x_n^{(kk')}$ and that the pairwise comparison is done
according to
$f\colon \{-1, 1\}^n \to \{-1, 1\}$:
\begin{align*}
  f(x_1, \ldots, x_n) = \sgn\left( \sum_{i=1}^m \sgn\left(w_i\right) \right) \; , \quad\text{where $w_i = x_{3i-2} + x_{3i-1} + x_{3i}$.}
\end{align*}
         This section contains proofs of non-chaotic behavior of $f$ under certain
conditionings. Section~\ref{sec:triplet-sum-proof} contains the proof
of Theorem~\ref{thm:triplet-sum}, dealing with conditioning on
small $\big|\sum_{i=1}^n x_i^{(kk')}\big|$. In
Section~\ref{sec:triplet-noise-proof} we prove Theorem~\ref{thm:triplet-noise},
which considers conditioning on small
$\big|T_\rho f(x^{(kk')} )\big|$.

\subsection{Proof of Theorem~\ref{thm:triplet-sum}}
\label{sec:triplet-sum-proof}
For $i \in [m]$, we take random tuple
$Z_i := \big(A^{(kk')}_i, B_i^{(kk')}\big)_{(kk')}$ for
$kk' \in \{ab, bc, ca\}$, where
$A_i^{(kk')} := w_i^{(kk')} / \sqrt{3}$ and
  $B_i^{(kk')} := \sgn\big(w_i^{(kk')}\big)$. Note that
$Z_1, \ldots, Z_m$ are i.i.d. Let us compute the first two moments
of the single-voter distribution
$Z = (A^{(ab)}, A^{(bc)}, A^{(ca)},\allowbreak B^{(ab)}, B^{(bc)}, B^{(ca)})$.
For this keep in mind that $\Cov\big[x^{(kk')}_i, x^{(k'k'')}_i \big] = -1/3$
and refer to Table~\ref{tab:01} for the joint distribution of
$w^{(kk')}$ and $w^{(k'k'')}$:
\begin{align}
  \EE\big[A^{(kk')}\big] &= \EE\big[B^{(kk')}\big] = 0\nonumber\\
  \Var\big[A^{(kk')}\big] &= \Var\big[B^{(kk')}\big] = 1\nonumber\\
  \Cov\big[A^{(kk')}, A^{(k'k'')}\big]
  &= -\frac{1}{3}\nonumber\\
  \Cov\big[B^{(kk')}, B^{(k'k'')}\big] &= \frac{80-136}{8 \cdot 27}
  = -\frac{7}{27}\label{eq:111}\\
  \Cov\big[A^{(kk')}, B^{(kk')}\big] &= \frac{1}{\sqrt{3}} \cdot \frac{3}{2}
  = \frac{\sqrt{3}}{2}\nonumber\\
  \Cov\big[A^{(kk')}, B^{(k'k'')}\big] &= \frac{1}{\sqrt{3}} \cdot
  \frac{3\cdot 14 + 66 - 96 - 3\cdot 40}{8 \cdot 27}
  = - \frac{1}{2\sqrt{3}} \; .\nonumber
\end{align}

\begin{table}[!h]\centering
  \begin{tabular}{|c|c|c|c|c|}
    \hline
    $w^{(kk')}$ vs.~$w^{(k'k'')}$ & $-3$ & $-1$ & $1$ & $3$
    \\ \hline
    $-3$ & $1$ & $6$ & $12$ & $8$\\ \hline
    $-1$ & $6$ & $27$ & $36$ & $12$\\ \hline
    $1$ & $12$ & $36$ & $27$ & $6$\\ \hline
    $3$ & $8$ & $12$ & $6$ & $1$ \\ \hline
  \end{tabular}
  \caption{Probabilities of values for $w^{(kk')}, w^{(k'k'')}$ pairs multiplied by
    common denominator $8 \cdot 27$. Keep in mind that $x_i^{(kk')}$ and
    $x_i^{(k'k'')} \in \{-1, 1\}$
    are equal with probability $1/3$.}
  \label{tab:01}
\end{table}

Let $\tilA^{(kk')} := \sum_{i=1}^m A_i^{(kk')} / \sqrt{m}$ and
$\tilB^{(kk')} := \sum_{i=1}^m B_i^{(kk')}/\sqrt{m}$
and let $\tilM^{(kk')}$ and $\tilN^{(kk')}$ be joint standard Gaussians with
the same covariance structure as $\tilA^{(kk')}$ and $\tilB^{(kk')}$ respectively.
After checking that our six by six covariance matrix is not singular,
by the multi-dimensional Berry-Esseen theorem
(see the statement e.g., in \cite{Ben05}),
we can move to the Gaussian space:
\begin{align}
  &\Pr\left[
    f (\mathbf{x}^{(ab)} ) = f (\mathbf{x}^{(bc)} )
    = f (\mathbf{x}^{(ca)} ) \land \cE_d \right]\nonumber\\
  &\qquad =
    2 \Pr\left[
    f (\mathbf{x}^{(ab)} ) = f (\mathbf{x}^{(bc)} )
    = f (\mathbf{x}^{(ca)} ) = 1 \land \cE_d \right]\nonumber\\
  &\qquad =
    2 \Pr\left[
    \| \tilA \|_{\infty} \le \frac{d}{\sqrt{3m}} \land
    \tilB \ge 0
    \right]\nonumber\\
  &\qquad =
    2\Pr\left[
    \| \tilM \|_{\infty} \le \frac{1}{\log n} \land
    \tilN\ge 0
    \right] + O\left( \frac{1}{\sqrt{n}} \right)\;,
    \label{eq:45}
\end{align}
where we write $\tilB\ge 0$ to indicate $\tilB^{(kk')}\ge 0$ for every component
of $\tilB$. Similarly,
\begin{align*}
  \Pr[\cE_d]=\Pr\left[\|\tilM\|_\infty\le\frac{1}{\log n}\right]
  +O\left(\frac{1}{\sqrt{n}}\right)\;.
\end{align*}

Let us define three more centered Gaussians $\tilR^{(kk')}$ according to
the formula
\begin{align}
  \label{eq:46}
  \tilN^{(kk')} = \frac{\sqrt{3}}{2}\tilM^{(kk')} + \frac{1}{2}\tilR^{(kk')}\;.
\end{align}
Since $\Cov[\tilM^{(kk')},\tilN^{(kk')}]=\Cov[A^{(kk')},B^{(kk')}]=\sqrt{3}/2$,
we immediately see that $\Var[\tilR^{(kk')}]=1$ and
$\Cov[\tilM^{(kk')},\tilR^{(kk')}]=0$. Furthermore, we calculate
\begin{align}
  \Cov[\tilM^{(kk')},\tilR^{(k'k'')}]
  &=2\Cov[\tilM^{(kk')},\tilN^{(k'k'')}]-\sqrt{3}\Cov[\tilM^{(kk')},\tilM^{(k'k'')}]
  \nonumber\\
  &=2\Cov[A^{(kk')},B^{(k'k'')}]-\sqrt{3}\Cov[A^{(kk')},A^{(k'k'')}]=0\;,
  \label{eq:113}\\
  \Cov[\tilR^{(kk')},\tilR^{(k'k'')}]
  &=4\Cov[\tilN^{(kk')},\tilN^{(k'k'')}]-4\sqrt{3}\Cov[\tilM^{(kk')},\tilN^{(k'k'')}]
  \nonumber\\
  &\qquad+3\Cov[\tilM^{(kk')},\tilM^{(k'k'')}]=-\frac{1}{27}\;.
    \nonumber
\end{align}
Recall the joint density function for centered Gaussians: in $k$ dimensions,
for the distribution with covariance matrix $\Sigma$ and
$\bfx=(x_1,\ldots,x_k)$ we have
\begin{align*}
  f_\Sigma(\bfx)=\frac{1}{\sqrt{(2\pi)^k|\Sigma|}}
  \exp\left(-\bfx^T\Sigma^{-1}\bfx\right)\;.
\end{align*}
In particular, letting $c_\Sigma:=f_\Sigma(0)$,
we have basic approximation
\begin{align}\label{eq:112}
  f_\Sigma(\bfx)= c_\Sigma+O(\|\bfx\|^2)\;.
\end{align}

Letting $D:=\{\bfm\in\mathbb{R}^{3}:\|\bfm\|_{\infty}\le 1/\log n\}$
and using this approximation, we have
\begin{align*}
  \Pr\left[\|\tilM\|_{\infty}\le\frac{1}{\log n}\right]
  =\int_D f_M(\bfm)\,d\bfm
  =\frac{8c_M}{\log^3 n}+O\left(\frac{1}{\log^5 n}\right)\;.
\end{align*}
As for calculating~\eqref{eq:45}, given $\bfm\in D$, let
\begin{align*}
  D_{\bfm}:=\left\{\bfr\in\mathbb{R}^3:
  \frac{\sqrt{3}}{2}\bfm+\frac{1}{2}\bfr\ge 0\right\}\;.
\end{align*}
In particular, we have $D_0=\{\bfr:\bfr\ge 0\}$.
Let $f_R$ be the density function of the Gaussian triple $\tilR$ and let
\begin{align*}
  \alpha^*:=2\Pr[\tilR\ge 0]=2\int_{D_0} f_R(\bfr)\,d\bfr\;.
\end{align*}
Note that if $\|\bfm\|_{\infty}\le1/\log n$ and
$\bfr\in D_0\Delta D_\bfm$, then there exists at least one coordinate
on which $|r_i|=O(1/\log n)$. Therefore, we obtain
\begin{align*}
  \left|\int_{D_\bfm}f_R(\bfr)\,d\bfr-\frac{\alpha^*}{2}\right|
  &\le\int_{D_0\Delta D_\bfm}f_R(\bfr)\,d\bfr\\
  &\le 3\Pr\left[|\tilR^{(kk')}|\le O\left(\frac{1}{\log n}\right)\right]
  =O\left(\frac{1}{\log n}\right)\;,
\end{align*}
where the error term is uniform in $\bfm$.

Finally, we recall~\eqref{eq:113} to observe that Gaussian triples $\tilM$
and $\tilR$ are independent and therefore their joint density decomposes
$f_{M,R}(\bfm, \bfr)=\allowbreak f_M(\bfm)\allowbreak f_R(\bfr)$.
That allows us to calculate,
using~\eqref{eq:46},
\begin{align*}
  &\Pr\left[\|\tilM\|_{\infty}\le\frac{1}{\log n}\land \tilN\ge 0\right]
    =\int_{D}f_M(\bfm)\int_{D_m}f_R(\bfr)\,d\bfr d\bfm\\
  &\qquad\qquad=\int_D f_M(\bfm)\left(\frac{\alpha^*}{2}+O\left(\frac{1}{\log n}\right)\right)
    \,d\bfm
    =\frac{8c_M}{\log^3 n}\cdot\frac{\alpha^*}{2}+O\left(\frac{1}{\log^4 n}\right)\;.
\end{align*}
In conclusion, we get
\begin{align*}
  &\Pr\left[
    f (\mathbf{x}^{(ab)} ) = f (\mathbf{x}^{(bc)} )
    = f (\mathbf{x}^{(ca)} ) \mid\cE_d \right]\\
  &\qquad\qquad
    =\frac{2\Pr\left[\|\tilM\|_{\infty}\le 1/\log n\land \tilN\ge 0\right]+O(1/\sqrt{n})}
    {\Pr\left[\|\tilM\|_{\infty}\le 1/\log n\right]+O(1/\sqrt{n})}\\
  &\qquad\qquad=
    \frac{\frac{8c_M}{\log^3 n}\alpha^*+O(1/\log^4 n)}
    {\frac{8c_M}{\log^3 n}+O(1/\log^5 n)}
    =\alpha^*+O\left(\frac{1}{\log n}\right)\\
  &\qquad\qquad\qquad\qquad\xrightarrow{n\to\infty}\alpha^*\approx 23.2\%\;,
\end{align*}
where in the very last step we employed a computer algebra system to
compute the approximate value of $\alpha^*$.

\subsection{Proof of Theorem~\ref{thm:triplet-noise}}
\label{sec:triplet-noise-proof}
The proof of Theorem~\ref{thm:triplet-noise} is a refinement
of the proof of Theorem~\ref{thm:triplet-sum}, which is a recommended
preliminary reading. In particular, we will use the notation
that was developed there. From now on the constants in the $O(\cdot)$ notation
are allowed to depend on $\rho$.
Recall that for $\mathbf{x} \in \{-1, 1\}^n$ and
$\mathbf{w} \in \{\pm 3, \pm 1\}^m$ we have defined 
\begin{align*}
  W_b(\mathbf{x}) = W_b(\mathbf{w})
  &=
    \left|\left\{i \in [m]: w_i = b\right\}\right| \; ,\\
  V_b(\mathbf{x}) = V_b(\mathbf{w})
  &=
    W_b(\mathbf{w}) - \EE_{\mathbf{w'}} \left[ W_b(\mathbf{w'}) \right]
    = W_b(\mathbf{w}) -
    \begin{cases}
      n/8& \text{if } b = \pm 3\; ,\\
      3n/8& \text{if } b = \pm 1 \; .
    \end{cases}
\end{align*}
We can write $W_b(\mathbf{w}) = \sum_{i=1}^m W_b(w_i)$
and $V_b(\mathbf{w}) = \sum_{i=1}^m V_b(w_i)$ in an obvious way, with
$W_b(w_i) \in \{0, 1\}$, $V_{\pm 3}(w_i) \in \{-1/8, 7/8\}$ and
$V_{\pm 1}(w_i) \in \{-3/8, 5/8\}$. Note that
$W_3(w_i)+W_1(w_i)+W_{-1}(w_i)+W_{-3}(w_i) = 1$
and $V_3(w_i)+V_1(w_i)+V_{-1}(w_i)+V_{-3}(w_i) = 0$.

Taking $w_i = x_{3i-2} + x_{3i-1} + x_{3i}$,
$w_i' = x'_{3i-2} + x'_{3i-1} + x'_{3i}$, $s_i=\sgn(w_i)$ and
$s'_i=\sgn(w'_i)$, where $(x_i, x'_i)$
are $\rho$-correlated, we also define
\begin{align}
  \eps
  &:= \Pr\left[ x_j \ne x'_j \right] = (1-\rho)/2 \; ,\label{eq:48}\\
  p_3
  &:= \Pr\left[ s_i = s'_i \mid w_i =  3 \right]
    = (1-\eps)^3 + 3\eps(1-\eps)^2 \; ,\label{eq:49}\\
  p_1
  &:= \Pr\left[ s_i = s'_i \mid w_i = 1 \right]
    = (1-\eps)^3 + \eps(1-\eps)^2 + 2\eps^2(1-\eps) \; .\label{eq:50}
\end{align}
Recall that
\begin{align*}
  T_\rho f(\mathbf{x})=\EE_{\mathbf{x'}\sim N_{\rho}(\mathbf{x})}[f(\mathbf{x'})]
\end{align*}
and observe that for our particular function $f$ the value of $T_\rho f$
depends only on $\mathbf{w}$ and equals
\begin{align*}
  T_\rho f(\mathbf{w})=\EE_{\mathbf{s'}\sim N_\rho(\mathbf{w})}\left[
  \sgn\left(\sum_{i=1}^m s'_i\right)
  \right]
  =2\Pr\left[
  \sum_{i=1}^m s'_i >0
  \right]-1\;,
\end{align*}
where random variables $s'_i\in\{-1,1\}$ are independent and
$\Pr[s_i=s'_i]=p_b$ if $|w_i|=b$ for $b=1,3$. In particular, we can also write
$T_\rho f(\mathbf{w})$ as a sum of four independent binomial random variables
\begin{align}
  T_\rho f(\mathbf{w})
  &=
    2\Pr\Big[
    \Bin\left(W_3(\mathbf{w}), p_3\right) +
    \Bin\left(W_1(\mathbf{w}), p_1\right) \nonumber \\
  &\quad +
  \Bin\left(W_{-1}(\mathbf{w}), 1-p_1\right) +
  \Bin\left(W_{-3}(\mathbf{w}), 1-p_3\right)
  > \frac{m}{2} \Big]-1\; . \label{eq:56}
\end{align}

Our plan is to use a CLT argument to conclude that, for most
values of $\mathbf{w}$ under event $\mathcal{F}_{\rho,d}$, the value of
$T_\rho f(\mathbf{w})$ is proportional to
\begin{align*}
  T_\rho f(\mathbf{w})
  &\asymp
    \frac{ p_3W_3(\mathbf{w}) + p_1W_1(\mathbf{w}) +
    (1-p_1)W_{-1}(\mathbf{w}) + (1-p_3)W_{-3}(\mathbf{w}) - m/2}{\sqrt{m}}\\
  &=
    \frac{p_3V_3(\mathbf{w}) + p_1V_1(\mathbf{w}) +
    (1-p_1)V_{-1}(\mathbf{w}) + (1-p_3)V_{-3}(\mathbf{w})}{\sqrt{m}}\\
  &=
    \frac{ q_3V_3(\mathbf{w}) + q_1V_1(\mathbf{w}) - q_1V_{-1}(\mathbf{w})
    -q_3V_{-3}(\mathbf{w})}{\sqrt{m}} \; ,
\end{align*}
where $q_3 := p_3 - 1/2$ and $q_1 := p_1-1/2$. We now state this more precisely
as a lemma, the proof of which we defer until later:
\begin{lemma}
  \label{lem:linear-noise}
  Let $\sigma_3^2 := p_3(1-p_3)$, $\sigma_1^2 := p_1(1-p_1)$ and
  $\sigma^2 := \frac{\sigma_3^2 + 3\sigma_1^2}{4}$. Let
  \begin{align*}
    A_i^{(kk')}
    &:=
      q_3V_3 \big(w_i^{(kk')} \big) +
      q_1V_1 \big(w_i^{(kk')} \big) -
      q_1V_{-1} \big(w_i^{(kk')} \big) -
      q_3V_{-3} \big(w_i^{(kk')} \big) \; ,\\
    \tilA^{(kk')}
    &:=
      \frac{1}{\sqrt{m}} \sum_{i=1}^m A_i^{(kk')}  \; .
  \end{align*}
  Take $ C := \sqrt{\frac{\pi}{2}} \sigma$ and define events
  \begin{align*}
    \cG_1
    &:\equiv \cF_{\rho, d} \equiv
      \quad\max\left(
      \big| T_\rho f (\mathbf{x}^{(ab)} )\big|,
      \big| T_\rho f (\mathbf{x}^{(bc)} )\big|,
      \big| T_\rho f (\mathbf{x}^{(ca)} )\big|
      \right) \le \frac{1}{\log m} \; ,\\
    \cG_2
    &:\equiv\qquad\qquad\|\tilA\|_\infty=\max\left(
      \big| \tilA^{(ab)} \big|,
      \big| \tilA^{(bc)} \big|,
      \big| \tilA^{(ca)} \big|
      \right) \le \frac{C}{\log m} \; .
  \end{align*}
  Let $\Delta$ stand for the symmetric difference of events. Then,
  \begin{align*}
    \Pr\left[ \cG_1 \Delta \cG_2 \right]
    \le O\left(\frac{1}{\log^5 m}\right) \; .
  \end{align*}
\end{lemma}
Assuming Lemma~\ref{lem:linear-noise} we continue along the lines
of the proof of Theorem~\ref{thm:triplet-sum}, letting
$B_i^{(kk')} := \sgn (w_i^{(kk')} )$ and
$Z_i :=  \big(A_i^{(kk')}, B_i^{(kk')} \big)_{(kk')}$. The random variables
$Z_1, \ldots, Z_m$ are i.i.d.~and for CLT purposes we can compute
(again Table~\ref{tab:01} is helpful)
the six by six covariance matrix of the distribution of $Z := Z_1$:
\begin{align}
  \EE\left[A^{(kk')}\right] &= \EE\left[B^{(kk')}\right] = 0\nonumber\\
  \Var\left[A^{(kk')}\right] &= \frac{q_3^2 + 3q_1^2}{4}\label{eq:60}\\
  \Var\left[B^{(kk')}\right] &= 1\nonumber\\
  \Cov\left[A^{(kk')}, A^{(k'k'')}\right]
  &= \frac{-14q_3^2-24q_1q_3-18q_1^2}{216}  \label{eq:61}\\
  \Cov\left[B^{(kk')}, B^{(k'k'')}\right] &= \frac{80-136}{8 \cdot 27}
  = -\frac{7}{27}\nonumber\\
  \Cov\left[A^{(kk')}, B^{(kk')}\right] &= \frac{q_3+3q_1}{4}\nonumber\\
  \Cov\left[A^{(kk')}, B^{(k'k'')}\right]
  &= \frac{-26q_3-30q_1}{216}\nonumber
\end{align}

Let $\big(\tilM^{(kk')}, \tilN^{(kk')}\big)_{(kk')}$ be joint Gaussians
with same covariance structure as
$\big(\tilA^{(kk')}, \tilB^{(kk')}\big)_{(kk')}$.
Further symbolic computations in a computer algebra system lead to expressing
$\tilN^{(kk')}$ as a linear combination
\begin{align}
  \label{eq:51}
  \tilN^{(kk')} = \beta \tilM^{(kk')} + \beta'
  \left(\tilM^{(k'k'')} + \tilM^{(k''k)}\right) + \gamma \tilR^{(kk')} \; ,
\end{align}
where $\gamma > 0$, random tuples $\big(\tilM^{(kk')}\big)_{(kk')}$ and
$\big(\tilR^{(kk')}\big)_{(kk')}$ are independent of each other
and each $\tilR^{(kk')}$ is a standard Gaussian. Furthermore, we obtain
\begin{align}
  \label{eq:52}
  \Cov\left[\tilR^{(kk')}, \tilR^{(k'k'')}\right] = \Cov(\rho)
\end{align}
with $\Cov(\rho)$ a decreasing function of $\rho\in(0, 1)$ and
\begin{align*}
  \Cov(\rho)\le-\frac{1}{27}=\lim_{\rho\to 0^+}\Cov(\rho)\;.
\end{align*}
Since the mutual covariance $\Cov(\rho)$ is decreasing, the expression
\begin{align*}
  \alpha(\rho):=2\Pr[\tilR\ge 0]
\end{align*}
is also decreasing in $\rho$, with $\lim_{\rho\to 0^+}\alpha(\rho)=\alpha^*$
and $\alpha(\rho)\ge\lim_{\rho\to 1^-}\alpha(\rho)\ge 0.17$.

Let $\delta:=C/\log m$ and recall Lemma~\ref{lem:linear-noise}.
We apply this lemma and
similar arguments as in the proof of
Theorem~\ref{thm:triplet-sum} and calculate
\begin{align}
  &\Pr\left[ f\left(\mathbf{x}^{(ab)}\right) =
    f\left(\mathbf{x}^{(bc)}\right) = f\left(\mathbf{x}^{(ca)}\right)
    \mid \cF_{\rho,d}\right]\nonumber\\
  &\qquad=
  \frac{2\Pr\left[ f\left(\mathbf{x}^{(ab)}\right) =
  f\left(\mathbf{x}^{(bc)}\right) = f\left(\mathbf{x}^{(ca)}\right) = 1
   \land \cG_1\right]}
    {\Pr[\cG_1]}\nonumber\\
  &\qquad=
    \frac{2\Pr\left[ \tilB\ge 0
    \land \cG_2\right] + O(1/\log^5 m)}
    {\Pr[\cG_2]+ O(1/\log^5 m)}\nonumber\\
  &\qquad=
    \frac{2\Pr\left[ \tilB\ge 0 \land \| \tilA \|_{\infty} \le \delta
    \right] + O(1/\log^5 m)}
    {\Pr\left[\|\tilA\|_{\infty} \le \delta \right]
    + O(1/\log^5 m)}\nonumber\\
  &\qquad=
    \frac{2\Pr\left[ \tilN\ge 0 \land \| \tilM \|_{\infty} \le \delta
    \right] + O(1/\log^5 m)}
    {\Pr\left[\|\tilM\|_{\infty} \le \delta \right]
    + O(1/\log^5 m)} \nonumber\\
  &\qquad=
    \frac{8c_M\delta^3\cdot \alpha(\rho)+O(1/\log^4 m)}
    {8c_M\delta^3+O(1/\log^5 m)}
    =\alpha(\rho)+O\left(\frac{1}{\log m}\right)\;.
\end{align}
It remains to prove Lemma~\ref{lem:linear-noise}.
\begin{proof}[Proof of Lemma~\ref{lem:linear-noise}]
Recall the definitions of $W_b(\mathbf{w})$ and $V_b(\mathbf{w})$.
We begin with estimating $T_\rho f(\mathbf{w})$ for a fixed $\mathbf{w}$.
In the following
we will sometimes drop dependence on $\mathbf{w}$
(writing, e.g., $W_b$, $V_b$, $\tilA$ instead of
$W_b(\mathbf{w})$, $V_b(\mathbf{w})$, $\tilA(\mathbf{w})$) in the interest
of clarity. Recall equation \eqref{eq:56} and 
let $Z := \sum_{i=1}^m Z_i$ be the sum of $m$ independent random
variables arising out of the four binomial distributions featured there.
We have:
\begin{align*}
  T_\rho f(\mathbf{w})
  &= 2\Pr\left[ Z > \frac{m}{2} \right]-1 \; ,\\
  \EE\left[Z - \frac{m}{2}\right]
  &= p_3W_3 + p_1W_1 + (1-p_1)W_{-1} + (1-p_3)W_{-3} - \frac{m}{2}\\
  &= p_3V_3 + p_1V_1 + (1-p_1)V_{-1} + (1-p_3)V_{-3}\\
  &= q_3V_3 + q_1V_1 - q_1V_{-1} - q_3V_{-3} = \sqrt{m} \tilA \; ,\\
  \Var[Z]
  &= \sigma_3^2(W_3+W_{-3}) + \sigma_1^2(W_1+W_{-1})\\
  &= m\sigma^2 + \sigma_3^2(V_3+V_{-3}) + \sigma_1^2(V_1+V_{-1})
    = m\sigma^2 \left(1 + t\right)\; ,
\end{align*}
for $t := t(\mathbf{w}) :=
\frac{\sigma_3^2(V_3+V_{-3})+\sigma_1^2(V_1+V_{-1})}{\sigma^2 m}$.
Since random variables $Z_i$ are bounded, we can apply the Berry-Esseen
theorem and, using $\erf(x/\sqrt{2}) = 2\Phi(x)-1$
where $\erf(y):=\frac{2}{\sqrt{\pi}} \int_{0}^y e^{-s^2}ds$, find
\begin{align}
  \Pr\left[ Z - \frac{m}{2} > 0\right]
  &= \Pr\left[
    \frac{Z-m/2-\sqrt{m}\tilA}{\sqrt{m}\sigma\sqrt{1+t}} >
    \frac{-\tilA}{\sigma\sqrt{1+t}} \right] \notag \\
  &  = \Phi\left(\frac{\tilA}{\sigma\sqrt{1+t}}\right) +
    O\left(\frac{1}{\sqrt{m(1+t)^3}}\right) \; ,\nonumber \\
  T_\rho f(\mathbf{w})
  &= \erf\left(\frac{\tilA}{\sqrt{2}\sigma\sqrt{1+t}}\right) +
    O\left(\frac{1}{\sqrt{m(1+t)^3}}\right)  \label{eq:57} \; .
\end{align}
From now on we consider a random election with vote vectors
$\mathbf{x}^{(ab)}$, $\mathbf{x}^{(bc)}$, $\mathbf{x}^{(ca)}$ that
induce $\mathbf{w}^{(ab)}$, $\mathbf{w}^{(bc)}$,
$\mathbf{w}^{(ca)}$. First, consider the marginal distribution of
$\mathbf{w}$. Since $t(\mathbf{w})$ can be written
as a sum of $m$ i.i.d.~random variables
$\sigma^2 m t(\mathbf{w}) = \sum_{i=1}^m t_i(w_i)$ with $\EE[t_i] = 0$
and $|t_i| \le 1$, a standard concentration bound gives
\begin{align}
  \label{eq:59}
  \Pr\left[ \left| t(\mathbf{w})\right|
  > \frac{1}{m^{1/4}} \right]
  \le 2\exp\left(-\frac{\sqrt{m}\sigma^4}{2}\right)
  \le O\left(\frac{1}{\sqrt{m}}\right) \; .
\end{align}
As a consequence of~\eqref{eq:57} and~\eqref{eq:59} and the Taylor expansion
$\erf(x) = \frac{2}{\sqrt{\pi}}x + O(x^3)$,
whenever $|t|\le m^{-1/4}$ holds, we have
\begin{align}\label{eq:115}
  T_\rho f(\mathbf{w})
  &=\frac{\tilA}{C}+O(\tilA^3)+O\left(\frac{1}{m^{1/4}}\right)
\end{align}
and, furthermore,
\begin{align}
  |T_\rho f(\mathbf{w})|\le\frac{1}{\log m}
  &\implies
    |\tilA|\le\frac{C}{\log m}+O\left(\frac{1}{\log^3 m}\right)\;,
  \label{eq:116}\\
  \pm T_\rho f(\mathbf{w})>\frac{1}{\log m}
  &\implies
    \pm\tilA\ge\frac{C}{\log m}-O\left(\frac{1}{\log^3 m}\right)\;.\label{eq:117}
\end{align}

We are now ready to bound the measure of the symmetric difference
\begin{align*}
  \Pr\left[ \cG_1 \Delta \cG_2 \right]
  = \Pr[ \cG_1 \land \lnot \cG_2] +
  \Pr[ \lnot \cG_1 \land \cG_2] \; .
\end{align*}
We will use the union bound over a small number of
cases and show that each of them has probability $O(\log^{-5} m)$.

First, if $\mathcal{G}_1$ holds, but $\mathcal{G}_2$ does not, then
$|\tilA^{(kk'})|>C/\log m$ for some comparison $(kk')$.
Let us assume that $\tilA^{(ab)}>C/\log m$, other five cases being symmetrical.
We now apply~\eqref{eq:59},~\eqref{eq:116} and multivariate Berry-Esseen
and get
\begin{align*}
  &\Pr\left[\mathcal{G}_1\land \tilA^{(ab)}>\frac{C}{\log m}\right]\\
  &\qquad
    \le\Pr\left[
    \|\tilA\|_{\infty}\le\frac{C}{\log m}+O\left(\frac{1}{\log^3 m}\right)
    \land \tilA^{(ab)}>\frac{C}{\log m}
    \right]\\
    &\qquad\qquad\qquad+\Pr\left[\|t\|_{\infty}>\frac{1}{m^{1/4}}\right]
    \\
  &\qquad
    =\Pr\left[
    \frac{C}{\log m}\le\tilA^{(ab)}\le\frac{C}{\log m}+
    O\left(\frac{1}{\log^3 m}\right)
    \land |\tilA^{(bc)}|,|\tilA^{(ca)}|\le\frac{C}{\log m}
    \right]\\
  &\qquad\qquad\qquad+O\left(\frac{1}{\sqrt{m}}\right)\\
  &\qquad
    =\Pr\left[
    \frac{C}{\log m}\le\tilM^{(ab)}\le\frac{C}{\log m}+
    O\left(\frac{1}{\log^3 m}\right)
    \land |\tilM^{(bc)}|,|\tilM^{(ca)}|\le\frac{C}{\log m}
    \right]\\
  &\qquad\qquad\qquad+O\left(\frac{1}{\sqrt{m}}\right)
    =O\left(\frac{1}{\log^5 m}\right)\;.
\end{align*}
Applying union bound over remaining, symmetric cases, we obtain
\begin{align*}
  \Pr[\mathcal{G}_1\land\lnot\mathcal{G}_2]\le O\left(\frac{1}{\log^5 m}\right)\;.
\end{align*}

On the other hand, if $\mathcal{G}_2$ holds, but $\mathcal{G}_1$ does not,
then we have $|T_\rho f(\mathbf{x}^{(kk')})|>1/\log m$ for some $(kk')$,
for example, $T_\rho f(\mathbf{x}^{(ab)})>1/\log m$. A similar calculation
using~\eqref{eq:117} gives
\begin{align*}
  &\Pr\left[T_\rho f(\mathbf{x}^{(ab)})>\frac{1}{\log m} \land\mathcal{G}_2 \right]\\
  &\qquad
    \le\Pr\left[
    \tilA^{(ab)}\ge\frac{C}{\log m}-O\left(\frac{1}{\log^3 m}\right)
    \land\|\tilA\|_{\infty}\le\frac{C}{\log m}
    \right]\\
    &\qquad\qquad\qquad+\Pr\left[\|t\|_{\infty}>\frac{1}{m^{1/4}}\right]
    \\
  &\qquad
    =\Pr\left[
    \frac{C}{\log m}-O\left(\frac{1}{\log^3 m}\right)
    \le\tilA^{(ab)}\le\frac{C}{\log m}
    \land |\tilA^{(bc)}|,|\tilA^{(ca)}|\le\frac{C}{\log m}
    \right]\\
  &\qquad\qquad\qquad+O\left(\frac{1}{\sqrt{m}}\right)\\
  &\qquad
    =\Pr\left[
    \frac{C}{\log m}-O\left(\frac{1}{\log^3 m}\right)
    \le\tilM^{(ab)}\le\frac{C}{\log m}
    \land |\tilM^{(bc)}|,|\tilM^{(ca)}|\le\frac{C}{\log m}
    \right]\\
  &\qquad\qquad\qquad+O\left(\frac{1}{\sqrt{m}}\right)
    =O\left(\frac{1}{\log^5 m}\right)
\end{align*}
and
\begin{align*}
  \Pr[\lnot\mathcal{G}_1\land\mathcal{G}_2]=O\left(\frac{1}{\log^5 m}\right)\;.
\end{align*}
\end{proof}

\section{Arrow's theorem for dice}
\label{sec:arrow}

Arguably the most famous result in social choice theory is
Arrow's impossibility theorem \cite{Arr50, Arr63}.
Intuitively, it states that the only reasonable voting systems
based on pairwise comparisons that never produce a Condorcet paradox
are ``dictators'', i.e., functions whose value depend
only on a single voter.

There are also quantitative versions,
proved by Kalai~\cite{Kal02} for balanced functions and by Mossel
\cite{Mos12} for general functions
(with tighter bounds obtained by Keller \cite{Kel12}).
For simplicity we consider three alternatives and the
impartial culture model.
Then, the quantitative Arrow's theorem
says that a reasonable pairwise comparison function $f$ that is
$\eps$-far from every dictator (in the sense of normalized Hamming distance),
must be such that the probability of Condorcet paradox is at least $\Omega(\eps^3)$.

There is an analogous question about transitive dice: What are the
methods for pairwise comparisons of $k$ dice that always produce a linear
order?
In particular, we know that comparing two dice $\bfa$ and $\bfb$ by
using the ``beats'' relation is not one of them.

We restrict ourselves to $k=3$.
Assume that we look at dice with $n$ sides labeled with $[m]$,
i.e., multisets of elements of $[m]$ of size $n$. Denote the set of such
dice as $\cD_{m, n}$.
A pairwise comparison is an anti-symmetric function
$f\colon (\mathcal{D}_{m,n} \times \cD_{m,n}) \setminus
\diag(\cD_{m,n}\times \cD_{m,n}) \to \{-1, 1\}$.
We want to understand which pairwise comparison functions are transitive,
i.e., there are no three distinct dice $\bfa, \bfb, \bfc$ such that
$f(\bfa, \bfb) = f(\bfb, \bfc) = f(\bfc, \bfa)$.

A little thought reveals that the answer is somewhat trivial. 
Let $\mathcal{O}$ be a linear order on $\cD_{m,n}$. We  think
of $\mathcal{O}$ as an injective function
$\mathcal{O}\colon \cD_{m,n} \to \mathbb{R}$. If we define $f$
as
\begin{align*}
  f(\bfa, \bfb) = 1 \text{ if and only if } \mathcal{O}(\bfa) < \mathcal{O}(\bfb) \; ,
\end{align*}
then $f$ is easily seen to be transitive.

On the other hand, every transitive $f$ must
be of this form.
To see this, consider a directed graph with vertex set $\cD_{m,n}$ where
there is an edge from $\bfa$ to $\bfb$ if and only if $f(\bfa, \bfb) = -1$.
This graph is a tournament
and transitivity of $f$ means that it does not contain a directed triangle.
But a triangle-free tournament does not contain a directed cycle
and, therefore, induces a linear order on its ground set.

We can extend this reasoning to a quantitative result. It seems easiest
to assume a model where a set of three dice is sampled
u.a.r.~from $\cD_{m,n}$.

There is a result about tournaments due to Fox and Sudakov \cite{FS08}.
A tournament on $n$ vertices is called \emph{$\eps$-far from transitive} if
at least $\eps n^2$ of its edges must be reversed to obtain a transitive
tournament.
\begin{theorem}[\cite{FS08}]
  \label{thm:fs08}
  There exists $c > 0$ such that if a tournament on $n$ vertices is
  $\eps$-far from transitive, then it contains at least
  $c\eps^2n^3$ directed triangles.
\end{theorem}

Theorem~\ref{thm:fs08} can be restated as a quantitative Arrow-like
statement for dice.
\begin{corollary}
  \label{lem:quantitative-arrow-dice}
  There exists $c > 0$ such that if a comparison function $f$
  on $\cD_{m,n}$ with $m,n > 1$
  is $\eps$-far from transitive, then the probability that a random
  triple of dice is intransitive is at least $c\eps^2$.
\end{corollary}

Since~\cite{FS08} gives an example which is tight up to a constant factor,
Corollary~\ref{lem:quantitative-arrow-dice} is similarly tight. However,
the obtained comparison function does not seem to correspond to any natural
method of comparing dice.

\section*{Acknowledgements} 
We thank Timothy Gowers for helpful discussions of \cite{Pol17},
Kathryn Mann for asking if there is an ``Arrow's theorem'' for dice, and the referee
for a careful reading and helpful comments. 


\newcommand{\etalchar}[1]{$^{#1}$}

\appendix
\section*{Missing calculations in the proof of
  Proposition~\ref{prop:general-anti}}\label{sec:appendix}

In this appendix we include, for ease of verification, several calculations
that were omitted from the proof of Proposition~\ref{prop:general-anti}.

\subsection*{Proof of Lemma~\ref{lem:double-cov}}
In the proof of the calculation lemma (Lemma~\ref{lem:double-cov})
we performed the final calculation in~\eqref{eq:118} only to
establish~\eqref{eq:79c}. Below we give analogous computations
for~\eqref{eq:79}--\eqref{eq:79b} and~\eqref{eq:78c}--\eqref{eq:78b}.

Each of the three calculations establishing~\eqref{eq:79}--\eqref{eq:79b}
proceeds in the same three steps: first, one of the
conclusions~\eqref{eq:95}--\eqref{eq:95b} of Corollary~\ref{cor:normalization}
is applied; second, identities given in~\eqref{eq:119}
and~\eqref{eq:120} are substituted for the integrals; third, the terms are
rearranged.

For~\eqref{eq:79} we have, letting
$D=\{(a_1,a_2,b_1,b_2):a_1>b_1\land a_2>b_2\}$,

\begin{align*}
  &\Pr\left[a_1>b_1\land a_2>b_2\mid\cE_0\right]
  \\
  &\qquad=\left(1+\frac{2}{n}\right)
    \iint_D f(a_1,a_2,b_1,b_2)
    \Bigg(1+\frac{2\alpha_2(a_1+b_1)}{n}
  \\
  &\qquad \qquad\qquad\qquad\qquad\qquad\qquad
    -\frac{a_1^2+b_1^2+a_1a_2+b_1b_2}{n}\Bigg)
     \,dab + o(n^{-1})
  \\
  &\qquad=\left(1+\frac{2}{n}\right)
    \left(\frac{1}{4}+\frac{\alpha_2 A}{n}-\frac{\alpha_2 A}{n}-\frac{B}{2n}
    -\frac{1-B}{2n}-\frac{A^2}{n}-\frac{A^2}{n}
    \right)+o(n^{-1})
  \\
  &\qquad=\frac{1}{4}-\frac{2A^2}{n}
    + o(n^{-1})\;.
\end{align*}
For~\eqref{eq:79a}, letting $D=\{(a_1,b_1):a_1>b_1\}$
\begin{align*}
  \Pr\left[a_1>b_1\mid\cE_a\right]
  &=\left(1+\frac{1}{2n}\right)
    \iint_D f(a_1,b_1)
    \Bigg(1+\frac{\alpha_2a_1}{n}
    -\frac{a_1^2}{2n}\Bigg)
    \,dab + o(n^{-1})
  \\
  &=\left(1+\frac{1}{2n}\right)
    \left(\frac{1}{2}+\frac{\alpha_2 A}{n}
    -\frac{B}{2n}
    \right)+o(n^{-1})
  \\
  &=\frac{1}{2}+\frac{1}{4n}+\frac{\alpha_2 A}{n}-\frac{B}{2n}
    + o(n^{-1})\;.
\end{align*}
For~\eqref{eq:79b}, letting $D=\{(a_1,a_2,b_1,b_2:a_1>b_1\land a_2>b_2\}$,
\begin{align*}
  &\Pr\left[a_1>b_1\land a_2>b_2\mid\cE_a\right]
  \\
  &\qquad=\left(1+\frac{1}{n}\right)
    \iint_D f(a_1,a_2,b_1,b_2)
    \Bigg(1+\frac{2\alpha_2a_1}{n}
    -\frac{a_1^2+a_1a_2}{n}\Bigg)
     \,dab + o(n^{-1})
  \\
  &\qquad=\left(1+\frac{1}{n}\right)
    \left(\frac{1}{4}+\frac{\alpha_2 A}{n}
    -\frac{B}{2n}-\frac{A^2}{n}
    \right)+o(n^{-1})
  \\
  &\qquad=\frac{1}{4}+\frac{1}{4n}+\frac{\alpha_2 A}{n}-\frac{B}{2n}
    -\frac{A^2}{n} + o(n^{-1})\;.
\end{align*}

The calculations showing~\eqref{eq:78c}--\eqref{eq:78b} employ
Lemma~\ref{lem:technical} directly. Each of them applies one
of~\eqref{eq:93b}--\eqref{eq:93c} and uses the fact that
both of those expressions can be approximated as
\begin{align*}
  \iint_D f(a_1,a_2,b_1,b_2)\,dab+o(1)\;.
\end{align*}
More precisely, for~\eqref{eq:78c} we take
$D=\{(a_1,a_2,b_1,b_2):a_1>b_1\land a_1>b_2\}$ and get
\begin{align*}
  \Pr\left[a_1>b_1\land a_1>b_2\mid\cE_0\right]
  &=\frac{\iint_{D} f(a_1,a_2,b_1,b_2)\phi_{n-1}(-a_1)\phi_{n-2}(-b_1-b_2)\,dab}
    {\iint_{\mathbb{R}^4}f(a_1,a_2,b_1,b_2)\phi_{n-1}(-a_1)\phi_{n-2}(-b_1-b_2)\,dab}\\
  &=\frac{\iint_D f(a_1,a_2,b_1,b_2)\,dab+o(1)}
    {\iint_{\mathbb{R}^4} f(a_1,a_2,b_1,b_2)\,dab+o(1)}\\
  &=\Pr[a_1>b_1\land a_1>b_2]+o(1)
    =\frac{1}{3}+o(1)\;.
\end{align*}
Similarly, for~\eqref{eq:78a}, letting
$D=\{(a_1,a_2,b_1,b_2):a_1>b_1\land a_1>b_2\}$,
\begin{align*}
  \Pr\left[a_1>b_1\land a_1>b_2\mid\cE_a\right]
  &=\frac{\iint_{D} f(a_1,a_2,b_1,b_2)\phi_{n-2}(-a_1-a_2)\,dab}
    {\iint_{\mathbb{R}^4}f(a_1,a_2,b_1,b_2)\phi_{n-2}(-a_1-a_2)\,dab}\\
  &=\Pr[a_1>b_1\land a_1>b_2]+o(1)
    =\frac{1}{3}+o(1)\;,
\end{align*}
and, for~\eqref{eq:78b}, letting
$D=\{(a_1,a_2,b_1,c_1):a_1>b_1\land a_1>c_1\}$
\begin{align*}
  \Pr\left[a_1>b_1\land a_1>c_1\mid\cE_a\cap\cE_b\right]
  &=\frac{\iint_{D} f(a_1,a_2,b_1,c_1)\phi_{n-2}(-a_1-a_2)\phi_{n-1}(-b_1)\,dabc}
    {\iint_{\mathbb{R}^4}f(a_1,a_2,b_1,c_1)\phi_{n-2}(-a_1-a_2)\phi_{n-1}(-b_1)\,dabc}\\
  &=\Pr[a_1>b_1\land a_1>c_1]+o(1)
    =\frac{1}{3}+o(1)\;.
\end{align*}

\subsection*{Proof of Lemma~\ref{lem:technical}}

In the proof of the integration lemma (Lemma~\ref{lem:technical}),
in~\eqref{eq:99} we included a detailed calculation only to
establish~\eqref{eq:93}. Below we give crucial steps of similar derivations
for~\eqref{eq:93a}--\eqref{eq:93c}. In each of them: first, we
substitute~\eqref{eq:101} for $\phi_j$; second, we rearrange and absorb
the error terms using~\eqref{eq:121} and~\eqref{eq:99}.

For~\eqref{eq:93a}, we have
\begin{align*}
  &\iint_D f \cdot \phi_{n-1}(-a) \, dab\\
  &\qquad=\iint_D f\cdot
    \Bigg[\exp\left(-\frac{a^2}{2(n-1)}\right)\\
  &\qquad\qquad\qquad\qquad\qquad
    \left(1+\frac{\alpha_1}{n-1}+\frac{\alpha_2 a}{n-1}
    +O\left(\frac{\max(|a|,a^6)}{n^{3/2}}\right)+o(n^{-1})
    \right)\Bigg]\,dab\\
  &\qquad=\iint_D f\cdot
    \left[1+\frac{\alpha_1}{n}+\frac{\alpha_2a}{n}
    -\frac{a^2}{2n}\right]\,dab+o(n^{-1})\;,
\end{align*}
for~\eqref{eq:93b},
\begin{align*}
  &\iint_D f \cdot \phi_{n-2}(-a) \, dab\\
  &\qquad=\iint_D f\cdot
    \Bigg[\exp\left(-\frac{a^2}{2(n-2)}\right)\\
  &\qquad\qquad\qquad\qquad\qquad
    \left(1+\frac{\alpha_1}{n-2}+\frac{\alpha_2 a}{n-2}
    +O\left(\frac{\max(|a|,a^6)}{n^{3/2}}\right)+o(n^{-1})
    \right)\Bigg]\,dab\\
  &\qquad=\iint_D f\cdot
    \left[1+\frac{\alpha_1}{n}+\frac{\alpha_2a}{n}
    -\frac{a^2}{2n}\right]\,dab+o(n^{-1})\;,
\end{align*}
and for~\eqref{eq:93c},
\begin{align*}
  &\iint_D f \cdot \phi_{n-2}(-a)\phi_{n-1}(-b) \, dabc\\
  &\qquad=\iint_D f\cdot
    \left[\exp\left(-\frac{a^2}{2(n-2)}\right)
    \left(1+\frac{\alpha_1}{n-2}+\frac{\alpha_2 a}{n-2}
    +O\left(\frac{\max(|a|,a^6)}{n^{3/2}}\right)+o(n^{-1})
    \right)\right]\\
  &\qquad\qquad\qquad\cdot
    \Bigg[\exp\left(-\frac{b^2}{2(n-1)}\right)\\
  &\qquad\qquad\qquad\qquad\qquad\qquad
    \left(1+\frac{\alpha_1}{n-1}+\frac{\alpha_2 b}{n-1}
    +O\left(\frac{\max(|b|,b^6)}{n^{3/2}}\right)+o(n^{-1})
    \right)\Bigg]\,dabc
    \\
  &\qquad=\iint_D f\cdot
    \left[1+\frac{2\alpha_1}{n}+\frac{\alpha_2(a+b)}{n}
    -\frac{a^2+b^2}{2n}\right]\,dabc+o(n^{-1})\;.
\end{align*}

\end{document}